\documentclass[11pt]{amsart}  	

\usepackage[margin=1in]{geometry}  

\usepackage[all]{xy}						
\usepackage{amssymb}
\usepackage{amscd,latexsym,amsthm,amsfonts,amssymb,amsmath,amsxtra}
\usepackage[mathscr]{eucal}
\usepackage[colorlinks]{hyperref}
\usepackage{mathtools}

\usepackage{chngcntr}
\numberwithin{equation}{section}
\newcounter{keepeqno}

\hypersetup{linkcolor=black, citecolor=black}

\newcommand{\BA}{{\mathbb {A}}}

\newcommand{\BC}{{\mathbb {C}}}

\newcommand{\BL}{{\mathbb {L}}}

\newcommand{\BN}{{\mathbb {N}}}

\newcommand{\BQ}{{\mathbb {Q}}}
\newcommand{\BR}{{\mathbb {R}}}

\newcommand{\BW}{{\mathbb {W}}}

\newcommand{\BZ}{{\mathbb {Z}}}

\newcommand{\Bl}{{\mathbf {l}}}

\newcommand{\Bt}{{\mathbf {t}}}

\newcommand{\CA}{{\mathcal {A}}}

\newcommand{\CC}{{\mathcal {C}}}

\newcommand{\CF}{{\mathcal {F}}}

\newcommand{\CH}{{\mathcal {H}}}

\newcommand{\CK}{{\mathcal {K}}}

\newcommand{\CN}{{\mathcal {N}}}
\newcommand{\CO}{{\mathcal {O}}}

\newcommand{\CS}{{\mathcal {S}}}

\newcommand{\CW}{{\mathcal {W}}}

\newcommand{\CZ}{{\mathcal {Z}}}

\newcommand{\FB}{{\mathfrak {B}}}

\newcommand{\FN}{{\mathfrak {N}}}

\newcommand{\Fa}{{\mathfrak {a}}}
\newcommand{\Fb}{{\mathfrak {b}}}
\newcommand{\Fc}{{\mathfrak {c}}}

\newcommand{\Fj}{{\mathfrak {j}}}

\newcommand{\Fl}{{\mathfrak {l}}}
\newcommand{\Fm}{{\mathfrak {m}}}
\newcommand{\Fn}{{\mathfrak {n}}}
\newcommand{\Fo}{{\mathfrak {o}}}
\newcommand{\Fp}{{\mathfrak {p}}}

\newcommand{\RG}{{\mathrm {G}}}

\newcommand{\RI}{{\mathrm {I}}}
\newcommand{\RJ}{{\mathrm {J}}}

\newcommand{\RM}{{\mathrm {M}}}

\newcommand{\RO}{{\mathrm {O}}}

\newcommand{\RU}{{\mathrm {U}}}

\newcommand{\cusp}{{\mathrm{cusp}}}

\newcommand{\GJ}{{\mathrm{GJ}}}
\newcommand{\GL}{{\mathrm{GL}}}

\newcommand{\Hom}{{\mathrm{Hom}}}

\renewcommand{\Im}{{\mathrm{Im}}}

\newcommand{\Id}{{\mathrm{Id}}}

\newcommand{\Kl}{{\mathrm{Kl}}}

\renewcommand{\Re}{{\mathrm{Re}}}
\newcommand{\reg}{{\mathrm{reg}}}
\newcommand{\Res}{{\mathrm{Res}}}

\newcommand{\SL}{{\mathrm{SL}}}

\newcommand{\sgn}{{\mathrm{sgn}}}

\newcommand{\Span}{{\mathrm{Span}}}

\newcommand{\tr}{{\mathrm{tr}}}

\newcommand{\ud}{\,\mathrm{d}}

\newcommand{\wt}{\widetilde}

\newcommand{\bs}{\backslash}

\def\diag{{\rm diag}}

\def\std{\rm std}

\def\vphi{\varphi}

\newtheorem{thm}{Theorem}[section]
\newtheorem{dfn}[thm]{Definition}
\newtheorem{rmk}[thm]{Remark}

\newtheorem{prp}[thm]{Proposition}
\newtheorem{lem}[thm]{Lemma}
\newtheorem{cor}[thm]{Corollary}

\makeatletter

\newcommand{\Rmnum}[1]{\expandafter\@slowromancap\romannumeral #1@}
\makeatother

\begin{document}
\title[Voronoi Formula and Godement-Jacquet Kernels]{The Voronoi Summation Formula for $\GL_n$ and the Godement-Jacquet Kernels}

\author{Dihua Jiang and Zhaolin Li}	
\address{School of Mathematics, University of Minnesota, 206 Church St. S.E., Minneapolis, MN 55455, USA.}
\email{dhjiang@math.umn.edu}
\email{li001870@umn.edu}



\thanks{The research of this paper is supported in part by the NSF Grant DMS-2200890.}

\date{\today}

\begin{abstract}
Let $\BA$ be the ring of adeles of a number field $k$ and $\pi$ be an irreducible cuspidal automorphic representation of $\GL_n(\BA)$. 
In \cite{JL22, JL23}, the authors introduced $\pi$-Schwartz space $\CS_\pi(\BA^\times)$ and $\pi$-Fourier transform $\CF_{\pi,\psi}$ with a non-trivial additive character $\psi$ of $k\bs\BA$, proved the associated Poisson summation formula over $\BA^\times$, based on the Godement-Jacquet theory for the standard $L$-functions 
$L(s,\pi)$, and provided interesting applications. In this paper, in addition to the further development of the local theory, we found two 
global applications. First, we find a Poisson summation formula proof of the Voronoi summation formula for $\GL_n$ over a number field, which was first proved by A. Ichino and N. Templier (\cite[Theorem 1]{IT13}). Then we introduce the notion of the Godement-Jacquet kernels $H_{\pi,s}$ and their dual kernels $K_{\pi,s}$ for any irreducible cuspidal automorphic representation $\pi$ of $\GL_n(\BA)$ and show in Theorems \ref{thm:H=FK} and \ref{thm:CTh-pi} that $H_{\pi,s}$ and $K_{\pi,1-s}$ are related by the nonlinear $\pi_\infty$-Fourier transform if and only if $s\in\BC$ is a zero of $L_f(s,\pi_f)=0$, the finite part of the standard automorphic $L$-function $L(s,\pi)$, which are the $(\GL_n,\pi)$-versions of \cite[Theorem 1.1]{Clo22}, where the Tate kernel with $n=1$ and $\pi$ the trivial character are considered. 
\end{abstract}

\maketitle
\tableofcontents


\section{Introduction}\label{sec-I}


\subsection{$\pi$-Poisson summation formula}
Let $k$ be a number field and $\BA=\BA_k$ the associated ring of adeles. For any irreducible cuspidal automorphic representation $\pi$ of 
$\GL_n(\BA)$ for any integer $n\geq 1$, the Godement-Jacquet theory (\cite{GJ72}) establishes the analytic continuation and functional equation for the standard $L$-functions $L(s,\pi)$. The key input from the harmonic analysis to the Godement-Jacquet theory is the 
classical Fourier analysis on the affine space $\RM_n$, the space of $n\times n$-matrices, and the associated Poisson summation formula. 
From the point of view in the Braverman-Kazhdan-Ng\^o program (\cite{BK00} and \cite{Ngo20}), this classical theory of Fourier analysis should be reformulated on the 
group $\GL_n(\BA)$. In such a reformulation, the classical (additive) Fourier transform is converted to a convolution integral with a kernel function and the Poisson summation formula is converted to a theta inversion formula, which is a generalization of the classical theta inversion formula. We refer to \cite[Section 2]{JL23} for details. 

In \cite{JL22, JL23}, the Godement-Jacquet theory has been reformulated as harmonic analysis on $\GL_1$ for $L(s,\pi)$ as a vast generalization of the 1950 thesis of J. Tate (\cite{Tat67}). More precisely, for any irreducible cuspidal automorphic representation $\pi$ of $\GL_n(\BA)$, the space of $\pi$-Schwartz functions on $\GL_1(\BA)=\BA^\times$ is defined, which is denoted by $\CS_\pi(\BA^\times)$, and the $\pi$-Foruier transform 
(or operator) $\CF_{\pi,\psi}$ is defined for any given non-trivial additive character $\psi$ of $k\bs\BA$, which takes the $\pi$-Schwartz 
space $\CS_\pi(\BA^\times)$ to the $\wt{\pi}$-Schwartz space $\CS_{\wt{\pi}}(\BA^\times)$, where $\wt{\pi}$ is the contragredient of $\pi$. We refer to Section \ref{sec-piPSF} for details. 
The $\pi$-Poisson summation formula as proved in \cite[Theorem 4.7]{JL23} (or recalled in Theorem \ref{thm:PSF}) takes the following form. 

\begin{thm}[$\pi$-Poisson summation formula]\label{PSF0}
With the notation introduced above, the $\pi$-theta function 
$\Theta_\pi(x,\phi):=\sum_{\alpha\in k^\times}\phi(\alpha x)$
converges absolutely and locally uniformly for any $x\in\BA^\times$ and any $\phi\in\CS_\pi(\BA^\times)$. Moreover, the following identity 
\begin{align}\label{PSF-1}
    \Theta_\pi(x,\phi)
=
\Theta_{\wt{\pi}}(x^{-1},\CF_{\pi,\psi}(\phi))
\end{align}
holds as functions in $x\in\BA^\times$. 
\end{thm}

This $\GL_1$-reformulation in \cite{JL22, JL23} of the Godement-Jacquet theory has found following nice applications, among others: 
\begin{enumerate}
    \item The local theory of such a $\GL_1$-reformulation as developed in \cite{JL22,JL23} proves that the (nonlinear) Fourier transform that is responsible for the local functional equation is given by a convolution operator with an explicitly defined kernel function $k_{\pi_\nu,\psi_\nu}$ on $\GL_1$, which will be related to a Bessel function 
    in Section \ref{sec-BF} of this paper. As consequences, all the Langlands local gamma factors take the form of those in the theory of 
    I. Gelfand, M. Graev, and I. Piatetski-Shapiro in \cite{GGPS} and of A. Weil in \cite{W95}, where the local gamma factors associated with quasi-characters of $\GL_1$ were considered. 
    \item The global theory of such a $\GL_1$-reformation as developed in \cite{JL23} gives the adelic formulation of A. Connes' theorem (\cite[Theorem III.1]{Cn99}) for $L(s,\pi)$, and the complete version of C. Soul\'e's theorem (\cite[Theorem 2]{Sou01}) that provides a spectral interpretation of the zeros of $L(s,\pi)$. We refer to \cite[Theorem 8.1]{JL23} for details.  
    \item In Section \ref{sec-NPVSF} of this paper, the local and global theory of such a $\GL_1$-reformulation in \cite{JL22,JL23} provides a Poisson summation formula proof of the Voronoi formula for any irreducible cuspidal automorphic representation $\pi$ of $\GL_n$, which was previously proved by S. Miller and W. Schmid in \cite{MS11} and by A. Ichino and N. Templier in \cite{IT13} by using the Rankin-Selberg convolution of H. Jacquet, I. Piatetski-Shapiro and J. Shalika in \cite{JPSS83,CPS04}.
    \item In Section \ref{sec-GJK} of this paper, the local and global theory of such a $\GL_1$-reformulation in \cite{JL22,JL23} defines the Godement-Jacquet kernels for $L(s,\pi)$ and proves Theorems \ref{thm:H=FK} and \ref{thm:CTh-pi}, which are the $(\GL_n,\pi)$-versions of Clozel's theorem (\cite[Theorem 1.1]{Clo22}) for any irreducible cuspidal automorphic representation $\pi$ of $\GL_n$. In \cite{Clo22}, Clozel formulates and proves such a theorem for the Tate kernel 
    associated with the Dedekind zeta function $\zeta_k(s)$ for any number field $k$. 
\end{enumerate}

\subsection{Voronoi summation formula}\label{ssec-VSF}

The classical Voronoi summation formula and its recent extension to the $\GL_n$-version have been one of the most powerful tools in Number Theory and relevant areas in Analysis. 
We refer to an enlightening survey paper by S. Miller and W. Schmid (\cite{MS04}) for a detailed account of the current state of the art of the Voronoi summation formula and its applications to important problems in Number Theory. 

The Voronoi summation formula for $\GL_n$ was first studied by S. Miller and W. Schmid in \cite{MS06} for $n=3$ and in \cite{MS11} for general $n$. They use two approaches. One is based on classical harmonic analysis that has been developed in their earlier paper (\cite{MS04-JFA}), and the other is based on the adelic version of the Rankin-Selberg convolutions for $\GL_n\times\GL_1$, which was developed by H. Jacquet, I. Piatetski-Shapiro and J. Shalika in \cite{JPSS83} and by J. Cogdell and Piatetski-Shapiro in \cite{CPS04} 
(and also by Jacquet in \cite{J09} for the Archimedean local theory). The classical approach to the Voronoi formula for $\GL_n$ has also been discussed in \cite{GL06} and \cite{GL08}. A complete treatment of the adelic approach to the Voronoi formula for $\GL_n$ over a general number field was given by A. Ichino and N. Templier in \cite{IT13}. 
We recall their general Voromoi formula for $\GL_n$ below. 

For each irreducible cuspidal automorphic representation $\pi$ of $\GL_n(\BA)$, the Voronoi summation formula is an identity of two summations. One side of the identity is given by certain data associated with $\pi$ and the other side is given by certain corresponding data associated with $\wt{\pi}$, the contragredient of $\pi$. 
Let $\psi=\otimes_{\nu}\psi_{\nu}$ be a non-trivial additive character on $k\bs \BA$. 
At each local place $\nu$ of $k$, to a smooth compactly supported function $w_{\nu}(x)\in\CC_c^{\infty}(k_{\nu}^{\times})$ is associated a dual function $\widetilde{w}_{\nu}(x)$ such that the following functional equation 
\begin{align}\label{dualfunction}
\int_{k_{\nu}^{\times}}\widetilde{w}_{\nu}(y)\chi_\nu(y)^{-1}|y|_{\nu}^{s-\frac{n-1}{2}}\ud^\times y=\gamma(1-s,\pi_{\nu}\times\chi_\nu,\psi_{\nu})\int_{k_{\nu}^{\times}}w_{\nu}(y)\chi_\nu(y)|y|_{\nu}^{1-s-\frac{n-1}{2}}\ud^\times y
\end{align}
holds for all $s\in\BC$ and all unitary characters $\chi_\nu$ of $k_{\nu}^{\times}$. 
Since any irreducible cuspidal automorphic representation $\pi$ of $\GL_n(\BA)$ is generic, i.e. it has a non-zero Whittaker-Fourier coefficient. If we write $\pi=\otimes_{\nu\in|k|}\pi_\nu$, where $|k|$ denotes the set of all local places of $k$, then at any local place $\nu$, the local component $\pi_{\nu}$ is an irreducible admissible and generic representation of $\GL_n(k_\nu)$. 
Let $\CW(\pi_\nu,\psi_\nu)$ be the local Whittaker model of $\pi_\nu$, and $W_\nu(g)$ be any Whittaker function on $\GL_n(k_\nu)$ that belongs to $\CW(\pi_\nu,\psi_\nu)$ 
(see Section \ref{sec-LHA} for the details). 

Let $S$ be a finite set of $|k|$ including all Archimedean places and the local places $\nu$ where $\pi_\nu$ or $\psi_\nu$ is ramified. 
As usual, we write $\BA=\BA_S\times\BA^S$, 
where $\BA_S=\prod_{\nu\in S}k_\nu$, which is naturally embedded as a subring of $\BA$, and $\BA^S$ the subring of adeles $\BA$ with trivial component above $S$. 
At $\nu\notin S$, we take the unramified Whittaker vector $^\circ W_\nu$ of $\pi_\nu$, which is so normalized that 
$^\circ W_\nu(\RI_n)=1$. Denote by $^\circ W^S:=\prod_{\nu\notin S}{^\circ W_{\nu}}$, which is the normalized unramified Whittaker function of 
$\pi^S=\otimes_{\nu\notin S}\pi_{\nu}$. Similarly, we define $^\circ\widetilde{W}^S=\prod_{\nu\notin S}{^\circ\widetilde{W}}_{\nu}$ to be the (normalized) unramified Whittaker function of $\widetilde{\pi}^S=\otimes_{\nu\notin S}\widetilde{\pi}_{\nu}$. We recall that the functions ${^\circ W^S}$ and ${^\circ\widetilde{W}^S}$ are related by the following 
\[
{^\circ\widetilde{W}^S(g)}={^\circ W^S}(w_n{^tg^{-1}})
\]
for all $g\in\RG_n(\BA^S)$, where $w_n$ is the longest Weyl element of $\RG_n=\GL_n$ as defined in \eqref{wlong}.
The following is the Voronoi formula proved in \cite[Theorem 1]{IT13}. The unexplained notation will be defined in Sections \ref{sec-LHA} and \ref{sec-NPVSF}.

\begin{thm}[Voronoi Summation Formula]\label{thm:VSF}
For $\zeta\in\BA^S$, let $R=R_\zeta$ be the set of places $\nu$ such that $|\zeta_{\nu}|>1$. At each $\nu\in S$ let $w_{\nu}\in\CC_{c}^{\infty}(k_{\nu}^{\times})$. Then:
\[
\sum_{\alpha \in k^{\times}}\psi(\alpha \zeta)\cdot{^\circ W^S}\left( \begin{pmatrix} \alpha & \\ & \RI_{n-1} \end{pmatrix}   \right)w_{S}(\alpha)=\sum_{\alpha\in k^{\times}} \Kl_R(\alpha,\zeta,{^\circ\widetilde{W}_R})\cdot{^\circ\widetilde{W}^{R\cup S}}\left( \begin{pmatrix} \alpha & \\ & \RI_{n-1} \end{pmatrix}  \right)\widetilde{w}_{S}(\alpha),
\]
where $w_{S}(\alpha):=\prod_{\nu\in S}w_\nu(\alpha)$ and the same for $\widetilde{w}_{S}(\alpha)$, and $\Kl_R(\alpha,\zeta,{^\circ\widetilde{W}_R})$ is 
a finite Euler product of the local Kloosterman integrals:
\[
\Kl_{R}(\alpha,\zeta,{^\circ{\widetilde{W}_R}}):=\prod_{\nu\in R}\Kl_{\nu}(\alpha,\zeta_{\nu},{^\circ\widetilde{W}_\nu}).
\]
For the place $\nu$, the local Kloosterman integral $\Kl_{\nu}(\alpha,\zeta_{\nu},{^\circ\widetilde{W}_\nu})$ is defined by 
\[
\Kl_{\nu}(\alpha,\zeta_{\nu},{^\circ\widetilde{W}_\nu})
:=|\zeta_{\nu}|_{\nu}^{n-2}\int_{U^-_{\tau}(F_{\nu})} \overline{\psi_{\nu}(u_{n-2,n-1})}{^\circ\widetilde{W}_\nu}(\tau u)\ud u,
\]
where 
\[
U^-_{\tau}=\left\{ \begin{pmatrix}
    \RI_{n-2} & * & 0\\ 0& 1 & 0\\ 0&  0& 1
\end{pmatrix}     \right\}
\quad {\rm and}\quad 
\tau=\begin{pmatrix}
     0& 1 & 0\\ \RI_{n-2} & 0& 0\\ 0&  0& 1
\end{pmatrix}\begin{pmatrix}
    \RI_{n-2} &0 & 0\\0 &-\alpha\zeta_{\nu}^{-1} & 0\\0 & 0& -\zeta_{\nu} 
\end{pmatrix}, 
\]
as given in \cite[Section 2.6]{IT13}.
\end{thm}

The proof of Theorem \ref{thm:VSF} in \cite{IT13} is based on the local and global theory of the Rankin-Selberg convolution for $\GL_n\times\GL_1$ (\cite{JPSS83, CPS04, J09}). 
It is important to mention that Theorem \ref{thm:VSF} and its proof has be extended by A. Corbett to cover an even more general situation with more applications in Number Theory (\cite[Theorem 3.4]{Cor21}).

From the historical development of the Voronoi summation formula, one expects that there should be a proof of the Voronoi formula via a certain kind of Poisson summation formula. 
In other words, the two sides of the Voronoi formula should be related by a certain kind of Fourier transform and the identity should be deduced from the corresponding Poisson 
summation formula. 
In the current proof of Theorem \ref{thm:VSF}, such important ingredients from the harmonic analysis were missing, although 
there were discussions in \cite{IT13} and \cite{Cor21} on the local Bessel transform with the kernels deduced from the local functional equation in the local theory of 
the Rankin-Selberg convolution in \cite{JPSS83} and \cite{J09}, and the identity was deduced from explicit computations from the global zeta integrals of the Rankin-Selberg convolution (\cite{JPSS83} and \cite{CPS04}). Over the Archimedean local fields, Z. Qi has developed in \cite{Qi20} a theory of fundamental Bessel functions of high rank and 
formulated those Bessel transforms in the framework of general Hankel transforms that are integral transforms with Bessel functions as the kernel functions. 

The first global result of this paper is to show that Theorem \ref{thm:VSF} is a special case of Theorem \ref{PSF0}. From our proof, 
it will be clear that any variant of Theorem \ref{thm:VSF} (see \cite[Theorem 3.4]{Cor21} for instance) is also a special case of Theorem \ref{PSF0}.  
Note that the $\pi$-Poisson summation formula on $\GL_1$ in Theorem \ref{PSF0}) relies heavily on the work of R. Godement and H. Jacquet (\cite{GJ72}). Hence our proof of Theorem \ref{thm:VSF} is in principle based on the local and global theory of the Godement-Jacquet integrals for the standard $L$-functions of $\GL_n\times\GL_1$. 

In order to carry out such a proof, we have to understand the nature of the functions occurring on the both side of the Voromoi formula 
in Theorem \ref{thm:VSF}, locally and globally, and show that they are $\pi$-Schwartz function on $\BA^\times$ and are related by the 
$\pi$-Fourier transform $\CF_{\pi,\psi}$ in the sense of \cite{JL22,JL23}. More precisely, for any irreducible smooth representation $\pi_\nu$ of $\GL_n(k_\nu)$, which is of Casselman-Wallach type if $\nu$ is an infinite local place of $k$, 
we define the $\pi_\nu$-Bessel function $\Fb_{\pi_\nu,\psi_\nu}(x)$ on $k_\nu^\times$ (Definitions \ref{piFb-p}, \ref{piFb-c} and \ref{piFb-r}) and obtain a series of 
results on the relations between the $\pi_\nu$-Bessel functions $\Fb_{\pi_\nu,\psi_\nu}(x)$, the $\pi_\nu$-Fourier transforms $\CF_{\pi_\nu,\psi_\nu}$ and the $\pi_\nu$-kernel functions 
$k_{\pi_\nu,\psi_\nu}(x)$ as introduced and studied in \cite{JL22, JL23} (see \eqref{eq:1-FO} and \eqref{kernel-ar} for details), and on new formulas for the dual functions $\wt{w_\nu}(x)$ of $w_\nu(x)\in\CC^\infty_c(k_\nu^\times)$. After all the local preparation, we deduce 
the Voromoi formula in Theorem \ref{thm:VSF} from the $\pi$-Poisson summation formula in Theorem \ref{PSF0}. 
We summarize those local results as the following theorem. 

\begin{thm}\label{thm:local}
    For any local place $\nu$ of the number field $k$, let $\pi_\nu$ be an irreducible smooth representation $\pi_\nu$ of $\GL_n(k_\nu)$, which is of Casselman-Wallach type if $\nu$ is infinite. For any $w_\nu(x)\in\CC^\infty_c(k_\nu^\times)$, $\wt{w_\nu}(x)$ is the dual function of $w_\nu(x)$ as \eqref{dualfunction} or in Theorem \ref{thm:VSF}. Then the following hold. 
    \begin{itemize}
        \item[(1)] The $\pi_\nu$-Fourier transform $\CF_{\pi_\nu,\psi_\nu}$ realizes the duality between $w_\nu(x)$ and $\wt{w_\nu}(x)$, up to normalization, 
         \[
    \CF_{\pi_\nu,\psi_\nu}(w_\nu(\cdot)|\cdot|_\nu^{1-\frac{n}{2}})(x)=\widetilde{w_\nu}(x)|x|_\nu^{1-\frac{n}{2}},\quad 
    \forall x\in k_\nu^\times.
    \]
    \item[(2)] The dual function $\wt{w}(x)$ of $w_\nu(x)$ enjoys the following formula:
    \[
    \wt{w_\nu}(x)=|x|_\nu^{\frac{n}{2}-1}\left(k_{\pi_\nu,\psi_\nu}(\cdot)*(w_\nu^\vee(\cdot)|\cdot|_\nu^{\frac{n}{2}-1})\right)(x),\quad
    \forall x\in k_\nu^\times,
    \]
    where $k_{\pi_\nu,\psi_\nu}(x)$ is the $\pi_\nu$-kernel function of $\pi_\nu$ as in \eqref{kernel-ar} and $w_\nu^\vee(x)=w_\nu(x^{-1})$.
    \item[(3)] As distributions on $k_\nu^\times$, the $\pi_\nu$-kernel function $k_{\pi_\nu,\psi_\nu}$ as in \eqref{kernel-ar} and the $\pi_\nu$-Bessel function are related by the following identity:
\[
k_{\pi_\nu,\psi_\nu}(x)=\Fb_{\pi_\nu,\psi_\nu}(x)|x|_\nu^{\frac{1}{2}},\quad \forall x\in k_\nu^{\times}.
\]
    \item[(4)] The dual function $\wt{w_\nu}(x)$ of $w_\nu(x)$ enjoys the following formula:
    \[
    \wt{w_\nu}(x)=|x|_\nu^{\frac{n-1}{2}}\left(\Fb_{\pi_\nu,\psi_\nu}(\cdot)*(w_\nu^\vee(\cdot)|\cdot|_\nu^{\frac{n-3}{2}})\right)(x),\quad
    \forall x\in k_\nu^\times.
    \]
    \end{itemize}
\end{thm}

The proof of Theorem \ref{thm:local} is given in Sections \ref{sec-LHA} and \ref{sec-BF}. More precisely, Part (1) of  Theorem \ref{thm:local} is Proposition \ref{FofTest}. 
Part (2) is Corollary \ref{dfw-k}. Part (3) is a combination of Propositions \ref{k=j-p}, \ref{k=j-c}, and \ref{k=j-r}. Part (4) is an easy consequence of Parts (2) and (3), which is Corollary~\ref{dfw-b}.

It is important to point out that the $\pi_\nu$-Bessel functions $\Fb_{\pi_\nu,\psi_\nu}(x)$ in the $p$-adic case is defined by means of the Whittaker model of $\pi_\nu$ following the general framework of 
E. Baruch in \cite{BE05}. Hence we have to assume in the $p$-adic case that $\pi_\nu$ is generic in the definition of the $\pi_\nu$-Bessel functions $\Fb_{\pi_\nu,\psi_\nu}(x)$. However, in the real or 
complex case, we follow the general theory of Z. Qi in \cite{Qi20} on Bessel functions of high rank, which works for general irreducible smooth representations of $\GL_n$ of Casselman-Wallach type. Hence 
in the real or complex case, the definition of the $\pi_\nu$-Bessel functions $\Fb_{\pi_\nu,\psi_\nu}(x)$ does not require that $\pi_\nu$ is generic. Since the $\pi_\nu$-kernel function $k_{\pi_\nu,\psi_\nu}$ as in \eqref{kernel-ar} is defined based on the Godement-Jacquet theory, which does not require that $\pi_\nu$ is generic, the uniform result in Part (3) of Theorem \ref{thm:local} suggests that 
one may define the $\pi_\nu$-kernel function $k_{\pi_\nu,\psi_\nu}$ to be the $\pi_\nu$-Bessel functions $\Fb_{\pi_\nu,\psi_\nu}(x)$ in the $p$-adic case when $\pi_\nu$ is not generic. Finally, let us mention that the $\pi_\nu$-Bessel functions $\Fb_{\pi_\nu,\psi_\nu}(x)$ in the real case as given in Definition \ref{piFb-r} is more general than the one defined in \cite{Qi20}, and refer to Remark \ref{Qi-rk} for details. 

\subsection{Godement-Jacquet kernels and Fourier transform}\label{ssec-GJKFT}
Write $|k|=|k|_\infty\cup|k|_f$, where $|k|_\infty$ denotes the subset of $|k|$ consisting of all Archimedean local places of $k$, and $|k|_f$ denotes the subset of $|k|$ consisting 
of all finite local places of $k$. 
Write $\BA_\infty=\prod_{\nu\in|k|_\infty}k_\nu$. For $x=(x_\nu)\in \BA_\infty$, set 
$|x|_\infty:=\prod_{\nu\in|k|_\infty}|x_\nu|_\nu$. Let $\CO=\CO_k$ be the ring of algebraic integers of $k$. L. Clozel defines in \cite{Clo22} the {\bf Tate kernel}:
\begin{align}\label{hsx}
    H_s(x):=X^{s-1}\sum_{\FN(\Fa)\leq X}\FN(\Fa)^{-s}-\frac{\kappa}{1-s}
\end{align}
for $x\in \BA_\infty^\times$ with $X=|x|_\infty$, where $\Fa$ runs over nonzero integral ideals $\Fa\subset \CO$, and $\kappa=\Res_{s=1}\zeta_k(s)$. Here $\zeta_k(s)=\sum_{\Fa\subset\CO}\FN(\Fa)^{-s}$ is the Dedekind zeta function of $k$ with 
$\FN(\Fa):=|N_{k/\BQ}\Fa|$, the absolute norm of $\Fa$; and the dual kernel 
\begin{align}\label{ksx}
    K_s(x):=D^{-\frac{1}{2}}
    X^{s-1}\sum_{\Fa\subset\mathcal{D}^{-1}, \FN(\Fa)\leq X}\FN(\Fa)^{-s}-
    \frac{\kappa\cdot D^{\frac{1}{2}}}{1-s},
\end{align}
where $\mathcal{D}=\mathcal{D}_k$ is the difference of $\CO$ and $\mathcal{D}^{-1}$ is the inverse difference; and $D=\FN(\mathcal{D})$ is the absolute value of the discriminant.
Theorem 1.1 of \cite{Clo22} expresses the relation between those tempered distributions $H_s(x)$ and $K_s(x)$ on $\BA_\infty^\times$ in terms of the condition: $\zeta_k(s)=0$ with $\sigma=\Re(s)\in(0,1)$, 
which is more precisely stated as follows. 

\begin{thm}[Clozel]\label{thm:CTh11}
    Assume that $\sigma=\Re(s)\in(0,1)$. Then $\zeta_k(s)=0$ if and only if 
    \[
    \CF_\infty(H_s)=-K_{1-s}
    \]
    where $\CF_\infty$ is the usual Fourier transform over $\BA_\infty$ with a suitable normalized measure. 
\end{thm}

The second global result of this paper is to define the kernel functions for any irreducible cuspidal automorphic representation $\pi$ of $\GL_n(\BA)$, which will be called the Godement-Jacquet kernels, and prove the analogy of Theorem \ref{thm:CTh11} for the Godement-Jacquet kernels and the standard $L$-functions $L(s,\pi)$ (see Theorems \ref{thm:H=FK} and \ref{thm:CTh-pi} for the exact statements). 

Let $\pi$ be an irreducible cuspidal automorphic representation of $\GL_n(\BA)$ and write 
$\pi=\pi_\infty\otimes\pi_f$, where $\pi_f:=\otimes_{p<\infty}\pi_p$, and write the standard $L$-function of $\pi$ as 
\begin{align}\label{L-1}
    L(s,\pi)=L_\infty(s,\pi_\infty)\cdot L_f(s,\pi_f)
\end{align}
for $\Re(s)$ sufficiently positive. 
As usual, $L(s,\pi)$ is called the complete $L$-function associated with $\pi$, and $L_f(s,\pi_f)$ is called the finite part of the $L$-function associated with $\pi$. 
The local and global theory of R. Godement and H. Jacquet in \cite{GJ72} introduces the global zeta integrals for $L(s,\pi)$ and proves that 
$L(s,\pi)$ has analytic continuation to an entire function in $s\in\BC$ and satisfies the functional equation 
\[
L(s,\pi)=\epsilon(s,\pi)\cdot L(1-s,\wt{\pi}).
\]
Following the reformulation as developed in \cite{JL23}, for an irreducible cuspidal automorphic representation $\pi$ of $\GL_n(\BA)$, there exists a $\pi$-Schwartz space 
$\CS_\pi(\BA^\times)$ as defined in \eqref{piSS-BA}, which defines the $\GL_1$-zeta integral
\begin{align}\label{zetaG10}
    \CZ(s,\phi)=\int_{\BA^\times}\phi(x)|x|_\BA^{s-\frac{1}{2}}\ud^\times x
\end{align}
for any $\phi\in\CS_\pi(\BA^\times)$. 
By \cite[Theorem 4.6]{JL23} the zeta integral $\CZ(s,\phi)$ converges absolutely for $\Re(s)>\frac{n+1}{2}$, admits analytic continuation to an entire function in $s\in\BC$, and satisfies the functional equation 
\begin{align}\label{GFE-GL1}
    \CZ(s,\phi)=\CZ(1-s,\CF_{\pi,\psi}(\phi)).
\end{align}
where $\CF_{\pi,\psi}$ is the $\pi$-Fourier transform as defined in \eqref{FO-BA}. From the global functional equation in \eqref{GFE-GL1}, we introduce the notion of the {\bf Godement-Jacquet kernels} for $L(s,\pi)$ in Definition \ref{dfn:CK}, which can be briefly explained as follows. 

Write $x\in\BA^\times$ as $x=x_\infty\cdot x_f$ with $x_\infty\in \BA_\infty^\times$ and $x_f\in\BA_f^\times$. 
Set 
$\BA^{>1}:=\{x\in\BA^\times\ \colon\ |x|_\BA>1\}$. 
For $x\in\BA^{>1}$, 
we have that $|x|=|x_\infty|_{\BA}\cdot|x_f|_{\BA}>1$ and $|x_f|_{\BA}>|x_\infty|_{\BA}^{-1}$. 
Write 
\[
\CS_\pi(\BA^\times)=\CS_{\pi_\infty}(\BA_\infty^\times)\otimes\CS_{\pi_f}(\BA_f^\times).
\]
For $\phi=\phi_\infty\otimes\phi_f\in\CS_\pi(\BA^\times)$ with $\phi_\infty\in\CS_{\pi_\infty}(\BA_\infty^\times)$ and $\phi_f\in\CS_{\pi_f}(\BA_f^\times)$,  we write
\begin{align}\label{zeta>1-0}
\int_{\BA^{>1}}\phi(x)|x|_{\BA}^{s-\frac{1}{2}}\ud^{\times}x
   =\int_{\BA_\infty^\times} \phi_{\infty}(x_{\infty})|x_{\infty}|_{\BA}^{s-\frac{1}{2}}\ud^{\times}x_{\infty}
   \int_{\BA_f^{\times}}^{>|x_{\infty}|^{-1}}\phi_f(x_f)|x_f|_{\BA}^{s-\frac{1}{2}}\ud^{\times}x_f, 
\end{align}
for any $s\in\BC$, where the inner integral is taken over the domain $\{ x_f\in\BA_f^{\times}\ \colon\  |x_f|_{\BA}> |x_{\infty}|_{\BA}^{-1} \}$. 
Proposition~\ref{zeta>1} shows that the integral converges absolutely for any $s\in\BC$ and is holomorphic in $s\in\BC$. 
By the Fubini theorem and the support in $\BA_f^\times$ of $\phi_f$, which is a fractional ideal of $k$ (Proposition \ref{prp:HKTD}) , the inner integral
$
\int_{\BA_f^{\times}}^{>|x_{\infty}|_{\BA}^{-1}}\phi_f(x_f)|x_f|_{\BA}^{s-\frac{1}{2}}\ud^{\times}x_f
$
converges absolutely for any $s\in\BC$ and any $x_\infty\in\BA_\infty^\times$.  The {\bf Godement-Jacquet kernel} for $L(s,\pi)$ is defined by 
\begin{align}\label{GJK}
  H_{\pi,s}(x_\infty,\phi_f):=
    |x_{\infty}|_{\BA}^{s-\frac{1}{2}}
   \int_{\BA_f^{\times}}^{>|x_{\infty}|_{\BA}^{-1}}\phi_f(x_f)|x_f|_{\BA}^{s-\frac{1}{2}}\ud^\times x_f,  
\end{align}
for $x_\infty\in\BA_\infty^\times$ and for all $s\in\BC$. Clozel defines in \cite{Clo22} the dual kernel for the case of $\GL_1$ with $\pi$ the trivial character. 
We define here the {\bf dual kernel} of the Godement-Jacquet kernel $H_{\pi,s}(x_\infty,\phi_f)$ for $L(s,\pi)$ to be 
\begin{align}\label{GJKd}
    K_{\pi,s}(x_\infty,\phi_f):=
    |x_{\infty}|_{\BA}^{s-\frac{1}{2}}
   \int_{\BA_f^{\times}}^{>|x_{\infty}|_{\BA}^{-1}}\CF_{\pi_f,\psi_f}(\phi_f)(x_f)|x_f|_{\BA}^{s-\frac{1}{2}}\ud^\times x_f,
\end{align}
for $x_\infty\in\BA_\infty^\times$ and for all $s\in\BC$.  Proposition \ref{prp:HKTD} shows that both kernel functions $H_{\pi,s}(x_\infty,\phi_f)$ and $K_{\pi,s}(x_\infty,\phi_f)$ on $\BA_\infty^\times$ can be extended uniquely to tempered distributions on $\BA_\infty$ for any $\phi_f\in\CS_{\pi_f}(\BA_f^\times)$ and for any $s\in\BC$, by using the work of 
S. Miller and W. Schmid in \cite{MS04-JFA}.
We are able to match the kernels $H_{\pi,s}(x_\infty,\phi_f)$ and $K_{\pi,s}(x_\infty,\phi_f)$ with the Euler product expression or Dirichlet series expression of the finite part $L$-function $L_f(s,\pi_f)$ by specifically choosing the $\pi_f$-Schwartz functions $\phi_f\in\CS_{\pi_f}(\BA_f^\times)$, and prove in Section \ref{sec-GJK} the $\pi$-versions of the Clozel theorem (Theorem \ref{thm:CTh11}). Here is an overly simplified version of Theorem \ref{thm:H=FK}, to which we refer the details. 

\begin{thm}\label{H=FK}
For any irreducible cuspidal automorphic representation $\pi$ of $\GL_n(\BA)$, with a choice of $\phi^\star=\phi_\infty\otimes\phi^\star_f\in\CS_\pi(\BA^\times)$, the Godement-Jacquet kernel 
$H_{\pi,s}(x)=H_{\pi,s}(x,\phi_f^\star)$ and its dual kernel $K_{\pi,s}(x)=K_{\pi,s}(x,\phi_f^\star)$ enjoy the following identity:
\begin{align}\label{CK-pi}
H_{\pi,s}(x)
=-
\CF_{\pi_\infty,\psi_\infty}(K_{\pi,1-s})(x)
=
-\int_{\BA_\infty^{\times}}k_{\pi_{\infty},\psi_{\infty}}(x y)K_{\pi,1-s}(y)\ud^{\times}y.
\end{align}
as distributions on $\BA_\infty^\times$ if and only if $s$ is a zero of $L_f(s,\pi_f)$. 
\end{thm}

Theorem \ref{thm:CTh-pi} is a more precise version of Theorem \ref{H=FK} (and Theorem \ref{thm:H=FK}), which is the exact $\pi$-analogy of Theorem \ref{thm:CTh11}.

\subsection{Organiztion of the paper}\label{ssec-OP}

We recall the $\pi$-Poisson summation formula on $\GL_1$ developed by Z. Luo and the first named author of this paper in \cite{JL23} in Section \ref{sec-piPSF}. Sections \ref{sec-piSpace} and \ref{sec-piFour} are to review briefly the local $\pi$-Schwartz spaces and local $\pi$-Fourier operators developed in \cite{JL22} and \cite{JL23}. Based on their work as well as the work of \cite{GJ72}, we recall the formulation of the $\pi$-Poission summation formula on $\GL_1$ in \cite{JL23} in Section \ref{piPSF}.

Sections \ref{sec-LHA} and \ref{sec-BF} are devoted to understand the duality between the function $w_\nu(x)$ and the function $\wt{w_\nu}(x)$ by means of the harmonic analysis on $\GL_1$ as developed in 
\cite{JL22} and \cite{JL23}, and to prove our main local results (Theorem \ref{thm:local}). 
By comparing the Godement-Jacquet theory with the $\GL_n\times\GL_1$ Rankin-Selberg convolution, we are able to express the dual function $\wt{w_\nu}(x)$ of $w_\nu(x)$ in terms of the $\pi_\nu$-Fourier transform $\CF_{\pi_\nu,\psi_\nu}$ up to certain normalization (Proposition \ref{FofTest}), based on Proposition \ref{S=W} that identifies the $\pi_\nu$-Schwartz space $\CS_{\pi_\nu}(k_\nu^\times)$, as introduced in \cite{JL23} and recalled in \eqref{piSS}, with the $\pi_\nu$-Whittaker-Schwartz space $\CW_{\pi_\nu}(k_\nu^\times)$ as defined in \eqref{WSS}. As a consequence, we obtain a formula 
that express the dual function $\wt{w_\nu}(x)$ of $w_\nu(x)$ as a convolution of the $\pi_\nu$-kernel function $k_{\pi_\nu,\psi_\nu}(x)$ with $w_\nu(x)$, up to certain normalization (Corollary \ref{dfw-k}). 
In Section \ref{sec-BF}, we introduce the notion of $\pi_\nu$-Bessel functions $\Fb_{\pi_\nu,\psi_\nu}(x)$ (Definitions \ref{piFb-p}, \ref{piFb-c} and \ref{piFb-r}) and prove the precise relation between the $\pi_\nu$-Bessel functions 
and $\pi_\nu$-kernel functions as defined in \eqref{kernel-ar} (Propositions \ref{k=j-p}, \ref{k=j-c}, and \ref{k=j-r}). In the $p$-adic case, the $\pi_\nu$-Bessel function $\Fb_{\pi_\nu,\psi_\nu}(x)$ on 
$k_\nu^\times$ is introduced following the work of E. Baruch in \cite{BE05}, which is recalled in Section \ref{ssec-NABF}. In the real or complex case, we introduce the $\pi_\nu$-Bessel function $\Fb_{\pi_\nu,\psi_\nu}(x)$ on $k_\nu^\times$ by following the general theory of Bessel functions of high rank by Z. Qi in \cite{Qi20}. It should be mentioned that the $\pi_\nu$-Bessel function $\Fb_{\pi_\nu,\psi_\nu}(x)$ on $\BR^\times$ in the real case is more general than the one considered in \cite{Qi20} (Remark \ref{Qi-rk}). As expected, when $n=2$ our results recover the previous known 
results as discussed by J. Cogdell in \cite{C14} and by D. Soudry in \cite{SD84}. 

With all these ingredients, we are able to give a new proof of the Voronoi formula for $\GL_n$ (Theorem \ref{thm:VSF}) as proved in \cite{IT13} in Section \ref{sec-NPVSF}. In fact, the proof of the Voronoi formula in \cite{IT13} is based on the Rankin-Selberg convolution for $\GL_n\times\GL_1$ in \cite{JPSS83}, \cite{CPS04} and \cite{J09}. And our proof is based on the Godement-Jacquet theory in \cite{GJ72} 
and its reformulation in \cite{JL22, JL23}. The main idea is that the Voronoi summation formula for $\GL_n$ as in Theorem \ref{thm:VSF} is a special case of the $\pi$-Poission summation formula 
on $\GL_1$ as in Theorem \ref{thm:PSF} (\cite[Theorem 4.7]{JL23}), after the long computations carried out in Section \ref{sec-LHA} of this paper and in \cite{JL22, JL23}. Those computations enable us to 
express the summands on the dual side (the right-hand side) of the Voronoi formula in Theorem \ref{thm:VSF} as the global $\pi$-Fourier transform of the summands on the given side (the left-hand side), 
which is Proposition \ref{FTphi}. 

In Section \ref{ssec-GJKD}, in order to define the Godement-Jacquet kernels $H_{\pi,s}(x,\phi_f)$ and their dual kernels $K_{\pi,s}(x,\phi_f)$ (Definition \ref{dfn:CK}) for any irreducible cuspidal automorphic 
representation $\pi$ of $\GL_n(\BA)$, we develop further properties (Propositions \ref{zeta>1} and \ref{prp:intFa}, and Corollary \ref{cor:zeta<1}) of the global zeta integrals $\CZ(s,\phi)$, as defined in \eqref{zetaG10}, by using the $\pi$-Fourier transform $\CF_{\pi,\psi}$ and the associated $\pi$-Poisson summation formula as developed in \cite{JL23}. 
In Proposition \ref{prp:HKTD}, we show that both kernel functions $H_{\pi,s}(x_\infty,\phi_f)$ and $K_{\pi,s}(x_\infty,\phi_f)$ on $\BA_\infty^\times$ can be extended uniquely to tempered distributions on $\BA_\infty$ for any $\phi_f\in\CS_{\pi_f}(\BA_f^\times)$ and for any $s\in\BC$. In Section \ref{ssec-FTGJK}, guided by Theorem \ref{thm:CTh11}, we prove in Proposition \ref{prp:FTGJK} that if $s\in\BC$ is a zero of $L_f(s,\pi_f)$, then the kernel $H_{\pi,s}(x_\infty,\phi_f)$ is equal to the negative of 
$\pi_\infty$-Fourier transform of the dual kernel $K_{\pi,1-s}(x_\infty,\phi_f)$.  For any $\phi_{\infty}\in\CS_{\pi_{\infty}}(\BA_\infty^{\times})$, take $\phi^\star=\phi_\infty\otimes\phi_f^\star$, where $\phi_f^\star:=\otimes_\nu\phi_\nu$ with $\phi_\nu$ as given in Proposition \ref{prp:phip}. Theorem \ref{thm:H=FK} proves 
the $\pi$-version of Theorem \ref{thm:CTh11} for the Euler product expression of $L_f(s,\pi_f)$. With the help of Lemma \ref{lem-testvec}, we obtain the Dirichlet series expression of the kernels in Propositions \ref{prp:CKpi} and \ref{prp:CKDpi}. Finally, Theorem \ref{thm:CTh-pi} establishes the $\pi$-version of Theorem \ref{thm:CTh11} for the Dirichlet series expression of $L_f(s,\pi_f)$.


\section{$\GL_1$-Reformulation of the Godement-Jacquet Theory}\label{sec-piPSF}

We recall from \cite{JL22, JL23} the $\GL_1$-reformulation of the Godement-Jacquet theory for the standard $L$-functions $L(s,\pi)$ 
associated with any irreducible cuspidal automorphic representation $\pi$ of $\GL_n$ (for $n\geq 1$), where $\BA$ is the ring of adeles of a number field $k$. More precisely, we recall 
the $\pi$-Schwartz spaces on $\GL_1$ and the $\pi$-Fourier operators over $\GL_1$ both for the local and global cases, and the $\pi$-Poisson summation formula on $\GL_1$. 

\subsection{$\pi$-Schwartz functions}\label{sec-piSpace}
Let $|k|$ be the set of all local places of $k$. 
For any local place $\nu$, we denote by $F=k_\nu$, the local field of $k$ at $\nu$. 
If $F$ is non-Archimedean, we denote by $\Fo=\Fo_F$ the ring of integers and by $\Fp=\Fp_F$ the maximal ideal of $\Fo$. 
Let $\RG_n=\GL_n$ be the general linear group defined over $F$. Fix the maximal compact subgroups $K$ of $\RG_n(F)$, where $K=\GL_n(\Fo_F)$ if $F$ is non-Archimedean, 
$K=\RO_n$ if $F=\BR$, and $K=\RU_n$ if $F=\BC$. 

Let $\RM_n(F)$ be the space of all $n\times n$ matrices over $F$ and $\CS(\RM_n(F))$ be the space of Schwartz functions on $\RM_n(F)$. When $F$ is Archimedean, it is the space of usual Schwartz functions 
on the affine space $\RM_n(F)$, and when $F$ is $p$-adic, it consists of all locally constant, compactly supported functions on $\RM_n(F)$. 
Let $|\cdot|_F$ be the normalized absolute value on the local field $F$, which is the modular function of the multiplication of $F^\times$ on $F$ with respect to the self-dual additive Haar measure $\ud^+x$ on $F$. As a reformulation of the local Godement-Jacquet theory in \cite[Section 2.2]{JL23}, the (standard) Schwartz space on $\RG_n(F)$ is defined to be 
\begin{align}\label{GJ-SS}
\CS_{\std}(\RG_n(F)):=\{\xi\in\CC^\infty(\RG_n(F)) \colon |\det g|_F^{-\frac{n}{2}}\cdot\xi(g)\in\CS(\RM_n(F))\}, 
\end{align}
where $\CC^\infty(\RG_n(F))$ denotes the space of all smooth functions on $\RG_n(F)$. 
By \cite[Prposition 2.5]{JL23}, the Schwartz space $\CS_{\std}(\RG_n(F))$ is a subspace of $L^2(\RG_n(F),\ud g)$, which is the space of square-integrable functions on $\RG_n(F)$.

Consider the determinant map
$
\det={\det}_F\colon\RM_n(F)=\GL_n(F)\to F.
$
When restricted to $F^\times$, we obtain that 
\begin{align}\label{det}
\det={\det}_F\colon\RG_n(F)=\GL_n(F)\to F^\times
\end{align}
and the fibers of the determinant map $\det$ are of the form:
\begin{align}\label{fiber}
\RG_n(F)_x:=
\{g\in \RG_n(F) \colon \det g=x\}.
\end{align}
When $x=1$, the fiber is the kernel of the map, i.e. $\ker(\det)=\SL_n(F)$. In general, each fiber $\RG_n(F)_x$ is an $\SL_n(F)$-torsor. 
Let $\ud^+g$ be the self-dual Haar measure on $\RM_n(F)$ with respect to the standard Fourier transform defined by (\ref{eq:FTMAT}) below.   On $\RG_n(F)$, we fix 
the Haar measure $\ud g = |\det g|_F^{-n}\cdot\ud^+g$. 
Let $\ud_1 g$ be the induced Haar measure $\ud g$ from $\RG_n(F)$ to $\SL_n(F)$. It follows that the Haar measure $\ud_1 g$ induces an $\SL_n(F)$-invariant measure $\ud_x g$ on each fiber $\RG_n(F)_x$. 

Let $\Pi_F(\RG_n)$ be the set of equivalence classes of irreducible smooth representations of $\RG_n(F)$ when $F$ is non-Archimedean;
and of irreducible Casselman-Wallach representations of $\RG_n(F)$ when $F$ is Archimedean. For $\pi\in\Pi_F(\RG_n)$,  we denote by $\CC(\pi)$ the space of all matrix coefficients of $\pi$. 
Write $\xi=|\det g|_F^{\frac{n}{2}}\cdot f(g)\in\CS_{\std}(\RG_n(F))$ with some $f\in\CS(\RM_n(F))$
as in \eqref{GJ-SS}.  For $\varphi_\pi\in\CC(\pi)$,  as in \cite[Section 3.1]{JL23}, 
we define
\begin{align}\label{fibration}
\phi_{\xi,\vphi_\pi}(x) := \int_{\RG_n(F)_x}
\xi(g)\vphi_\pi(g)\ud_x g
=
|x|_F^{\frac{n}{2}}
\int_{\RG_n(F)_x}
f(g)\vphi_\pi(g)\ud_x g.
\end{align}
By \cite[Proposition 3.2]{JL23}, the function $\phi_{\xi,\vphi_\pi}(x)$ is absolutely convergent for all $x\in F^\times$ and is smooth over $F^\times$. 
As in \cite[Definition 3.3]{JL23}, for any $\pi\in\Pi_F(\RG_n)$, the space of $\pi$-Schwartz functions is defined as
\begin{align}\label{piSS}
\CS_\pi(F^\times) = \Span
\{
\phi=\phi_{\xi,\vphi_\pi}\in\CC^\infty(F^\times) \colon \xi\in \CS_{\std}(\RG_n(F)),\vphi_\pi\in \CC(\pi)\}.
\end{align}
By \cite[Corollary 3.8]{JL23}, we have
\begin{align}\label{CSC}
\CC_c^\infty(F^\times)\subset
\CS_\pi(F^\times)
\subset
\CC^\infty(F^\times).
\end{align}

\subsection{$\pi$-Fourier transform}\label{sec-piFour}
Let $\psi=\psi_F$ be a fixed non-trivial additive character of $F$. 
The (standard) Fourier transform $\CF_\psi$ on $\CS(\RM_n(F))$ is defined as follows,
\begin{equation}\label{eq:FTMAT}
\CF_\psi(f)(x) = \int_{\RM_n(F)}
\psi(\tr(xy))f(y)\ud^+y.
\end{equation}
It is well-known that the Fourier transform $\CF_\psi$ extends to a unitary operator on the space $L^2(\RM_n(F),\ud^+x)$ and satisfies the following identity:
\begin{equation}\label{eq:FTId}
\CF_\psi\circ \CF_{\psi^{-1}}  =\Id.
\end{equation}
Following the reformulation of the local Godement-Jacquet theory in \cite[Section 2.3]{JL23}, the Fourier transform $\CF_\psi$ on $\CS(\RM_n(F))$ yields 
a (nonlinear) Fourier transform $\CF_{\GJ}$ on $\CS_{\std}(\RG_n(F))$, which is a convolution operator with the distribution kernel:
\begin{align}\label{GJ-kernel}
\Phi_{\GJ}(g):=\psi(\tr g)\cdot|\det g|_F^{\frac{n}{2}}. 
\end{align}
More precisely, the Fourier transform $\CF_{\GJ}$ is defined to be 
\begin{align}\label{GJ-FO}
\CF_{\GJ}(\xi)(g):=\left(\Phi_{\GJ}*\xi^\vee\right)(g)
\end{align}
for any $\xi\in\CS_{\std}(\RG_n(F))$, where $\xi^{\vee}(g):=\xi(g^{-1})$. From \cite[Proposition 2.6]{JL23}, a relation between
the (nonlinear) Fourier operator $\CF_{\GJ}$ and the (classical or linear) Fourier transform $\CF_\psi$ is given by 
\begin{align}\label{FOFT}
\CF_{\GJ}(\xi)(g)=\left(\Phi_{\GJ}*\xi^\vee\right)(g)
=
|\det g|_F^{\frac{n}{2}}\cdot\CF_\psi(|\det g|_F^{-\frac{n}{2}}\xi)(g).
\end{align}
From the proof of \cite[Proposition 2.6]{JL23}, it is easy to obtain that
\begin{align}\label{convolution}
\left(\Phi_{\GJ}*\xi^\vee\right)(g)
=
|\det g|_F^{\frac{n}{2}}\left(\psi(\tr(\cdot))*(|\det(\cdot)|_F^{\frac{n}{2}}\xi)^\vee\right)(g)
\end{align}
for any $\xi\in\CS_{\std}(\RG_n(F))$.

As in \cite[Section 3.2]{JL23}, the $\pi$-Fourier transform $\CF_{\pi,\psi}$ is defined through the following diagram:
\begin{align}\label{diag:F}
\xymatrix{
\CS_{\std}(\RG_n(F))\otimes \CC(\pi)\ar[d]\ar[rrr]^{(\CF_{\GJ},(\cdot)^{\vee})}&&& \CS_{\std}(\RG_n(F))\otimes \CC(\wt{\pi})\ar[d]\\
\CS_\pi(F^\times) \ar[rrr]^{\CF_{\pi,\psi}} &&& \CS_{\wt{\pi}}(F^\times)
}
\end{align}
More precisely, for $\phi=\phi_{\xi,\vphi_\pi}\in\CS_\pi(F^\times)$ with a $\xi\in\CS_{\std}(\RG_n(F))$ and
a $\vphi_\pi\in\CC(\pi)$, the $\pi$-Fourier transform $\CF_{\pi,\psi}$ is defined by 
\begin{align}\label{eq:1-FO}
\CF_{\pi,\psi}(\phi)=\CF_{\pi,\psi}(\phi_{\xi,\vphi_\pi}):=\phi_{\CF_{\GJ}(\xi),\vphi_\pi^\vee},
\end{align}
where $\vphi_\pi^\vee(g)=\vphi_\pi(g^{-1})\in\CC(\wt{\pi})$. It was verified in \cite[Proposition 3.9]{JL23} that the descending $\pi$-Fourier transform $\CF_{\pi,\psi}$ is well defined. 
From \cite[Theorem 5.1]{JL22}, the $\pi$-Fourier transform $\CF_{\pi,\psi}$ can also be represented as a convolution operator by some kernel function $k_{\pi,\psi}$, 
which is explicitly given as follows. 

We fix a $\varphi_{\widetilde{\pi}}\in \CC(\widetilde{\pi})$ with $\varphi_{\widetilde{\pi}}(\RI_n)=1$. We also
choose a sequence of test functions $\{\Fc_\ell\}_{\ell=1}^\infty\subset \CC^\infty_c(\RG_n(F))$, such that for any $h\in \CC_c^\infty(\RG_n(F))$,
\begin{align}\label{Fc-I}
\lim_{\ell\to \infty}
\int_{\RG_n(F)}
\Fc_\ell(g)h(g) \ud g= h(\RI_n).
\end{align}
In other words, the sequence $\{\Fc_\ell\}_{\ell=1}^\infty$ tends to the delta mass supported at the identity $\RI_n$ as $\ell\rightarrow\infty$.
The $\pi$-kernel function $k_{\pi,\psi}(x)$ is defined as
\begin{equation}\label{kernel-ar}
k_{\pi,\psi}(x)
:=
\int^\reg_{\det g=x}
\Phi_{\GJ}(g)\vphi_{\wt{\pi}}(g)\ud_x g=
|x|^{\frac{n}{2}}_F\int^\reg_{\det g=x}\psi(\tr(g))
\vphi_{\wt{\pi}}(g)\ud_x g
\end{equation}
where $\Phi_{\GJ}$ is the kernel function as defined in \eqref{GJ-kernel} and the integral is regularized as follows:
\begin{align}\label{k-reg-1}
\int^\reg_{\det g=x}
\Phi_{\GJ}(g)\vphi_{\wt{\pi}}(g)\ud_x g
:=
\lim_{\ell\to \infty}
\int_{\det g=x}
\left(\Phi_{\GJ}*\Fc_\ell^\vee\right)(g)
\vphi_{\wt{\pi}}(g)
\ud_xg.
\end{align}
It is shown in \cite[Proposition 3.5, Corollary 3.7, Corollary 4.5 and Theorem 4.6]{JL22} that $k_{\pi,\psi}$ is a smooth function on $F^{\times}$ and is independent of the choice of the matrix coefficient $\varphi_{\widetilde{\pi}}$ and the chosen sequence $\{ \Fc_{\ell}\}_{\ell=1}^{\infty}$ that tends to the delta mass supported at $\RI_n$.
By \cite[Theorem 5.1]{JL22}, we have that for any $\phi\in\CC_c^{\infty}(F^{\times})$
\begin{align}\label{FO-k}
  \CF_{\pi,\psi}(\phi)(x)=(k_{\pi,\psi}*\phi^\vee)(x).  
\end{align}
Following \cite{Ngo20}, one may call the $\pi$-Fourier transform $\CF_{\pi,\psi}$ a generalized Hankel transform or the $\pi$-Hankel transform.

\subsection{$\pi$-Poisson summation formula on $\GL_1$}\label{piPSF}

Recall that $|k|$ is the set of all local places of the number field $k$. Let $|k|_{\infty}$ be the subset of $|k|$ consisting of all Archimedean local places of $k$. We may write
$|k|=|k|_\infty\cup|k|_f$, 
where $|k|_f$ is the set of non-Archimedean local places of $k$. 
Let $\Pi_\BA(\RG_n)$ be the set of equivalence classes of irreducible admissible representations
of $\RG_n(\BA)$. We write $\pi=\otimes_{\nu\in|k|}\pi_\nu$ and assume that
$\pi_\nu\in\Pi_{k_\nu}(\RG_n)$ and at almost all finite local places $\nu$ the local representations $\pi_\nu$ are unramified. This means that when $\nu\in |k|_f$, $\pi_\nu$ is an irreducible admissible representation of $\RG_n(k_\nu)$, and when $\nu\in|k|_\infty$, $\pi_\nu$ is of Casselman-Wallach type as a representation of $\RG_n(k_\nu)$.
Let $\CA(\RG_n)\subset\Pi_\BA(\RG_n)$ be the subset consisting of equivalence classes of irreducible admissible automorphic representations
of $\RG_n(\BA)$, and $\CA_\cusp(\RG_n)$ be the subset of cuspidal members in $\CA(\RG_n)$. We refer to \cite[Chepter 1]{Arthur13} or \cite{JL23} for the notation and definition of automorphic representations. 

Take any $\displaystyle{\pi=\otimes_{\nu\in|k|}}\pi_{\nu}\in\Pi_{\BA}(\RG_n)$. For each local place $\nu \in |k|$, the $\pi_{\nu}$-Schwartz space $\CS_{\pi_{\nu}}(k_{\nu}^{\times})$ is defined as in \eqref{piSS}. Recall from \cite[Theorem 3.4]{JL23} that the basic function $\BL_{\pi_{\nu}}\in\CS_{\pi_{\nu}}(k_{\nu}^{\times})$ is defined when the local component $\pi_{\nu}$ of $\pi$ is unramified. Then the $\pi$-Schwartz space on $\BA^\times$ is defined to be
\begin{align}\label{piSS-BA}
\CS_\pi(\BA^\times):=\otimes_{\nu\in|k|}\CS_{\pi_\nu}(k_\nu^\times),
\end{align}
which is the restricted tensor product of the local $\pi_\nu$-Schwartz space $\CS_{\pi_\nu}(k_\nu^\times)$ with respect to the family of the basic functions $\BL_{\pi_\nu}$ for all the local places $\nu$ at which
$\pi_\nu$ are unramified. The factorizable vectors $\phi=\otimes_\nu\phi_\nu$ in $\CS_\pi(\BA^\times)$ can be written as
\begin{align}\label{SF-factorize}
\phi(x)=\prod_{\nu\in|k|}\phi_\nu(x_\nu),\;x=(x_{\nu})_{\nu}.
\end{align}
Here at almost all finite local places $\nu$, $\phi_\nu(x_\nu)=\BL_{\pi_\nu}(x_\nu)$. According to the normalization (\cite[Theorem 3.4]{JL23}), we have that $\BL_{\pi_\nu}(x_\nu)=1$ when $x_\nu\in\Fo_\nu^\times$, the unit group of
the ring $\Fo_\nu$ of integers at $\nu$. Hence for any given $x\in\BA^\times$, the product in \eqref{SF-factorize} is a finite product.

For any factorizable vectors $\phi=\otimes_\nu\phi_\nu$ in $\CS_\pi(\BA^\times)$, we define the
$\pi$-Fourier transform (or operator):
\begin{align}\label{FO-BA}
\CF_{\pi,\psi}(\phi):=\otimes_{\nu\in|k|}\CF_{\pi_\nu,\psi_\nu}(\phi_\nu).
\end{align}
Here at each $\nu\in|k|$, $\CF_{\pi_\nu,\psi_\nu}$ is the local $\pi_\nu$-Fourier transform as defined in
\eqref{diag:F} and \eqref{eq:1-FO}, which takes the $\pi_\nu$-Schwartz space
$\CS_{\pi_\nu}(k_\nu^\times)$ to the $\wt{\pi}_\nu$-Schwartz space $\CS_{\wt{\pi}_\nu}(k_\nu^\times)$,
and has the property that 
$\CF_{\pi_\nu,\psi}(\BL_{\pi_\nu})=\BL_{\wt{\pi}_\nu}$, 
when the data are unramified at $\nu$ (see \cite[Theorem 3.10]{JL23}). Hence the Fourier transform $\CF_{\pi,\psi}$ as defined in
\eqref{FO-BA} maps the $\pi$-Schwartz space $\CS_\pi(\BA^\times)$ to the $\wt{\pi}$-Schwartz space $\CS_{\wt{\pi}}(\BA^\times)$, where $\wt{\pi}\in\Pi_\BA(\RG_n)$ is the 
contragredient of $\pi$. The $\pi$-Poisson summation formula (\cite[Theorem 4.7]{JL23}) can be stated as below. 

\begin{thm}[$\pi$-Poisson summation formula]\label{thm:PSF}
For any  $\pi\in\CA_\cusp(\RG_n)$, the $\pi$-theta function
$\Theta_\pi(x,\phi):=\sum_{\alpha\in k^\times}\phi(\alpha x)$ 
converges absolutely and locally uniformly for any $x\in\BA^\times$ and any $\phi\in\CS_\pi(\BA^\times)$. 
Let $\wt{\pi}\in\CA_\cusp(\RG_n)$ be the contragredient of $\pi$. Then the following identity 
\[
\Theta_\pi(x,\phi)
=
\Theta_{\wt{\pi}}(x^{-1},\CF_{\pi,\psi}(\phi)),
\]
holds as functions in $x\in\BA^\times$, where $\CF_{\pi,\psi}$ is the $\pi$-Fourier transform as defined in
\eqref{FO-BA}.
\end{thm}


\section{Local Harmonic Analysis}\label{sec-LHA}


In this section, we take $F=k_\nu$ to be a local field of characteristic zero and fix a non-trivial additive character $\psi=\psi_F$ of $F$. Since the representations 
$\pi\in\Pi_F(\RG_n)$ considered in this section are the local components of irreducible cuspidal automorphic representations of $\RG_n(\BA)$, we may only consider 
generic $\pi\in\Pi_F(\RG_n)$ without loss of generality. 

Let $B_n=T_nN_n$ be the Borel subgroup of $\RG_n$, which consisting of all upper-triangular matrices of $\RG_n$, where $T_n$ is the maximal torus consisting of all 
diagonal matrices of $\RG_n$, and $N_n$ is the unipotent radical of $B_n$, which consists of matrices $n=(n_{i,j})$ with $n_{i,j}=0$ if $1\leq j<i\leq n$, and 
$n_{i,i}=1$ for $i=1,2,\dots,n$. Without loss of generality, we may take a generic character as 
\[
\psi(n)=\psi_{N_n}(n)=\psi_F(n_{1,2}+n_{2,3}+\cdots+n_{n-1,n}). 
\]
Let $\ell_\psi$ be a non-zero member in $\Hom_{N_n(F)}(\pi,\psi)$, which is one-dimensional if $\pi\in\Pi_F(\RG_n)$ is generic. For any $v\in V_\pi$, define the Whittaker function by $W_v(g):=\ell_\psi(\pi(g)v)$. 
Let $\CW(\pi,\psi)$ be the Whittaker model of $\pi$, which consisting of Whittaker functions $W_v(g)$ with $v$ runs through the space $V_\pi$ of $\pi$. Let $V_\pi^\infty$ be 
the subspace of $V_\pi$ consisting of all smooth vectors of $V_\pi$. We define the $\pi$-Whittaker-Schwartz space on $F^\times$ to be 
\begin{align}\label{WSS}
    \CW_{\pi,\psi}(F^\times):=\{\omega(x):=|x|^{1-\frac{n}{2}}\cdot W_v\left( \begin{pmatrix}  x& \\ & \RI_{n-1} \end{pmatrix}\right) \colon v\in V_\pi^\infty\}, 
\end{align}
where $|\cdot|=|\cdot|_F$ is the normalized absolute value on $F$.

\begin{prp}\label{S=W}
For any $\pi\in\Pi_F(\RG_n)$, which is generic, the $\pi$-Schwartz space and the $\pi$-Whittaker-Schwartz space coincide with each other:
$\CS_{\pi}(F^{\times})=\CW_{\pi,\psi}(F^\times)$. 
\end{prp}

\begin{proof}
We first show that
$\CW_{\pi,\psi}(F^\times)\subset \CS_{\pi}(F^{\times})$. 
For any unitary character $\chi$ of $F^{\times}$ and $W\in\CW(\pi,\psi)$, the local Rankin-Selberg integral for $\GL_n\times\GL_1$ 
\begin{align*}
    \Psi(s,W,\chi):=\int_{F^{\times}} W\left( \begin{pmatrix}  x& \\ & \RI_{n-1} \end{pmatrix}  \right)\chi(x)|x|^{s-\frac{n-1}{2}} \ud ^{\times}x
    =\int_{F^{\times}} \omega(x)\chi(x)|x|^{s-\frac{1}{2}} \ud ^{\times}x,
\end{align*}
where $\omega(x)\in \CW_{\pi,\psi}(F^\times)$ as in Proposition \ref{WSpi}, is absolutely convergent when $\Re(s)$ is sufficiently positive and the fractional ideal generated by all such integrals is $\BC[q^{-s},q^s]L(s,\pi\times\chi)$ by \cite[Theorem 2.7]{JPSS83} for the non-Archimedean case and a holomorphic multiple of $L(s,\pi\times\chi)$, bounded at infinity in vertical strips due to \cite[Theorem 2.1]{J09} for Archimedean case.

According to \cite[Theorem 3.4]{JL23}, there is some $\phi\in\CS_{\pi}(F^{\times})$ such that 
\begin{align*}
\Psi(s,W,\chi)=\CZ(s,\phi,\chi):=\int_{F^{\times}}\phi(x)\chi(x)|x|^{s-\frac{1}{2}}\ud^{\times}x
\end{align*}
when $\Re(s)$ is sufficiently positive. In particular, fix a $s_0\in\BR$ sufficiently positive such that both functions $\phi(\cdot)|\cdot|^{s_0-\frac{1}{2}}$ and $\omega(\cdot)|\cdot|^{s_0-\frac{1}{2}}$ belong to $L^1(F^{\times})$, the space of $L^1$-functions on $F^\times$. 
It follows that 
\[
\int_{F^{\times}}\left(\phi(x)|x|^{s_0-\frac{1}{2}}-\omega(x)|x|^{s_0-\frac{1}{2}} \right)\chi(x)\ud^\times x=0
\]
for all unitary character $\chi$ of $F^{\times}$. From the general theory about absolutely continuous measures on local compact abelian groups (See \cite[Theorem 23.11]{HR79} for instance), we must have that 
$\phi(x)|x|^{s_0-\frac{1}{2}}-\omega(x)|x|^{s_0-\frac{1}{2}} =0$ 
for a.e. $x\in F^{\times}$, which implies
$\omega(x)=\phi(x)$ 
for all $x\in F^{\times}$ since both functions are smooth. Hence we obtain that $\CW_{\pi,\psi}(F^\times)\subset \CS_{\pi}(F^{\times})$. 

Again, by \cite[Theorem 3.4]{JL23} and the local theory of the Rankin-Selberg convolution of $\GL_n\times\GL_1$ as in \cite[Theorem 2.7]{JPSS83} for the non-Archimedean case and in 
\cite[Theorem 2.1]{J09} for Archimedean case, we can repeat the above discussion to prove that 
$\CS_{\pi}(F^{\times})\subset\CW_{\pi,\psi}(F^{\times})$. 
Hence we get that $\CS_{\pi}(F^{\times})=\CW_{\pi,\psi}(F^{\times})$.
\end{proof}

From Proposition \ref{S=W}, the following assertion is clear, since the $\pi$-Schwartz space $\CS_{\pi}(F^{\times})$ is independent of the choice of the character $\psi$.

\begin{cor}\label{WSpi}
    The space of $\pi$-Whittaker-Schwartz functions $\CW_{\pi,\psi}(F^\times)$ defined in \eqref{WSS} is independent of the choice of the character $\psi$.
\end{cor}
    
By Corollary \ref{WSpi}, we may denote by $\CW_{\pi}(F^\times)$ the $\pi$-Whittaker-Schwartz space on $F^\times$ as defined in \eqref{WSS}. 
After identifying the $\pi$-Schwartz space $\CS_{\pi}(F^{\times})$ with the $\pi$-Whittaker-Schwartz space $\CW_{\pi}(F^\times)$, we are going to understand 
the $\pi$-Fourier transform 
$
\CF_{\pi,\psi}\colon \CS_\pi(F^\times) \to \CS_{\wt{\pi}}(F^\times)
$
in terms of the structure of Whittaker models. 

\begin{prp}\label{F=W}
    For $\phi\in\CS_{\pi}(F^{\times})$, we may write as in \eqref{WSS} that 
    \begin{align*}
        \phi(x)=\omega(x)=W\left( \begin{pmatrix}  x& \\ & \RI_{n-1} \end{pmatrix}  \right)|x|^{1-\frac{n}{2}}
    \end{align*}
for some $W\in\CW(\pi,\psi)$. Then the $\pi$-Fourier transform can be expressed by the following formula:
\begin{align*}
    \CF_{\pi,\psi}(\phi)(x)=\CF_{\pi,\psi}(\omega)(x)=|x|^{1-\frac{n}{2}}\int_{F^{n-2}}\left(\pi(w_{n,1})\widetilde{W}\right)\begin{pmatrix}
         x& & \\y & \RI_{n-2} & \\ & & 1 
    \end{pmatrix}\ud y,
\end{align*}
where $\widetilde{W}(g):=W(w_0\,^tg^{-1})$ for any $g\in\GL_n(F)$ is a Whiitaker function in $\CW(\wt{\pi},\psi^{-1})$ and 
\begin{align}\label{wn1}
    w_{n,1}=\begin{pmatrix}
        1 & \\ & w_{n-1}
    \end{pmatrix}.
\end{align}
Here we denote by $w_m$ the longest Weyl element of $\RG_m=\GL_m$, which is defined inductively by 
\begin{align}\label{wlong}
    w_m=\begin{pmatrix}&1\\ w_{m-1}&\end{pmatrix},\quad {\rm with}\quad w_2=\begin{pmatrix}&1\\ 1&\end{pmatrix}. 
\end{align}
\end{prp}

\begin{proof}
From the functional equation for the local zeta integrals $\CZ(s,\phi,\chi)$ as proved in \cite[Theorem 3.10]{JL23}, we have that 
\begin{align*}
    \Psi(s,W,\chi)=\CZ(s,\phi,\chi)=\CZ(1-s,\CF_{\pi,\psi}(\phi),\chi^{-1})\gamma(s,\pi\times\chi,\psi)^{-1}.
\end{align*}
On the other hand, from the functional equation for the local zeta integrals $\Psi(s,W,\chi)$ as proved in \cite[Theorem 2.7]{JPSS83} for the non-Archimedean case and in \cite[Theorem 2.1]{J09} for the Archimedean case, we have that 
\begin{align*}
\Psi(s,W,\chi)=\gamma(s,\pi\times\chi,\psi)^{-1}\int_{F^{\times}}\int_{F^{n-2}}\left(\pi(w_{n,1})\widetilde{W}\right)\begin{pmatrix}
         x& & \\y & \RI_{n-2} & \\ & & 1 
    \end{pmatrix}\ud y\;\chi^{-1}(x)|x|^{\frac{3-n}{2}-s}\ud^{\times}x.
\end{align*}
From the absolute convergence of the local zeta integrals $\CZ(s,\phi,\chi)$ and $\Psi(s,W,\chi)$, we may choose and fix a $s_0\in\BC$ with $\Re(s_0)$ sufficiently negative, 
such that both functions 
\begin{align*}
    \CF_{\pi,\psi}(\phi)(\cdot)|\cdot|^{\frac{1}{2}-s_0}\quad {\rm and}\quad 
    |\cdot|^{\frac{3-n}{2}-s_0}\int_{F^{n-2}}\left(\pi(w_{n,1})\widetilde{W}\right)\begin{pmatrix}
         \cdot & & \\y & \RI_{n-2} & \\ & & 1 
    \end{pmatrix}\ud y
\end{align*}
belong to $L^1(F^{\times})$. It follows that 
\begin{align*}
   \int_{F^{\times}} \left( \CF_{\pi,\psi}(\phi)(x)|x|^{\frac{1}{2}-s_0}- \int_{F^{n-2}}\left(\pi(w_{n,1})\widetilde{W}\right)\begin{pmatrix}
         x & & \\y & \RI_{n-2} & \\ & & 1 
    \end{pmatrix}\ud y|x|^{\frac{3-n}{2}-s_0}
    \right) \chi^{-1}(x)\ud^{\times}x=0 
\end{align*}
for any unitary character $\chi$. Now we use the same argument as in the proof of Proposition \ref{S=W} to deduce that
\[
\CF_{\pi,\psi}(\phi)(x)=|x|^{1-\frac{n}{2}}\int_{F^{n-2}}\left(\pi(w_{n,1})\widetilde{W}\right)\begin{pmatrix}
         x & & \\y & \RI_{n-2} & \\ & & 1 
    \end{pmatrix}\ud y
\]
for any $x\in F^{\times}$, as they are smooth in $x$.
\end{proof}

In particular, in the case $n=2$, we have a much simpler formula. 

\begin{cor}\label{GL2F=W}
When $n=2$, the action of the longest Weyl group element $w_2$ of $\RG_2$ on the Kirillov model of $\pi$ is given by the 
(non-linear) Fourier transform $\CF_{\pi,\psi}$: 
\begin{align*}
    \CF_{\pi,\psi}(\phi)=\pi(w_2)(\phi)\cdot\omega_{\pi}^{-1}, \quad {\rm for}\quad \phi\in \CS_{\pi}(F^\times),
\end{align*}
where the $\pi$-Schwartz space $\CS_{\pi}(F^\times)$ and the $\pi$-Whittaker-Schwartz space $\CW_\pi(F^\times)$ can be identified with the Kirillov model of $\pi$ by Proposition \ref{S=W} and $\omega_{\pi}$ is the central character of $\pi$.
\end{cor}

\begin{proof}
    According to Proposition \ref{F=W}, let $\phi(x)=W\left( \begin{pmatrix}
        x & \\ & 1
    \end{pmatrix}  \right)$,
    we have
    \begin{align*}
        \CF_{\pi,\psi}(\phi)(x)=W\left(w_2  \begin{pmatrix}
            x^{-1} & \\ & 1 
        \end{pmatrix}  \right)=W\left(  \begin{pmatrix}
            x^{-1} & \\ & x^{-1}
        \end{pmatrix}  \begin{pmatrix}
            x & \\  & 1 
        \end{pmatrix}w_2\right)=\omega_{\pi}(x^{-1})(\pi(w_0)\phi)(x).
    \end{align*}
\end{proof}

According to \cite[Lemma 5.2]{IT13}, for any $w(x)\in\CC_c^{\infty}(F^{\times})$, there is a unique smooth function $\widetilde{w}(x)$ on $F^{\times}$ of rapid decay at infinity and with at most polynomial growth at zero such that the local functional equation \eqref{dualfunction} holds as meromorphic functions in $s\in\BC$. The map $w(x)\mapsto \wt{w}(x)$ is 
called the Bessel transform in \cite{IT13}. Some more discussions and explicit formulas related to this map were given in \cite[Section 4]{Cor21} based on the local functional equation of the Rankin-Selberg convolution for $\GL_n\times\GL_1$ from \cite{JPSS83} and \cite{J09}. Over the Archimedean local fields, the map $w(x)\mapsto \wt{w}(x)$ has been 
studied in \cite{Qi20} in the framework of Hankel transforms with the Bessel functions of high rank as the kernel functions. 
The following result says that the map $w(x)\mapsto \wt{w}(x)$ is given by the $\pi$-Fourier transform up to certain normalization. 

\begin{prp}\label{FofTest}
The dual function $\widetilde{w}(x)$ of $w(x)\in\CC_c^{\infty}(F^{\times})$ as defined by \eqref{dualfunction} can be expressed in terms of $\pi$-Fourier transforms:
    \[
    \CF_{\pi,\psi}(w(\cdot)|\cdot|^{1-\frac{n}{2}})=\widetilde{w}(\cdot)|\cdot|^{1-\frac{n}{2}}.
    \]
\end{prp}

\begin{proof}
    Since $w(x)\in\CC_c^{\infty}(F^{\times})$, we have $w(x)|x|^{1-\frac{n}{2}}\in\CC_{c}^{\infty}(F^{\times})$ as well. 
    By \cite[Theorem 3.4]{JL23} and as in the proof of Proposition \ref{S=W}, the right-hand side of \eqref{dualfunction} can be written as 
    \begin{align*}
        \gamma(1-s,\pi\times\chi,\psi)\int_{F^{\times}}w(y)\chi(y)|y|^{1-s-\frac{n-1}{2}}\ud^\times y=\gamma(1-s,\pi\times\chi,\psi)\CZ(1-s,w(\cdot)|\cdot|^{1-\frac{n}{2}},\chi).
    \end{align*}
By the local functional equation in \cite[Theorem 3.10]{JL23}, we have 
\[
\gamma(1-s,\pi\times\chi,\psi)\CZ(1-s,w(\cdot)|\cdot|^{1-\frac{n}{2}},\chi)=\CZ(s, \CF_{\pi,\psi}(w(\cdot)|\cdot|^{1-\frac{n}{2}}),\chi^{-1})
\]
It follows that the left-hand side of \eqref{dualfunction} can be written as 
\[
\int_{F^\times}\wt{w}(y)\chi^{-1}(y)|y|^{s-\frac{n-1}{2}}\ud^\times y
=
\int_{F^\times}\CF_{\pi,\psi}(w(\cdot)|\cdot|^{1-\frac{n}{2}})(y)\chi^{-1}(y)|y|^{s-\frac{1}{2}}\ud^\times y
\]
as meromorphic functions in $s\in\BC$. Since the integrals on both sides of the above equation converge absolutely for $\Re(s)$ sufficiently negative, we choose one of such $s_0\in\BC$ and fix it such that the two smooth functions 
$\wt{w}(y)|y|^{s_0-\frac{n-1}{2}}$ and $\CF_{\pi,\psi}(w(\cdot)|\cdot|^{1-\frac{n}{2}})(y)|y|^{s_0-\frac{1}{2}}$ 
belong to $L^1(F^\times)$. Again, by the general theory as in \cite[Theorem 23.11]{HR79}, we obtain that 
\[
\wt{w}(y)|y|^{s_0-\frac{n-1}{2}}=\CF_{\pi,\psi}(w(\cdot)|\cdot|^{1-\frac{n}{2}})(y)|y|^{s_0-\frac{1}{2}},
\]
which implies that 
$\CF_{\pi,\psi}(w(\cdot)|\cdot|^{1-\frac{n}{2}})(x)=\widetilde{w}(x)|x|^{1-\frac{n}{2}}$,  
as functions on $F^\times$, is smooth, of rapid decay at infinity, and with at most polynomial growth at zero.
\end{proof}

Combining Proposition \ref{FofTest} with the formula in \eqref{FO-k}, we obtain a formula for $\wt{w}(x)$ for any $w\in\CC^\infty_c(F^\times)$.

\begin{cor}\label{dfw-k}
For any $\pi\in \Pi_F(\RG_n)$, the dual function $\wt{w}(x)$ associated with any $w\in\CC^\infty_c(F^\times)$ is given by the following formula:
    \[
    \wt{w}(x)=|x|_F^{\frac{n}{2}-1}\left(k_{\pi,\psi}(\cdot)*(w^\vee(\cdot)|\cdot|_F^{\frac{n}{2}-1})\right)(x),
    \]
    where $k_{\pi,\psi}(x)$ is the $\pi$-kernel function associated with $\pi$ as in \eqref{kernel-ar} and $w^\vee(x)=w(x^{-1})$.
\end{cor}


\section{$\pi$-Bessel functions}\label{sec-BF}

The $\pi$-Fourier transform $\CF_{\pi,\psi}$ can be expressed as a convolution operator with the $\pi$-kernel function $k_{\pi,\psi}$ as in \eqref{kernel-ar} using the 
structures of the $\pi$-Schwartz space $\CS_{\pi}(F^{\times})$ and the $\wt\pi$-Schwartz space $\CS_{\wt{\pi}}(F^{\times})$. When consider 
the $\pi$-Fourier transform $\CF_{\pi,\psi}$ as a transformation from the $\pi$-Whittaker-Schwartz space $\CW_{\pi}(F^{\times})$ to $\wt\pi$-Whittaker-Schwartz space 
$\CW_{\wt\pi}(F^{\times})$, we are able to show that the $\pi$-Fourier transform $\CF_{\pi,\psi}$ can be expressed as a convolution operator with certain Bessel functions as 
the kernel functions. We do this for the Archimedean case and non-Archimedean case, separately.

\subsection{$\pi$-Bessel functions: $p$-adic case}\label{ssec-NABF} 
Assume that $F$ is non-Archimedean. In this case, a basic theory of Bessel functions was developed by E. Baruch in \cite{BE05}, from which we recall some relevant definitions and results on Bessel functions in order to understand the $\pi$-Fourier transform.

Let $\Phi=\{  \alpha_{i,j}=e_i-e_j\mid 1\leq i< j\leq n \}$ be the roots of $\RG_n$ with respect to the $F$-split maximal torus $T_n$, $\Phi^+=\{\alpha_{i,j}\mid i<j  \}$ be the set of positive roots with respect to $B_n$ and $\Phi^{-}=\{\alpha_{i,j}\mid i>j  \}$ be the corresponding set of negative roots. Let $\Delta=\{\alpha_{i,i+1}\mid 1\leq i\leq n-1 \}$ be the set of simple roots. 
Let $\BW$ be the Weyl group of $\RG_n$. For every $w\in \BW$, denote 
\begin{align*}
    S(w)=\{ \alpha\in\Delta\mid w(\alpha)<0  \}\quad \mathrm{and}\quad S^\circ(w)=S(ww_n),
\end{align*}
where $w_n$ is the longest Weyl element of $\RG_n$ as in Proposition \ref{F=W}. We also write
\begin{align*}
    S^-(w)=\{\alpha\in\Phi^+\mid w(\alpha)<0  \}\quad \mathrm{and}\quad S^+(w)=\{\alpha\in\Phi\mid w(\alpha)>0\}.
\end{align*}
Let $N_w^-$ ($N_w^+$ resp.) be the unipotent subgroup associated to $S^-(w)$ ($S^+(w)$ resp.). Let 
\begin{align*}
    T_w=\{t\in T_n\mid \psi(u)=\psi (    w(t)uw(t)^{-1}  ) ,\;\forall u\in N_{w}^{-}  \}.
\end{align*}
For every $\lambda\in X(T_n)\otimes_{\BZ}\BR$, where $X(T_n)$ is the character group of $T_n$, define
\begin{align*}
    |\lambda|(t):=|\lambda(t)|_F,\;\forall t\in T_n.
\end{align*}
Recall from Section \ref{sec-piSpace} that $K=\GL_n(\Fo)$ is the maximal open compact subgroup of $\RG_n$. With the Iwasawa decomposition $\RG_n=N_nT_nK$, for any $g=utk$, we set 
$|\lambda|(g):=|\lambda|(t)$. It is easy to check that this is well defined. 

Let $\pi\in\Pi_F(\RG_n)$ be generic and $\CW(\pi,\psi)$ be the space of Whittaker functions. 
Following \cite[Definition 5.1]{BE05}, we denote by $\CW^{\circ}(\pi,\psi)$ the set of functions $W\in\CW(\pi,\psi)$ such that for every $w\in\BW$ and every $\alpha\in S^{\circ}(w)$, there exist positive constants $D_{\alpha}<E_{\alpha}$ such that if $g\in B_nwB_n$ then $W(g)\neq 0$ implies that
$D_{\alpha}<|\alpha|(g)<E_{\alpha}$. 
For a positive integer $m$, we denote $K_m$ the congruence subgroup given by 
$K_m=\RI_n+\mathrm{M}_n(\Fp^m)$. 
Write $d=\diag(1,\varpi^2,\varpi^4,\cdots,\varpi^{2n-2})\in\RG_n(F)$, 
where $\varpi$ is a fixed uniformizer of $F$. Let $N_n(m)=N_n\cap(d^mK_md^{-m})$. For any $W\in\CW(\pi,\psi)$, denote
\begin{align*}
    W_m(g)=\int_{N_n(m)}W(gn)\psi^{-1}(n)\ud n.
\end{align*}
According to \cite[Theorem 7.3]{BE05}, $W_m\in\CW^{\circ}(\pi,\psi)$ for all sufficiently large $m$. Due to \cite[Proposition 8.1]{BE05}, for $m$ large enough, the integral
$\int_{N_w^{-}}W_m\left( g n\right)\psi^{-1}(n)\ud n$ converges and is independent of $m$ for $g\in N_nT_wwN_w^{-}$. Moreover, by the uniqueness of Whittaker functionals, it follows that there exists a function, which we denote by $j_{\pi,\psi,w}(g)$ such that 
\begin{align*}
    \frac{1}{\mathrm{vol}(N_m)  }\int_{N_w^{-}}W_m\left( g
 n\right)\psi^{-1}(n)\ud n=j_{\pi,\psi,w}(g)W(\RI_n)
\end{align*}
for $g\in N_nT_wwN_w^{-}$. This function $j_{\pi,\psi,w}(g)$ was called the Bessel function of $\pi$ attached to the Weyl group element $w$ in \cite[Section 8]{BE05}. 
Moreover, if $W\in\CW^{\circ}(\pi,\psi)$, then the integral
    $\displaystyle{\int_{N_w^{-}} W(gn)\psi^{-1}(n)\ud n}$
converges absolutely for $g\in N_nT_wwN_w^{-}$ and the Bessel function $j_{\pi,\psi,w}(g)$ has the following integral representation:
\begin{align}\label{j-IR}
    j_{\pi,\psi,w}(g)\cdot W(\RI_n)=\int_{N_w^{-}} W(gn)\psi^{-1}(n)\ud n
\end{align}
according to \cite[Theorem 5.7 and Theorem 8.1]{BE05}.

\begin{lem}\label{smallW}
Let $F$ be a non-Archimedean local field. Define
\begin{align*}
    \CW^{\circ}_{\pi,\psi}(F^{\times}):=\{\omega(x)=|x|^{1-\frac{n}{2}}\cdot W\left( \begin{pmatrix}  x& \\ & \RI_{n-1} \end{pmatrix}\right)\ \colon W\in\CW^{\circ}(\pi,\psi)  \}.
\end{align*}
Then this space can be identified with the space $\CC_c^{\infty}(F^{\times})$:
$\CW^{\circ}_{\pi,\psi}(\pi,\psi)=\CC_c^{\infty}(F^{\times})$. 
In particular, the space $\CW^{\circ}_{\pi,\psi}(\pi,\psi)$ is independent of the choice of the character $\psi$. 
\end{lem}

\begin{proof}
We first prove that
   \begin{align*}
        \left\{ W\left( \begin{pmatrix}
            \cdot & \\ & \RI_{n-1}
        \end{pmatrix}\right)  :\; W\in\CW^{\circ}(\pi,\psi)   \right\}=\CC_c^{\infty}(F^{\times}).
    \end{align*}

    Take $w_*=\begin{pmatrix}
        & \RI_{n-1} \\ 1 & 
    \end{pmatrix}\in\BW$. Then we have that 
    $\alpha_{1,2}=e_1-e_2\in S(w_*w_n)$. It follows that there are positive constants $D$ and $E$ such that
$W\left( \begin{pmatrix}
            x & \\ & \RI_{n-1}
        \end{pmatrix}\right)\neq 0 $ implies that 
        \begin{align*}
            D<|\alpha_{1,2}|\left(\begin{pmatrix}
            x & \\ & \RI_{n-1}
        \end{pmatrix}\right)=|x|<E,
        \end{align*}
which implies that 
\begin{align*}
        \left\{ W\left( \begin{pmatrix}
            \cdot & \\ & \RI_{n-1}
        \end{pmatrix}\right)  :\; W\in\CW^{\circ}(\pi,\psi)   \right\}\subset\CC_c^{\infty}(F^{\times}).
    \end{align*}
On the other hand, take $W\in\CW(\pi,\psi)\neq 0$, for any positive integer $m$, we have
\begin{align*}
    W_m(\RI_n)=\int_{N_m}W(n)\psi^{-1}(n)\ud n=\mathrm{Vol}(N_m)\neq0,
\end{align*}
and since $W_m\in\CW^{\circ}(\pi,\psi)$ for $m$ sufficiently large, we get
\begin{align*}
        \left\{ W\left( \begin{pmatrix}
            \cdot & \\ & \RI_{n-1}
        \end{pmatrix}\right)  \colon W\in\CW^{\circ}(\pi,\psi)   \right\}\neq 0.
    \end{align*}

According to \cite[Corollary 5.5]{BE05}, $\CW^{\circ}(\pi,\psi)$ is invariant under right translations by $B_n$, in particular, for 
$b^{\prime}=\begin{pmatrix}
        b & \\ & \RI_{n-2} 
    \end{pmatrix}\in B$, 
and $W\in\CW^{\circ}(\pi,\psi)$,
where 
$b=\begin{pmatrix}
        t & n \\ & 1
    \end{pmatrix}$, 
we have
\begin{align*}
    (bW)\left(\begin{pmatrix}
        x& \\ & \RI_{n-1}
    \end{pmatrix}\right)=\psi(nx)W\left(\begin{pmatrix}
        tx& \\ & \RI_{n-1}
    \end{pmatrix}\right).
\end{align*}
According to \cite[Lemma 2.9.1]{JL70}, $\CC_c^{\infty}(F^{\times})$ is an irreducible representation under the above action. Hence we obtain that 
\begin{align*}
        \left\{ W\left( \begin{pmatrix}
            \cdot & \\ & \RI_{n-1}
        \end{pmatrix}\right) \colon W\in\CW^{\circ}(\pi,\psi)   \right\}=\CC_c^{\infty}(F^{\times}).
    \end{align*}
Note that $f\mapsto f|\cdot|^{1-\frac{n}{2}}$ is a bijection from $\CC_{c}^{\infty}(F^{\times})$ to itself. Therefore we obtain that 
$\CW^{\circ}_{\pi,\psi}(F^{\times})=\CC_c^{\infty}(F^{\times})$. 
\end{proof}

In order to understand the $\pi$-Fourier transform $\CF_{\pi,\psi}$ and the associated $\pi$-kernel function $k_{\pi,\psi}$ as in \eqref{kernel-ar} in terms of 
the $\pi$-Whittaker-Schwartz space $\CW_{\pi}(F^{\times})$ to $\wt\pi$-Whittaker-Schwartz space $\CW_{\wt\pi}(F^{\times})$,
we define the $\pi$-Bessel function of $\pi$ on $F^\times$, which is related to the one attached to the particular Weyl element $w_*=\begin{pmatrix}
        & \RI_{n-1} \\ 1 & 
    \end{pmatrix}\in\BW$, up to normalization. 
\begin{dfn}\label{piFb-p}
    Let $F$ be a non-Archimedean local field of characteristic zero. For any $\pi\in\Pi_F(\RG_n)$, which is generic, the associated $\pi$-Bessel function $\Fb_{\pi,\psi}(x)$
    on $F^\times$ is defined by 
    \begin{align}\label{bf-p}
        \Fb_{\pi,\psi}(x)=|x|_F^{\frac{1-n}{2}}\cdot j_{\pi,\psi,w_*}\left(  \begin{pmatrix}
            & \RI_{n-1} \\ x^{-1} & 
        \end{pmatrix}   \right)
    \end{align}
    where $w_*=\begin{pmatrix}
        & \RI_{n-1} \\ 1 & 
    \end{pmatrix}\in\BW$ is a Weyl group element of $\RG_n$
\end{dfn}

\begin{prp}\label{k=j-p}
For any $\pi\in\Pi_F(\RG_n)$, which is generic,
as functions on $F^\times$, the $\pi$-kernel function $k_{\pi,\psi}$ as in \eqref{kernel-ar} and the $\pi$-Bessel function as defined in \eqref{bf-p} are related by the following identity:
\[
k_{\pi,\psi}(x)=\Fb_{\pi,\psi}(x)|x|^{\frac{1}{2}}, \quad \forall x\in F^{\times}.
\]
\end{prp}

\begin{proof}
For any $\phi\in\CC_c^{\infty}(F^{\times})\subset\CS_{\pi}(F^{\times})$, we know from Proposition \ref{smallW} that there is some $W\in\CW^{\circ}(\pi,\psi)$ such that 
\begin{align*}
    \phi(x)=W\left(\begin{pmatrix}
        x & \\  & \RI_{n-1}
    \end{pmatrix}\right)|x|^{1-\frac{n}{2}}
\end{align*}
for any $x\in F^{\times}$. 
According to Proposition \ref{F=W}, we have
\begin{align*}
    \CF_{\pi,\psi}(\phi)(x)
    &=|x|^{1-\frac{n}{2}}\int_{F^{n-2}}\left(\pi(w_{n,1})\widetilde{W}\right)\left(\begin{pmatrix}
         x& & \\y & \RI_{n-2} & \\ & & 1 
    \end{pmatrix}\right)\ud y\\
    &=|x|^{1-\frac{n}{2}}\int_{F^{n-2}} W\left( w_*
    \begin{pmatrix}
        x^{-1} & \\ & \RI_{n-1}
    \end{pmatrix}\begin{pmatrix}
        1 & 0 & y_1 & \cdots & y_{n-2} \\
        0  &  1 & \cdots & 0 & 0\\
        \rotatebox{90}{$\cdots$}&\rotatebox{90}{$\cdots$}& \rotatebox{135}{$\cdots$} &\rotatebox{90}{$\cdots$}&\rotatebox{90}{$\cdots$} \\
       0 & 0& \cdots&1 &  0 \\
       0&0 & \cdots&0 & 1
    \end{pmatrix}     \right)\ud y 
\end{align*}
as $w_*=w_nw_{n,1}$. 
According to \cite[Theorem 5.7]{BE05}, the function
\begin{align*}
    (z,y_1,y_2,\cdots,y_{n-2})\mapsto W\left( w_*
    \begin{pmatrix}
        x^{-1} & \\ & \RI_{n-1}
    \end{pmatrix}\begin{pmatrix}
        1 & z & y_1 & \cdots & y_{n-2} \\
        0  &  1 & \cdots & 0 & 0\\
        \rotatebox{90}{$\cdots$}&\rotatebox{90}{$\cdots$}& \rotatebox{135}{$\cdots$} &\rotatebox{90}{$\cdots$}&\rotatebox{90}{$\cdots$} \\
       0 & 0& \cdots&1 &  0 \\
       0&0 & \cdots&0 & 1
    \end{pmatrix}     \right)
\end{align*}
is compactly supported once we fix $x$. Hence the function
\begin{align*}
    f(z,x):=\int_{F^{n-2}} W\left( w_*
    \begin{pmatrix}
        x^{-1} & \\ & \RI_{n-1}
    \end{pmatrix}\begin{pmatrix}
        1 & z & y_1 & \cdots & y_{n-2} \\
        0  &  1 & \cdots & 0 & 0\\
        \rotatebox{90}{$\cdots$}&\rotatebox{90}{$\cdots$}& \rotatebox{135}{$\cdots$} &\rotatebox{90}{$\cdots$}&\rotatebox{90}{$\cdots$} \\
       0 & 0& \cdots&1 &  0 \\
       0&0 & \cdots&0 & 1
    \end{pmatrix}     \right)\ud y.
\end{align*}
belongs to the space $\CC_c^{\infty}(F)$, as a function in $z$ with $x$ fixed, and its Fourier transform $\widehat{f}$ along $z$ at $1$ 
\begin{align*}
    \widehat{f}(1,x)=\int_{F}f(z,x)\psi^{-1}(z)\ud z
\end{align*}
exists. For the Weyl group element $w_*$, it is easy to check that
\[
N_{w_*}^-=\left\{\begin{pmatrix}
        1 & z & y_1 & \cdots & y_{n-2} \\
        0  &  1 & \cdots & 0 & 0\\
        \rotatebox{90}{$\cdots$}&\rotatebox{90}{$\cdots$}& \rotatebox{135}{$\cdots$} &\rotatebox{90}{$\cdots$}&\rotatebox{90}{$\cdots$} \\
       0 & 0& \cdots&1 &  0 \\
       0&0 & \cdots&0 & 1
    \end{pmatrix}\mid z,y_1,\cdots,y_{n-2}\in F\right\}, 
\]
from which we deduce the following formula for $\widehat{f}(1,x)$:
\[
\widehat{f}(1,x)=\int_{N_{w_*}^-}W\left( w_*\begin{pmatrix}
        x^{-1} & \\ & \RI_{n-1}
    \end{pmatrix}n\right)\psi^{-1}(n)\ud n
\]
where $\psi(n)=\psi(z)$. By \eqref{j-IR}, we obtain that 
\[
\widehat{f}(1,x)=j_{\pi,\psi,w_*}\left(w_*\begin{pmatrix}
        x^{-1} & \\ & \RI_{n-1}
    \end{pmatrix}\right)\cdot W(\RI_n)=j_{\pi,\psi,w_*}\left(\begin{pmatrix}
         & \RI_{n-1}\\ x^{-1}& 
    \end{pmatrix}\right)\cdot W(\RI_n).
\]
From the definition of the Bessel function $j_{\pi,\psi}(x)$ in \eqref{bf-p}, we obtain that 
\[
\widehat{f}(1,x)=\Fb_{\pi,\psi}(x)|x|^{\frac{n-1}{2}}W(\RI_n)
\]
as functions in $x\in F^\times$. Now we calculate for a fixed $x\in F^\times$, the Fourier transform $\widehat{f}(t,x)$ with $t\in F^{\times}$,
\begin{align*}
    \widehat{f}(t,x)&=\int_{F}f(z,x)\psi^{-1}(tz)\ud z\\
    &=\int_{N_{w_*}^-}
    W\left( w_*
    \begin{pmatrix}
        (xt)^{-1} & \\ & \RI_{n-1}
    \end{pmatrix}n
    \begin{pmatrix}
        t& \\ & \RI_{n-1}
    \end{pmatrix}    \right)\psi^{-1}(z)\ud n\cdot|t|^{1-n}\\
    &=|t|^{1-n}\Fb_{\pi,\psi}(xt)|xt|^{\frac{n-1}{2}}\pi\left( \begin{pmatrix}
        t & \\ & \RI_{n-1}
    \end{pmatrix}   \right)W(\RI_n).
\end{align*}
According to \cite[Theorem 5.7]{BE05}, we can apply the Fourier inversion formula to obtain
\begin{align*}
    f(0,x)&=\int_F\widehat{f}(t,x)\ud t\\
    &=\int_F \Fb_{\pi,\psi}(xt)|xt|^{\frac{n-1}{2}}W\left(  \begin{pmatrix}
        t &  \\ & \RI_{n-1}
    \end{pmatrix}   \right) |t|^{1-n} \ud t\\
    &=\int_{F^\times} \Fb_{\pi,\psi}(xt)|xt|^{\frac{n-1}{2}}\phi(t)|t|^{1-\frac{n}{2}}\ud^{\times}t.
\end{align*}
Hence we obtain from the above calculation that 
\begin{align*}
    \CF_{\pi,\psi}(\phi)(x)=|x|^{1-\frac{n}{2}}f(0,x)=\int_{F^\times} \Fb_{\pi,\psi}(xt)|tx|^{\frac{1}{2}}\phi(t)\ud^{\times}t.
\end{align*}
On the other hand, we know from \cite[Theorem 5.2]{JL22} that
\begin{align*}
     \CF_{\pi,\psi}(\phi)(x)=\int_F k_{\pi,\psi}(xt)\phi(t)\ud^{\times}t 
\end{align*}
for any $\phi\in\CC_c^{\infty}(F^{\times})$. Therefore, as distributions on $F^\times$, we obtain
that $k_{\pi,\psi}(x)=\Fb_{\pi,\psi}(x)|x|^{\frac{1}{2}}$ 
for any $x\in F^{\times}$. Since both functions are smooth, the identity holds as functions in $x\in F^\times$. 
\end{proof}

Recall that in the $\GL_2$ case, D. Soudry defined in \cite{SD84} the Bessel function $J_{\pi}(x)$ on $F^\times$ by the following equation
\begin{align}\label{Soudry-J}
  \int_{F} W\left( \begin{pmatrix}
     & x \\ -1 & 
\end{pmatrix}  \begin{pmatrix}
    1 & y \\ & 1
\end{pmatrix} \right)\psi^{-1}(y)\ud y =J_{\pi}(x)W(\RI_2)  
\end{align}
for all $W\in\CW(\pi,\psi)$, where the integral converges in the sense that it stabilizes for large compacts as in \cite[Lemma 4.1]{SD84}.
By an elementary computation, we see the relation between these two Bessel functions is
\begin{align*}
    J_{\pi}(x)=\omega_{\pi}(x)\Fb_{\pi,\psi}(-x)|x|^{\frac{1}{2}}
\end{align*}

In \cite{SD84}, Soudry computes the Mellin transform of the product of two Bessel functions instead of showing the gamma factor is the Mellin transform of $J_{\pi}$. 
In fact, we have

\begin{cor}
  \[
  \int_{F^{\times}}^{\mathrm{pv}}J_{\pi}(y)\chi^{-1}(y)|y|^{-s}\ud^{\times}y=\omega_{\pi}(-1)\chi(-1)\gamma(\frac{1}{2},\pi\times\omega_{\pi}^{-1}\chi_s,\psi).
 \]  
\end{cor}

\begin{proof}
    \cite[Theorem 5.2]{JL22} tells us
\[
  \int_{F^{\times}}^{\mathrm{pv}}k_{\pi,\psi}(y)\chi^{-1}(y)|y|^{-s}\ud^{\times}y=\gamma(\frac{1}{2},\pi\times\chi_s,\psi).
\]
Taking into account their relations, we can obtain what we want.
\end{proof}

We refer to \cite{C14} for further discussion of the $\GL_2$-Bessel functions and related topics. 

\subsection{$\pi$-Bessel functions: complex case}\label{ssec-CBF}

If $F=\BC$, let us first recall from \cite{Kna94} the classification of irreducible admissible representations of $\RG_n=\GL_n(\BC)$. For $z\in\BC$, let $[z]=z/\sqrt{z\overline{z}}$ and $|z|_{\BC}=z\overline{z}$, where $\overline{z}$ is the complex conjugate of $z$. For any $l\in\BZ$ and $t\in\BC$, let $\sigma=\sigma(l,t)$ be the representation of $\GL_1(\BC)$ given by
$z\mapsto[z]^l|z|_{\BC}^t$, which we write $[\cdot]^l\otimes|\cdot|_{\BC}^t$. For each $j$ with $1\leq j\leq n$, let $\sigma_j$ be the representation $[\cdot]^{l_j}\otimes|\cdot|_{\BC}^{t_j}$ of $\GL_1(\BC)$. Then $(\sigma_1,\cdots,\sigma_n)$ defines a one-dimensional representation of the diagonal maximal torus $T_n$ of $\RG_n$, which can be extended trivially to a one-dimensional representation of the upper triangular Borel subgroup $B_n$. We set 
\begin{align*}
    \RI(\sigma_1,\cdots,\sigma_n)=\mathrm{ind}^{\RG_n}_{B_n}(\sigma_1,\cdots,\sigma_n),
\end{align*}
which is the unitary induction as in \cite[Chapter VII]{Kna01}. According to \cite{Z75,ZN66}, we have
\begin{thm}[Classification]\label{LCBC}
    The irreducible admissible representations of $\RG_n=\GL_n(\BC)$ can be classified as follows.
    \begin{itemize}
        \item [(1)] If the parameters $t_j$ of $(\sigma_1,\cdots,\sigma_n)$ satisfies
        $\Re \;t_1\geq \Re\; t_2\geq\cdots\geq\Re \;t_n$, 
        then $\RI(\sigma_1,\cdots,\sigma_n)$ has a unique irreducible quotient $\RJ(\sigma_1,\cdots,\sigma_n)$.
        \item [(2)] the representations $\RJ(\sigma_1,\cdots,\sigma_n)$ exhaust the irreducible admissible representations of $G_n$, up to infinitesimal equivalence.
        \item [(3)] Two such representations $\RJ(\sigma_1,\cdots,\sigma_n)$ and $\RJ(\sigma_1^{\prime},\cdots,\sigma_n^{\prime})$ are infinitesimally equivalent if and only if there exists a permutation $j$ of $\{1,\cdots,n\}$ such that $\sigma_i^{\prime}=\sigma_{j(i)}$ for $1\leq i\leq n$.
    \end{itemize}
\end{thm}

According to \cite{Jac79}, the associated local factors can be expressed as follows. 

\begin{thm}[Local Factors]\label{LF}
Let $\pi=\RJ(\sigma_1,\cdots,\sigma_n)$ be an irreducible admissible representation of $\RG_n=\GL_n(\BC)$ with $\sigma_j=[\cdot]^{l_j}\otimes|\cdot|_{\BC}^{t_j}$, where $l_j\in\BZ$ and $t_j\in\BC$ for every $1\leq j\leq n$. The local $L$-factor and local $\epsilon$-factor associated with $\pi$ are given by 
\begin{align*}
    L(s,\pi)=\prod_{j=1}^n 2(2\pi)^{-(s+t_i+\frac{|l_j|}{2})}\Gamma(s+t+\frac{|l_j|}{2})\quad {\rm and}\quad \epsilon(s,\pi,\psi)=\prod_{j=1}^ni^{|l_j|}.
\end{align*}
\end{thm}
 
For any $m\in\BZ$, the local $\gamma$-factor associated with $\pi$ is given by 
\begin{align}\label{lgma}
    \gamma(1-s,\pi\times[\cdot]^m,\psi)&=\epsilon(1-s,\pi\times[\cdot]^m,\psi)\frac{L(s,\widetilde{\pi}\times[\cdot]^{-m})}{L(1-s,\pi\times[\cdot]^m)}\nonumber\\
    &=\prod_{j=1}^n i^{|l_j+m|}\cdot (2\pi)^{1-2(s-t_j)}   \cdot\frac{\Gamma(s-t_j+\frac{|l_j+m|}{2})}{\Gamma(1-s+t_j+\frac{|l_j+m|}{2})}.
\end{align}

\begin{rmk}
Using notations in \cite{Qi20}, we have that 
\[
\gamma(1-s,\pi\times[\cdot]^m,\psi)=G_{(\mathbf{t},\mathbf{l}+m\mathbf{e^n})}(s),
\]
where $\mathbf{t}=(t_1,\cdots,t_n)\in\BC^n$, $\mathbf{l}=(l_1,\cdots,l_n)\in\BZ^n$, and $\mathbf{e^n}=(1,\cdots,1)$. 
\end{rmk}

In \cite{Qi20}, Z. Qi defines a Bessel kernel function $j_{\mathbf{t},\Bl}$ for any $(\Bt,\Bl)\in\BC^n\times\BZ^n$ by the following Mellin-Barnes type integral,
\begin{align*}
    j_{\Bt,\Bl}(x)=\frac{1}{2\pi i}\int_{\CC_{(\Bt,\Bl)}} G_{(\Bt,\Bl)}(s)x^{-2s}\ud s,
\end{align*}
where 
\begin{align*}
    G_{\Bt,\Bl}(s):=\prod_{j=1}^ni^{|l_j|}(2\pi)^{1-2(s-t_j)}\frac{\Gamma(s-t_j+\frac{|l_j|}{2})}{\Gamma(1-(s-t_j)+\frac{|l_j|}{2})}
\end{align*}
and $\CC_{(\Bt,\Bl)}$ is any contour such that
\begin{itemize}
    \item $2\cdot\CC_{\Bt,\Bl}$ is upward directed from $\sigma-\infty$ to $\sigma+\infty$, where $\sigma<1+\frac{1}{n}\left(\Re \;\displaystyle{\left(\sum_{j=1}^n t_j \right) -1   }\right)$,
    \item all the set $t_j-|l_j|-\BN$ lie on the left side of $2\cdot\CC_{(\Bt,\Bl)}$, and
    \item if $s\in2\cdot\CC_{(\Bt,\Bl)}$ and $|\Im\;s|$ large enough, then $\Re\;s=\sigma$.
\end{itemize}
For more details, we refer to \cite[Definition 3.2]{Qi20}. 
Then \cite{Qi20} defines 
\begin{align}\label{Qi-BF-c}
    J_{\Bt,\Bl}(z)=\frac{1}{2\pi}\sum_{m\in\BZ}j_{(\Bt,\Bl+m\mathbf{e^n})}(|z|_{\BC}^{1/2})[z]^m,
\end{align}
and \cite[Lemma 3.10]{Qi20} secures the absolute convergence of this series. The following is the analogy in the complex case of Proposition \ref{k=j-p}.

\begin{prp}\label{k=j-c}
For any $\pi\in\Pi_\BC(\RG_n)$, which is parameterized by $\pi=\pi(\mathbf{t},\mathbf{l})$ as in Theorem \ref{LCBC},
as distributions on $\BC^\times$, the identity: $k_{\pi,\psi}(z)=J_{(\mathbf{t},\mathbf{l})}(z)|z|_{\BC}^{\frac{1}{2}}$ 
holds for any $z\in\BC^{\times}$.
\end{prp}

\begin{proof}
According to \cite[Theorem 3.15]{Qi20}, for any $\phi\in\CC_c^{\infty}(\BC^{\times})$, there is a unique function $\Upsilon(z)|z|_{\BC}^{\frac{1}{2}}\in\mathcal{S}_{\mathrm{sis}}^{(-\Bt,-\Bl)}(\BC^{\times}
)$, which is contained in the space $\CF(\BC^{\times})$ as defined in \cite[Definition 2.1]{JL23}, such that
\begin{align*}
    \CZ(1-s,\phi|\cdot|_{\BC}^{\frac{1}{2}},[\cdot]^m)\gamma(1-s,\pi\times[\cdot]^m,\psi)=\CZ(s,\Upsilon|\cdot|_{\BC}^{\frac{1}{2}},[\cdot]^{-m}).
\end{align*}
It follows that $\Upsilon|\cdot|_{\BC}^{\frac{1}{2}}=\CF_{\pi,\psi}(\phi |\cdot|_{\BC}^{\frac{1}{2}})$ according to \cite[Theorem 2.3, Proposition 3.7, and Corollary 3.8]{JL23}. 
From \cite[Proposition 3.17]{Qi20}, we have that 
\begin{align*}
    \Upsilon(z)=\int_{\BC^{\times}}\phi(y)J_{(\mathbf{t},\mathbf{l})}(zy)\ud y. 
\end{align*}
On the other hand, we have that 
\begin{align*}
    \Upsilon(z)=\int_{\BC^{\times}}\phi(y)k_{\pi,\psi}(yz)|yz|_{\BC}^{-\frac{1}{2}}\ud y
\end{align*}
due to \cite[Theorem 5.1]{JL22}. The $\pi$-kernel function $k_{\pi,\psi}$ is a smooth function on $\BC^{\times}$ according to \cite[Corollary 4.5]{JL22}, while 
the function $J_{\Bt,\Bl}$ is real analytic on $\BC^{\times}$ due to \cite[Proposition 3.17]{Qi20}.  Since $\phi\in\CC_c^{\infty}(\BC^{\times})$ is arbitrary, 
we thus deduce that 
$k_{\pi,\psi}(z)=J_{(\mathbf{t},\mathbf{l})}(z)|z|_{\BC}^{\frac{1}{2}}$ 
for any $z\in\BC^{\times}$, as functions on $\BC^\times$.
\end{proof}

As in Definition \ref{piFb-p}, we introduce the $\pi$-Beesel function on $\BC^\times$.

\begin{dfn}\label{piFb-c}
    For any $\pi\in\Pi_\BC(\RG_n)$, which is generic, the $\pi$-Beesel function $\Fb_{\pi,\psi}(x)$ on $\BC^\times$ is given as
    \[
    \Fb_{\pi,\psi}(x)=J_{(\mathbf{t},\mathbf{l})}(x)
    \]
    for any $x\in\BC^\times$, where $\pi=\pi(\mathbf{t},\mathbf{l})$ is given by the classification in Theorem \ref{LCBC}, and $J_{(\mathbf{t},\mathbf{l})}(x)$ is 
    given in \eqref{Qi-BF-c} and was originally defined in \cite{Qi20}.
\end{dfn}

\subsection{$\pi$-Bessel functions: real case}\label{ssec-RBF}
If $F=\BR$, we recall from \cite{Kna94} the classification of irreducible admissible representations of $\RG_n=\GL_n(\BR)$. For any $l\geq 1$, let $D_l^+$ be the discrete series of $\SL_2(\BR)$, that is, the representation space consists of analytic functions $f$ in the upper half-plane with
\begin{align*}
    \|f\|^2:=\iint |f(z)|^2y^{l-1}\ud x\ud y
\end{align*}
finite, and the action of $g=\begin{pmatrix}
    a & b \\ c & d
\end{pmatrix}$ is given by
\begin{align*}
    D_l^+(g)f(z):=(bz+d)^{-(l+1)}f\left(\frac{az+c}{bz+d}\right).
\end{align*}
Let $\SL_2^{\pm}(\BR)$ be the subgroup of elements $g$ in $\GL_2(\BR)$ with $|\det g|=1$ and
\begin{align*}
    D_l:=\mathrm{ind}_{\SL_2(\BR)}^{\SL_2^{\pm}(\BR)}(D_l^+)
\end{align*}
be the induced representation of $\SL_2^{\pm}(\BR)$, where we still use the unitary induction as in \cite[Chapter VII]{Kna01}. For each pair $(l,t)\in\BZ_{\geq 1}\times\BC$, let $\sigma=\sigma(l,t)$ be the representation of $\GL_2(\BR)$ obtained by tensoring the above representation on $\SL^{\pm}(\BR)$ with the quasi-character $g\mapsto|\det g|^t$, 
that is, $\sigma=D_l\otimes|\det(\cdot)|^t$, 
where $|\cdot|=|\cdot|_\BR$. 
For a pair $(\delta,t)\in\mathbb{Z}/2\mathbb{Z}\times\BC$, let $\sigma=\sigma(\delta,t)$ be the representation of $\GL_1(\BR)=\BR^{\times}$:
$\sigma=\sgn^{\delta}\otimes|\cdot|^t$. 

For any partition of $n$: $(n_1,\cdots,n_r)$ with each $n_j$ equal to $1$ or $2$ and with $\displaystyle{\sum_{j=1}^r n_j=n}$, we associate the block diagonal subgroup 
$M=\GL_{n_1}(\BR)\times\cdots\times\GL_{n_r}(\BR)$. 
For each $1\leq j\leq r$, let $\sigma_j$ be the representation of $\GL_{n_j}(\BR)$ of the form $\sigma(l_j,t_j)$ or $\sigma(\delta_j,t_j)$ as defined above. We extend the tensor product of these representations to the corresponding block upper triangular subgroup $Q$ by making it the identity on the block strictly upper triangular subgroup. We set
\begin{align*}
    \RI(\sigma_1,\cdots,\sigma_r):=\mathrm{ind}_Q^{\RG_n}(\sigma_1,\cdots,\sigma_r).
\end{align*}

\begin{thm}[Classification]\label{thm:CR}
The irreducible admissible representations of $\RG_n=\GL_n(\BR)$ can be classified as follows.
    \begin{itemize}
        \item [(1)] If the parameters $n_j^{-1}t_j$ of $(\sigma_1,\cdots,\sigma_r)$ satisfy
        \begin{align*}
            n_1^{-1}\Re\;t_1\geq n_2^{-1}\Re\;t_2\geq\cdots\geq n_r^{-1}\Re\;t_r,
        \end{align*}
        then $\RI(\sigma_1,\cdots,\sigma_r)$ has a unique irreducible quotient $\RJ(\sigma_1,\cdots,\sigma_r)$.
        \item [(2)] The representations $\RJ(\sigma_1,\cdots,\sigma_r)$ exhaust the irreducible admissible representations of $\RG_n$, up to infinitesimal equivalence.
        \item [(3)] Two such representations $\RJ(\sigma_1,\cdots,\sigma_r)$ and $\RJ(\sigma_1^{\prime},\cdots,\sigma_r^{\prime})$ are infinitesimally equivalent if and only if $r^{\prime}=r$ and there exists a permutation $j(i)$ such that $\sigma_i^{\prime}=\sigma_{j(i)}$ for each $1\leq i\leq r$.
    \end{itemize}
\end{thm}

According to \cite{Jac79} again, the local factors can be expressed as follows:

\begin{thm}[Local Factors]\label{thm:LFR}
    For a representation $\sigma$ of $\GL_1(\BR)$ or $\GL_2(\BR)$ as defined above, denote
    \begin{align*}
        L(s,\sigma) = \begin{cases}
         \pi^{-\frac{s+t+\delta}{2}}\Gamma(\frac{s+t+\delta}{2}) & {\rm if}\  n=1,\;\sigma=\sgn^{\delta}\otimes|\cdot|^t,  \\
         2(2\pi)^{-(s+t+\frac{l}{2})}\Gamma(s+t+\frac{l}{2}) & {\rm if}\  n=2,\;\sigma=D_l\otimes|\det(\cdot)|^t.
                \end{cases}
    \end{align*}
then for $\pi=\RJ(\sigma_1,\cdots,\sigma_r)$, we have
$L(s,\pi)=\prod_{j=1}^r L(s,\sigma_j)$.
Similarly denote
\begin{align*}
    \epsilon(s,\sigma,\psi) =\begin{cases}
         i^{\delta} & {\rm if}\ n=1,\;\sigma=\sgn^{\delta}\otimes|\cdot|^t,  \\
          i^{l+1}& {\rm if}\ n=2,\;\sigma=D_l\otimes|\det(\cdot)|^t.
                \end{cases}
\end{align*}
then the $\epsilon$-factor of $\pi=\RJ(\sigma_1,\cdots,\sigma_r)$ is given by
$\epsilon(s,\pi,\psi)=\prod_{j=1}^l\epsilon(s,\sigma_j,\psi)$.
Finally, the local $\gamma$-factor associated with $\pi=\RJ(\sigma_1,\cdots,\sigma_r)$ is given by
\begin{align*}
    \gamma(s,\pi\times\sgn^{\delta},\psi)=\epsilon(s,\pi\times\sgn^{\delta},\psi)\frac{L(1-s,\widetilde{\pi}\times\sgn^{\delta})}{L(s,\pi\times\sgn^{\delta})},
\end{align*}
where $\widetilde{\pi}$ is the contragredient of $\pi$.

\end{thm}

For any $\phi(x)\in\CC_c^{\infty}(\BR^{\times})$, according to \cite[Theorem 3.10]{JL23}, there is some function $\Upsilon$ such that $\Upsilon|\cdot|^{\frac{1}{2}}=\CF_{\widetilde{\pi},\psi}(\phi|\cdot|^{\frac{1}{2}})$ such that
\begin{align*}
    \CZ(s,\Upsilon|\cdot|^{\frac{1}{2}},\sgn^{\delta})=\gamma(1-s,\pi\times\delta,\psi)\cdot\CZ(1-s,\phi|\cdot|^{\frac{1}{2}},\sgn^{\delta}).
\end{align*}
Due to \cite[Theorem 4.2]{IJ78}, for $\Re \;s=\sigma_0$ large enough, we have
\begin{align*}
    \Upsilon(x)
    &=
    \frac{1}{2}\sum_{\delta\in\BZ/2\BZ} \left(  \frac{1}{2\pi i}
    \int_{\sigma_0-i\infty}^{\sigma_0+i\infty}\gamma(1-s,\pi\times\sgn^{\delta},\psi)\cdot\CZ(1-s,\phi|\cdot|^{\frac{1}{2}},\sgn^{\delta})|x|^{-s}\ud s\right)(\sgn x)^{\delta}\\
    &=\frac{1}{2}\sum_{\delta\in\BZ/2\BZ} \left(  \frac{1}{2\pi i}
    \int_{\sigma_0-i\infty}^{\sigma_0+i\infty}\gamma(1-s,\pi\times\sgn^{\delta},\psi)\cdot \int_{\BR^{\times}}\phi(y)|y|^{-s}\ud y   |x|^{-s}\ud s\right)(\sgn x)^{\delta}.
\end{align*}
We choose a contour $\CC$ with the following three properties: 
\begin{itemize}\label{contour}
    \item [(1)] $\CC$ is upward directed from $\sigma_0^{\prime}-i\infty$ to $\sigma_0^{\prime}+i\infty$, where $\sigma_0^{\prime}$ is small enough, say
    \begin{align*}
        \sigma_0^{\prime}< \frac{1}{2}+\frac{\Re\left(\sum_{j=1}^r n_jt_j\right)-1 }{n},
    \end{align*}
    \item [(2)] The sets $t_j-\delta_j-\BN$ for $n_j=1$ and $t_j-\frac{l_j}{2}-\BN$ for $n_j=2$, $1\leq j\leq r$ all lie on the left side of $\CC$, and 
    \item [(3)] If $s\in\CC$, then for $|\Im\;s|$ large enough, $\Re s=\sigma_0^{\prime}$.
\end{itemize}
Then for $t=|\Im\;s|$ large enough, we have that for fixed $x\in\BR^{\times}$ and $\phi\in\CC_c^{\infty}(\BR^{\times})$,
\begin{align*}
    \int_{\BR^{\times}}\phi(y)|y|^{-s}\ud y  |x|^{-s}\leq C
\end{align*}
for some constant $C$ for all s with $\sigma_0^{\prime}\leq \Re \;s\leq\sigma_0$, and the constant $C$ only depends on $x$, $\varphi$, $\sigma_0$, and $\sigma_0^{\prime}$, and is independent of $t=|\Im\;s|$. It follows that 
\begin{align*}
    \int_{\sigma_0^{\prime}+it}^{\sigma_0+it} \gamma(1-s,\pi\times\sgn^{\delta},\psi) \int_{\BR^{\times}}\phi(y)|y|^{-s}\ud y|x|^{-s}\ud s\leq C^{\prime} t^{-1}
\end{align*}
for some other constant $C^{\prime}$ according to \cite[Lemma 1.3]{Qi20} and Property (1) of the contour $\CC$. Hence, as $t\rightarrow\infty$, the above integral goes zero, and we are able to change the integral from $(\sigma_0-i\infty,\sigma_0+i\infty)$ to $\CC$ according to the Cauchy residue theorem and Property (2) of the contour $\CC$, that is 
\begin{align*}
     &\sum_{\delta\in\BZ/2\BZ} \left(  \frac{1}{2\pi i}
     \int_{\sigma_0-i\infty}^{\sigma_0+i\infty}\gamma(1-s,\pi\times\sgn^{\delta},\psi)\cdot \int_{\BR^{\times}}\phi(y)|y|^{-s}\ud y   |x|^{-s}\ud s\right)(\sgn x)^{\delta}\\
   &\qquad =\sum_{\delta\in\BZ/2\BZ} \left(  \frac{1}{2\pi i}
   \int_{\CC}\gamma(1-s,\pi\times\sgn^{\delta},\psi)\cdot \int_{\BR^{\times}}\phi(y)|y|^{-s}\ud y  |x|^{-s}\ud s\right)(\sgn x)^{\delta}.\\
\end{align*}
According to Property (1) of the contour $\CC$ and \cite[Lemma 1.3]{Qi20} again, we have that 
\begin{align*}
    \int_{\CC}\int_{\BR^{\times}} |\gamma(1-s,\pi\times\sgn^{\delta},\psi)\phi(y)|\cdot|xy|^{-s}\ud y\ud s<\infty.
\end{align*}
Hence we can change the order of integration using Fubini's theorem to obtain that
\begin{align*}
    \Upsilon(x)
    =
    \int_{\BR^{\times}}\phi(y)\left(\frac{1}{2}\sum_{\delta\in\BZ/2\BZ} \frac{1}{2\pi i} 
    \int_{\CC} \gamma(1-s,\pi\times\sgn^{\delta},\psi)|xy|^{-s}\sgn(x)^{\delta}\ud s \right)       \ud y.
\end{align*}
As in Definitions \ref{piFb-p} and \ref{piFb-c}, we define the $\pi$-Bessel function $\Fb_{\pi,\psi}(x)$ on $\BR^\times$ as follows

\begin{dfn}\label{piFb-r}
    For any $\pi\in\Pi_\BR(\RG_n)$, which is generic, the $\pi$-Beesel function $\Fb_{\pi,\psi}(x)$ on $\BR^\times$ is given as
    \begin{align*}
        \Fb_{\pi,\psi}(\pm x)=\frac{1}{2}\sum_{\delta\in\BZ/2\BZ}\frac{1}{2\pi i}\int_{\CC}\gamma(1-s,\pi\times\sgn^{\delta},\psi)|x|^{-s}(\pm)^{\delta}\ud s,\;x>0.
    \end{align*}
\end{dfn}

The integral in Definition \ref{piFb-r} is absolutely convergent to a smooth function in $x$ because of Property (1) of the contour $\CC$ and \cite[Lemma 1.3]{Qi20}. 
Moreover we prove the following proposition, which is the analogy in the real case of Propositions \ref{k=j-p} and \ref{k=j-c}.

\begin{prp}\label{k=j-r}
For any $\pi\in\Pi_\BR(\RG_n)$, which is generic, the $\pi$-kernel function $k_{\pi,\psi}(x)$ and the $\pi$-Bessel function $\Fb_{\pi,\psi}(x)$ are related by the following 
identity as functions on $\BR^\times$, i.e. 
\begin{align*}
    k_{\pi,\psi}(x)=\Fb_{\pi,\psi}(x)|x|_\BR^{\frac{1}{2}},\quad \forall x\in\BR^{\times}.
\end{align*}
\end{prp}

\begin{proof}
    Similar to Proposition 4.9, let us compare the integral 
    \begin{align*}
        \Upsilon(x)=\int_{\BR^{\times}}\phi(y)\Fb_{\pi,\psi}(xy)\ud y
    \end{align*}
    with the integral 
    \begin{align*}
        \Upsilon(x)=\int_{\BR^{\times}}\phi(y)k_{\pi,\psi}(xy)|xy|_\BR^{-\frac{1}{2}}\ud y, 
    \end{align*}
    for any $\phi\in\CC^\infty_c(\BR^\times)$. 
It is clear that $k_{\pi,\psi}=\Fb_{\pi,\psi}(x)|x|_\BR^{\frac{1}{2}}$ because of Definition \ref{piFb-r} and the smothness of both $k_{\pi,\psi}$ (\cite[Corollary 4.5]{JL22}) and $\Fb_{\pi,\psi}$ as functions on $\BR^\times$.
\end{proof}

\begin{rmk}\label{Qi-rk}
    In the special case that $\pi=\pi(\Bt,\mathbf{\delta})=\RJ(\sigma_1,\cdots,\sigma_n)$ as Theorem \ref{thm:CR} with all $n_j=1$ for $1\leq j\leq n$, the $\pi$-Bessel function $\Fb_{\pi,\psi}$ in Definition \ref{piFb-r} is exactly the Bessel function $J_{(\Bt,\mathbf{\delta})}$ defined in \cite[Section 3.3.2]{Qi20}, where $\Bt=(t_1,\cdots,t_n)\in\BC^n$ and $\mathbf{\delta}=(\delta_1,\cdots,\delta_n)\in(\BZ/2\BZ)^n$.
\end{rmk}

\subsection{$\pi$-Bessel functions and dual functions}\label{ssec-piBFDF}
From Definitions \ref{piFb-p}, \ref{piFb-c} and \ref{piFb-r}, for a given $\pi\in\Pi_F(\RG_n)$, we define the (normalized) $\pi$-Bessel function $\Fb_{\pi,\psi}(x)$ on $F^\times$ for every local field $F$ of characteristic zero. In Propositions \ref{k=j-p}, \ref{k=j-c} and \ref{k=j-r}, we obtain the relation between the $\pi$-kernel function 
$k_{\pi,\psi}(x)$ and the $\pi$-Bessel function $\Fb_{\pi,\psi}(x)$. As a record, we state the corresponding formula for the dual function $\wt{w}(x)$ of $w(x)\in\CC^\infty_c(F^\times)$ following Corollary \ref{dfw-k}

\begin{cor}\label{dfw-b}
For any $\pi\in \Pi_F(\RG_n)$, the dual function $\wt{w}(x)$ associated with any $w\in\CC^\infty_c(F^\times)$ is given by the following formula:
    \[
    \wt{w}(x)=|x|_F^{\frac{n-1}{2}}\left(\Fb_{\pi,\psi}(\cdot)*(w^\vee(\cdot)|\cdot|_F^{\frac{n-3}{2}})\right)(x),
    \]
    for all $x\in F^\times$, where $w^\vee(x)=w(x^{-1})$.
\end{cor}


\section{A New Proof of the Voronoi Summation Formula}\label{sec-NPVSF}

In this section, we give a new proof of the Voronoi summation formula based on the $\pi$-Poisson summation formula (\cite[Theorem 4.7]{JL23}), which was recalled in Theorem \ref{thm:PSF}.
Let $k$ be a number field, the notations are all as in Section \ref{sec-piPSF}.

\begin{lem}\label{addtwwt}
At any local place $\nu$ of $k$, for any $w_{\nu}(x)\in\CC_c^{\infty}(k_{\nu}^{\times})$ and $\zeta_\nu\in k_\nu^\times$, the function 
\[
\phi_\nu(x)=\psi_{\nu}(x\zeta_{\nu})\cdot w_{\nu}(x)\cdot|x|_\nu^{1-\frac{n}{2}}
\]
belongs to the space $\CS_{\pi_{\nu}}(k_{\nu}^{\times})$. If $\nu<\infty$ and $\pi_{\nu}$ is unramified, let $^{\circ}W_{\nu}$ be the normalized unramifield Whittaker function 
associated with $\pi_\nu$, then the function 
\[
\varphi_\nu(x)=\psi_{\nu}(x\zeta_{\nu})\cdot{^{\circ}W_{\nu}}\left( \begin{pmatrix}
    x & \\ & \RI_{n-1}
\end{pmatrix}  \right)\cdot|x|_\nu^{1-\frac{n}{2}}
\]
belongs to the space $\CS_{\pi_{\nu}}(k_{\nu}^{\times})$.
\end{lem}

\begin{proof}
    The first claim is trivial because for any $w_{\nu}(x)\in\CC_c^{\infty}(k_{\nu}^{\times})$, we have that 
    \begin{align*}
           \phi_\nu(x)=\psi_{\nu}(x\zeta_{\nu})w_{\nu}(x)|x|_\nu^{1-\frac{n}{2}}\in\CC_{c}^{\infty}(k_{\nu}^{\times})\subset  
  \CS_{\pi_{\nu}}(k_{\nu}^{\times}).
    \end{align*}
As for the second claim, since $\nu<\infty$, we observe that for any given $\zeta_\nu\in k_\nu^\times$, $\psi_{\nu}(x\zeta_{\nu})=1$ if $|x|_\nu$ is small enough. 
Hence we have that 
\begin{align*}
\varphi_\nu(x)=\psi_{\nu}(x\zeta_{\nu})\cdot{^{\circ}W_{\nu}}\left( \begin{pmatrix}
    x & \\ & \RI_{n-1}
\end{pmatrix}  \right)\cdot|x|_\nu^{1-\frac{n}{2}}
\end{align*}
shares the same asymptotic behavior as $|x|\rightarrow 0$ with the function 
$^{\circ} W_{\nu}\left( \begin{pmatrix}
    x & \\ & \RI_{n-1}
\end{pmatrix}  \right)|x|_\nu^{1-\frac{n}{2}}$, 
which belongs to the space $\CS_{\pi_{\nu}}(k_{\nu}^{\times})$ by Proposition \ref{S=W}. Hence we must have the function 
\[
\varphi_\nu(x)=\psi_{\nu}(x\zeta_{\nu})\cdot{^{\circ}W_{\nu}}\left( \begin{pmatrix}
    x & \\ & \RI_{n-1}
\end{pmatrix}  \right)\cdot|x|_\nu^{1-\frac{n}{2}}
\]
belonging to the space $\CS_{\pi_{\nu}}(k_{\nu}^{\times})$.
\end{proof}

\begin{lem}\label{basvswt}
    Let $\nu$ be a finite place such that both $\pi_{\nu}$ and $\psi_{\nu}$ are unramified, the $\pi_\nu$-basic function $\BL_{\pi_{\nu}}(x)\in \CS_{\pi_{\nu}}(k_{\nu}^{\times})$ 
    as defined in \cite[Theorem 3.4]{JL23} enjoys the following formula:
    \[
    \BL_{\pi_{\nu}}(x)={^{\circ}W_{\nu}}\left( \begin{pmatrix}
    x & \\ & \RI_{n-1}
\end{pmatrix}  \right)|x|_\nu^{1-\frac{n}{2}}.
    \]
\end{lem}

\begin{proof}
By \cite[Theorem 3.4]{JL23}, the Mellin transform of the $\pi_\nu$-basic function $\BL_{\pi_{\nu}}(x)$ equals $L(s,\pi\times\chi)$, and the same happens to the function 
${^{\circ}W_{\nu}}\left( \begin{pmatrix}
    x & \\ & \RI_{n-1}
\end{pmatrix}  \right)|x|_\nu^{1-\frac{n}{2}}$ 
by the Rankin-Selberg convolution for $\GL_n\times\GL_1$ in \cite{JPSS83}. Hence the two functions are equal by the Mellin inversion, following the same argument in the proof of 
Proposition \ref{S=W}.
\end{proof}

\begin{lem}\label{addtwbas}
    For the finite places where $\psi_{\nu}$ and $\pi_{\nu}$ are unramified, $\psi_\nu(x)\BL_{\pi_{\nu}}(x)=\BL_{\pi_{\nu}}(x)$.
\end{lem}

\begin{proof}
    According to \cite[Lemma 5.3]{JL23}, the $\pi_\nu$-basic function $\BL_{\pi_{\nu}}$ is supported in $\Fo_{\nu}\setminus\{0\}$.  The assertion follows clearly.
\end{proof}

Now we are ready to prove Theorem \ref{thm:VSF} by using Theorem \ref{thm:PSF}. Recall that $S$ is the finite set of local places of $k$ that contains all the Archimedean places and 
those local places $\nu$ where either $\pi_\nu$ or $\psi_\nu$ is ramified. For any $\zeta\in\BA^S$, we take 
    \[
    w(\cdot)
    :=
    {^{\circ}W^S}\left( \begin{pmatrix}
        \cdot & \\ & \RI_{n-1}
    \end{pmatrix}  \right)\prod_{\nu\in S}w_{\nu}(\cdot)
    =
    {^{\circ}W^S}\left( \begin{pmatrix}
        \cdot & \\ & \RI_{n-1}
    \end{pmatrix}  \right)w_{S}(\cdot)
    \]
and
\begin{align}\label{phi}
    \phi(\cdot):=\psi^S(\cdot\zeta)w(\cdot)|\cdot|_{\BA}^{1-\frac{n}{2}}.
\end{align}
Then the function $\phi(x)$ belongs to the space $\CS_{\pi}(\BA^{\times})$ according to Lemmas \ref{addtwwt}, \ref{basvswt} and \ref{addtwbas}.
It is clear that the function $\phi(x)$ is factorizable: $\phi(x)=\prod_\nu\phi_\nu(x_\nu)$. 
In order to use Theorem \ref{thm:PSF} in the proof, we calculate its local $\pi_\nu$-Fourier transform of $\phi_\nu$ at each place $\nu$. Let $R=R_\zeta$ be as in Theorem \ref{thm:VSF}. 

For the unramified paces $\nu\notin R\cup S$, by Lemma \ref{basvswt}, we obtain that 
\[
\phi_\nu(x_\nu)=\psi_{\nu}(x_\nu\zeta_{\nu})\cdot {^{\circ}W_{\nu}}\left(\begin{pmatrix}
        x_\nu & \\ & \RI_{n-1}
    \end{pmatrix}\right)|x_\nu|_\nu^{1-\frac{n}{2}}=\psi_{\nu}(x_\nu\zeta_{\nu})\BL_{\pi_{\nu}}(x_\nu).
\]
By \cite[Lemma 5.3]{JL23} (or Lemma \ref{addtwbas}), if $\BL_{\pi_\nu}(x_\nu)\neq 0$, then $x_\nu\in\Fo_{\nu}\setminus\{0\}$. 
Since $|\zeta_\nu|_\nu\leq 1$ when $\nu\notin R$, we obtain that $\psi_{\nu}(x_\nu\zeta_{\nu})=1$ if $\BL_{\pi_{\nu}}(x_\nu)\neq 0$. It follows that 
$\phi_\nu(x_\nu)=\BL_{\pi_{\nu}}(x_\nu)$. 
Applying the $\pi_\nu$-Fourier transform to the both sides, we obtain that 
\begin{align}\label{FT-1}
  \CF_{\pi_{\nu},\psi_{\nu}}(\phi_\nu)(x_\nu)
  =  \CF_{\pi_{\nu},\psi_{\nu}}(\BL_{\pi_{\nu}})(x_\nu)
  =\BL_{\widetilde{\pi}_{\nu}}(x_\nu)
    ={^{\circ}\widetilde{W}_{\nu}}\left(\begin{pmatrix}
        x_\nu & \\ & \RI_{n-1}
    \end{pmatrix}\right)|x_\nu|_\nu^{1-\frac{n}{2}}
\end{align}
according to \cite[Theorem 3.10]{JL23}. Note that $\BL_{\widetilde{\pi}_{\nu}}$ the basic function in the $\wt{\pi}_\nu$-Schwartz space $\CS_{\widetilde{\pi}_{\nu}}(k_{\nu}^{\times})$ and $^{\circ}\widetilde{W}_{\nu}\in \CW(\widetilde{\pi}_{\nu},\psi^{-1}_{\nu})$, the Whittaker model of $\wt{\pi}_\nu$.

At $\nu\in S$, the function $\phi_\nu(x_\nu)$ takes the following form 
$\phi_\nu(x_\nu)=w_\nu(x_\nu)|x_\nu|_\nu^{1-\frac{n}{2}}$. By Proposition \ref{FofTest}, we obtain that 
  \begin{align}\label{FT-2}
    \CF_{\pi_{\nu},\psi_{\nu}}(\phi_\nu)(x_\nu)  
 = \CF_{\pi_{\nu},\psi_{\nu}}(w_{\nu}(\cdot)|\cdot|^{1-\frac{n}{2}})(x_\nu)
 =\widetilde{w}_{\nu}(x_\nu)|x_\nu|_\nu^{1-\frac{n}{2}}.  
  \end{align}    
Finally, at the local places $v\in R$, since $R$ is disjoint to $S$, the function $\phi_\nu$ takes the following form
\[
\phi_\nu(x_\nu)=\psi_{\nu}(x_\nu\zeta_{\nu})\cdot {^{\circ}W_{\nu}}\left(\begin{pmatrix}
        x_\nu & \\ & \RI_{n-1}
    \end{pmatrix}\right)|x_\nu|_\nu^{1-\frac{n}{2}}
\]
with $|\zeta_\nu|_\nu>1$.
Recall from Section \ref{ssec-NABF} that $\alpha_{i,i+1}$ be the simple root for the root system $\Phi$ with respect to $(\RG_n,B_n,T_n)$. The one-parameter subgroups
associated with $\alpha_{1,2}$ and $\alpha_{2,1}$ are given by  
\begin{align}\label{1pg}
    \chi_{\alpha_{1,2}}(u):=\begin{pmatrix}
        1 & u & \\ & 1 & \\ & & \RI_{n-2}
    \end{pmatrix}\quad {\rm and}\quad 
    \chi_{\alpha_{2,1}}(u):=\begin{pmatrix}
        1 &  & \\ u& 1 & \\ & & \RI_{n-2}
    \end{pmatrix}.
\end{align}
Then the function $\phi_\nu$ can be written as 
\[
\phi_\nu(x_\nu)={^{\circ}W_{\nu}}\left(\begin{pmatrix}
        x_\nu & \\ & \RI_{n-1}
    \end{pmatrix} \chi_{\alpha_{1,2}}(\zeta_\nu)\right)|x_\nu|_\nu^{1-\frac{n}{2}}
    =W_{\zeta_\nu}\left(\begin{pmatrix}
        x_\nu & \\ & \RI_{n-1}
    \end{pmatrix} \right)|x_\nu|_\nu^{1-\frac{n}{2}}
\]
where $W_{\zeta_\nu}(g):={^{\circ}W_{\nu}}(g\chi_{\alpha_{1,2}}(\zeta_\nu))$.
It is clear that $W_{\zeta_\nu}\in\CW(\pi_\nu,\psi_\nu)$. 
By Proposition \ref{F=W}, the $\pi_\nu$-Fourier transform of $\phi_\nu$ is give by 
\[
\CF_{\pi_{\nu},\psi_{\nu}}(\phi_\nu)(x_\nu)
    =|x_\nu|_\nu^{1-\frac{n}{2}}\int_{k_\nu^{n-2}}\left(\pi_\nu(w_{n.1})\wt{W_{\zeta_\nu}}\right)
    \left(\begin{pmatrix}
          x_\nu & & \\ y & \RI_{n-2} & \\ & & 1
      \end{pmatrix}\right)\ud y. 
\]
Since 
\begin{align*}
    \pi_\nu(w_{n.1})\wt{W_{\zeta_\nu}}
    \left(\begin{pmatrix}
          x & & \\ y & \RI_{n-2} & \\ & & 1
      \end{pmatrix}\right)
      &= \wt{W_{\zeta_\nu}}
    \left(\begin{pmatrix}
          x & & \\ y & \RI_{n-2} & \\ & & 1
      \end{pmatrix}w_{n,1}\right)\\
     &=
      {^\circ W}_{\nu}
    \left(w_n\begin{pmatrix}
          x & & \\ y & \RI_{n-2} & \\ & & 1
      \end{pmatrix}^{-t}w_{n,1}^{-t}\chi_{\alpha_{1,2}}(\zeta_\nu)\right)
\end{align*}
where $\wt{W}(g)=W(w_ng^{-t})$ with $g^{-t}:={^tg^{-1}}$, we obtain that 
\[
\pi_\nu(w_{n.1})\wt{W_{\zeta_\nu}}
    \left(\begin{pmatrix}
          x & & \\ y & \RI_{n-2} & \\ & & 1
      \end{pmatrix}\right)
      =
     \wt{{^\circ W}_{\nu}}
    \left(\begin{pmatrix}
          x & & \\ y & \RI_{n-2} & \\ & & 1
      \end{pmatrix}w_{n,1}\chi_{\alpha_{2,1}}(-\zeta_\nu)\right).
\]
Hence the $\pi_\nu$-Fourier transform of $\phi_\nu$ can be written as 
\begin{align*}
\CF_{\pi_{\nu},\psi_{\nu}}(\phi_\nu)(x_\nu)
&=
|x_\nu|_\nu^{1-\frac{n}{2}}\int_{k_\nu^{n-2}}\wt{{^\circ W}_{\nu}}
    \left(\begin{pmatrix}
          x_\nu & & \\ y & \RI_{n-2} & \\ & & 1
      \end{pmatrix}w_{n,1}\chi_{\alpha_{2,1}}(-\zeta_\nu)\right) \ud y\\
     &=
     |x_\nu|_\nu^{1-\frac{n}{2}}
     \int_{k_{\nu}^{n-2}}\widetilde{^\circ W}_v \left( 
      \begin{pmatrix}
          x_\nu & & \\  & \RI_{n-2} & \\ & & 1
      \end{pmatrix} \begin{pmatrix}
          1 & & \\y & \RI_{n-2} & \\ & & 1 
      \end{pmatrix}   w_{n,1} 
      \chi_{\alpha_{2,1}}(-\zeta_\nu)    \right)\ud y.      
\end{align*}
By the explicit computation of the last integral in \cite[Section 2.6]{IT13}, we obtain that 
\begin{align}\label{FT-3}
    \CF_{\pi_{\nu},\psi_{\nu}}(\phi_\nu)(x_\nu)
=
|x_\nu|_\nu^{1-\frac{n}{2}}
      \Kl_{\nu}(x_\nu,\zeta_{\nu},{^{\circ}\widetilde{W}_{ \nu}}). 
\end{align}
Thus, by \eqref{FT-1}, \eqref{FT-2}, and \eqref{FT-3}, we obtain a formula for the $\pi$-Fourier transform of $\phi$, which is the product of the local $\pi_\nu$-Fourier transform of $\phi_\nu$ at all local places $\nu$. 

\begin{prp}\label{FTphi}
    Let $\phi\in\CS_\pi(\BA^\times)$ be the function as defined in \eqref{phi}. The $\pi$-Fourier transform of $\phi$ can be explicitly written as 
    \[
    \CF_{\pi,\psi}(\phi)(x)
    =
    \prod_{\nu}\CF_{\pi_{\nu},\psi_{\nu}}(\phi_\nu)(x_\nu)
    =
    |x|_\BA^{1-\frac{n}{2}}\Kl_R(x,\zeta,\wt{^\circ W}_R)\ \wt{^\circ W}^{S\cup R}\left(\begin{pmatrix}x&\\ &\RI_{n-1}\end{pmatrix}
    \right)\wt{w}_S(x),
    \]
    where the Kloosterman integral is given by 
$\Kl_R(x,\zeta,\wt{^\circ W}_R)=\prod_{\nu\in R}\Kl_{\nu}(x_\nu,\zeta_{\nu},{^{\circ}\widetilde{W}_{ \nu}})$.
\end{prp}

Finally we write the summation on the one side as 
\[
\sum_{\alpha\in k^{\times}}\phi(\alpha)
=
\sum_{\alpha\in k^{\times}}\psi^S(\alpha \zeta)\ {^{\circ}W^S}\left( \begin{pmatrix}
        \alpha & \\ & \RI_{n-1}
    \end{pmatrix}  \right)w_{S}(\alpha)
\]
and that on the other side as 
\[
\sum_{\alpha\in k^{\times}}\CF_{\pi,\psi}(\phi)(\alpha)
=
\sum_{\alpha\in k^{\times}}
\Kl_R(\alpha,\zeta,\wt{^\circ W}_R)\ \wt{^\circ W}^{S\cup R}\left(\begin{pmatrix}\alpha&\\ &\RI_{n-1}\end{pmatrix}
    \right)\wt{w}_S(\alpha),
\]
because $|\alpha|_{\BA}=1$ for every $\alpha\in k^{\times}$. 
By the $\pi$-Poisson summation formula in Theorem \ref{thm:PSF}, which is 
\[
\sum_{\alpha\in k^{\times}}\phi(\alpha)=\sum_{\alpha\in k^{\times}}\CF_{\pi,\psi}(\phi)(\alpha),
\]
we deduce the Voronoi formula in Theorem \ref{thm:VSF}:
\[
\sum_{\alpha\in k^{\times}}\psi^S(\alpha \zeta)\ {^{\circ}W^S}\left( \begin{pmatrix}
        \alpha & \\ & \RI_{n-1}
    \end{pmatrix}  \right)w_{S}(\alpha)
=
\sum_{\alpha\in k^{\times}}
\Kl_R(\alpha,\zeta,\wt{^\circ W}_R)\ \wt{^\circ W}^{S\cup R}\left(\begin{pmatrix}\alpha&\\ &\RI_{n-1}\end{pmatrix}
    \right)\wt{w}_S(\alpha).
\]
We deduce the Voronoi summation formula for $\GL_n$ from the $\pi$-Poisson summation formula in Theorem \ref{thm:VSF}.

\begin{rmk}\label{Corbett-rmk}
In \cite[Theorem 3.4]{Cor21}, Corbett extends the Voronoi formula in Theorem \ref{thm:VSF} to a more general situation by allowing the local component $\phi_\nu$ at 
$\nu\in R$ to be more general functions in $\CW_{\pi_\nu}(k_\nu^\times)$. More precisely, if one take 
$\phi_{\nu}(x)=\psi_{\nu}(x \zeta_{\nu})\cdot w_{\nu}(x)\cdot |x|_{\nu}^{1-\frac{n}{2}}$ 
for $\nu\in S$  and 
$\phi_{\nu}(x)=\psi_{\nu}(x\zeta_{\nu})\cdot W_{\nu}\left(  \begin{pmatrix}
        x & \\ & \RI_{n-1}
    \end{pmatrix} \xi \right)\cdot|x|_{\nu}^{1-\frac{n}{2}}$ 
for $\nu\notin S$, where $w_{\nu}\in \CC_c^{\infty}(F^{\times})$, $W_{\nu}\in \CW(\pi_{\nu},\psi_{\nu})$ and $S$, $\zeta$, $\xi$ are as in \cite[Theorem 3.4]{Cor21}, 
then according to Lemmas \ref{addtwwt}, \ref{basvswt} and \ref{addtwbas}, the function $\phi:=\otimes_{\nu}\phi_{\nu}\in\CS_{\pi}(\BA^{\times})$. 
It is clear that the proof of Proposition \ref{FTphi} works for such special choices of functions $\phi$ as well. In particular, we obtain from 
Proposition \ref{F=W} that at each local place $\nu$, the Fourier transform $\CF_{\pi_{\nu},\psi_{\nu}}(\phi_{\nu})$ is equal to
the function $\mathfrak{H}_{\nu}(x;\zeta_{\nu},\xi_{\nu})|x|^{1-\frac{n}{2}}$ in \cite[ proof of Theorem 3.4]{Cor21}. The extended Voronoi formula for $\GL_n$ in \cite[Theorem 3.4]{Cor21} by using the Rankin-Selberg convolution for $\GL_n\times\GL_1$, can be deduced by the same argument as in our proof of Theorem \ref{thm:VSF} from the 
$\pi$-Poisson summation formula in \cite[Theorem 4.7]{JL23}. We omit further details. 
\end{rmk}


\section{On the Godement-Jacquet Kernels}\label{sec-GJK}

For any $\pi\in\CA_\cusp(\RG_n)$, the goal of this section is to define the Godement-Jacquet kernels for $L_f(s,\pi)$ and their dual kernels, and to
prove the $\pi$-versions of \cite[Theorem 1.1]{Clo22}, which can be viewed as the case of $n=1$ and is recalled in Theorem \ref{thm:CTh11}. 

\subsection{Godement-Jacquet kernel and its dual}\label{ssec-GJKD}
We recall from \cite[Section 4.2]{JL23} the global 
zeta integral for the standard $L$-function $L(s,\pi)$ as stated in \eqref{zetaG10} is 
\begin{align}\label{zetaG1}
    \CZ(s,\phi)=\int_{\BA^\times}\phi(x)|x|_\BA^{s-\frac{1}{2}}\ud^\times x
\end{align}
for any $\phi\in\CS_\pi(\BA^\times)$. 
By \cite[Theorem 4.6]{JL23} the zeta integral $\CZ(s,\phi)$ converges absolutely for $\Re(s)>\frac{n+1}{2}$, admits analytic continuation to an entire function in $s\in\BC$, and satisfies the functional equation 
\begin{align}
    \CZ(s,\phi)=\CZ(1-s,\CF_{\pi,\psi}(\phi))
\end{align}
where $\CF_{\pi,\psi}$ is the $\pi$-Fourier transform as defined in \eqref{eq:1-FO}. As explained in \cite{JL23}, this is a reformulation of the Godement-Jacquet theory for the 
standard $L$-functions $L(s,\pi)$. 

Consider the fibration through the idele norm map $|\cdot|_\BA$:
\[
1\to\BA^1\to\BA^\times\to\BR^\times_+\to1
\]
where $\BR^\times_+=\{x\in\BR^\times\ \colon\ x>0\}$ and $\BA^1=\{x\in\BA^\times\ \colon\ |x|_\BA=1\}$. 
One can have a suitable Haar measure $\ud^\times \Fa$ on $\BA^{1}$ that is compatible with the Haar measures $\ud^\times x$ 
on $\BA^{\times}$ and the Haar measure $\ud^\times t$ on $\BR^{\times}_+=\BA^{\times}/\BA^{1}$. Write $\BA^\times=\BA_\infty^\times\times\BA_f^\times$, where $\BA_\infty^\times=\prod_{\nu\in|k|_\infty}k_\nu^\times$, and $\BA_f^\times$ is the subset of $\BA^\times$ consisting of elements $(x_\nu)\in\BA^\times$ with $x_\nu=1$ for all $\nu\in|k|_\infty$. 

When $\Re(s)>\frac{n+1}{2}$, the absolutely convergent zeta integral $\CZ(s,\phi)$ as in \eqref{zetaG1} can be written as 
\begin{align}\label{zeta1+1}
    \int_{\BA^{\times}}\phi(x)|x|_\BA^{s-\frac{1}{2}}\ud^{\times}x
    =\int_{1}^{\infty}\int_{\BA^{1}}\phi(t\Fa)t^{s-\frac{1}{2}}\ud^\times\Fa\ud^{\times}t+\int_{0}^{1}\int_{\BA^{1}}\phi(t\Fa)t^{s-\frac{1}{2}}\ud^\times\Fa\ud^{\times}t.
\end{align}

\begin{prp}\label{zeta>1}
    The first integral on the right-hand side of \eqref{zeta1+1}:  
\begin{align}
    \int_{1}^{\infty}\int_{\BA^{1}}\phi(t\Fa)t^{s-\frac{1}{2}}\ud^\times\Fa\ud^{\times}t
    \end{align}
    converges absolutely at any $s\in\BC$ and is holomorphic as a function in $s\in\BC$, for any $\phi\in\CS_\pi(\BA^\times)$.
\end{prp}

\begin{proof}
Let $\phi_f=\otimes_{\nu}\phi_{\nu}\in\CS_{\pi_f}(\BA_f^{\times})$ be a factorizable $\pi$-Schwartz function. Let $S=S(\pi,\psi,\phi_f)$ be a finite subset $S$ of $|k|=|k|_\infty\cup|k|_f$ (the set of all local places of $k$) that contains $|k|_\infty$ and such that for any $\nu\notin S$ both $\pi_{\nu}$ and $\psi_{\nu}$ are unramified and $\phi_{\nu}=\BL_{\pi_{\nu}}$, the basic function in $\CS_{\pi_{\nu}}(k_{\nu}^{\times})$ as in \cite[Theorem 3.4]{JL23}. Write 
$S_f=S\cap|k|_f=\{\nu_1,\nu_2,\cdots,\nu_\kappa\}$. 
According to \cite[Proposition 5.5 and Lemma 5.2]{JL23}, there is a positive real number $s_{\pi}$, which depends only on the given $\pi\in\CA_\cusp(\RG_n)$, 
such that for any real number $a_0>s_{\pi}$, the limit 
$\lim_{|x|_{\nu}\rightarrow 0}\phi_{\nu}(x)|x|_{\nu}^{a_0}=0$
holds for every $\phi_\nu\in\CS_{\pi_\nu}(k_\nu^\times)$ and for every $\nu\in|k|$. From the definition of the $\pi_\nu$-Schwartz space $\CS_{\pi_\nu}(k_\nu^\times)$ in \eqref{piSS} and \cite[Proposition 3.7]{JL23}, we know that $\phi_{\nu}(x)=0$ when $|x|_{\nu}$ is large enough for all $\nu<\infty$. Hence for every $\nu\in S_f$, there is a constant $C_{\nu}>0$ such that
$|\phi_{\nu}(x)|\leq C_{\nu}|x|_{\nu}^{-a_0}$.
By \cite[Lemma 5.3]{JL23}, there is a positive real number $b_{\pi}>s_{\pi}>0$, which also depends only on the given $\pi$, such that for any $b_0>b_{\pi}$, we have that $|\BL_{\pi_{\nu}}(x)|\leq|x|_{\nu}^{-b_0}$
holds for every $\nu\notin S$. 
It is clear that for any constant $c>b_{\pi}$ and constant $C_1$ with $\displaystyle{\max_{\nu\in S_f}\{C_\nu,1\}\leq C_1}$, we must have that the inequality:
\begin{align}\label{es-phif}
    |\phi_f(x_f)|\leq C_1|x_f|_{\BA}^{-c}
\end{align}
holds for every $x_f\in\BA^{\times}_f$.

We first estimate the inner integral, which can be written as 
\begin{align}\label{es1}
  \int_{\BA^{1}} |\phi(t\Fa)|\ud^\times\Fa=\int_{\BA^{1}/k^{\times}}\sum_{\gamma\in k^{\times}}|\phi(\gamma t\Fa)|\ud^\times\Fa.  
\end{align}
Fix a $\nu_0\in|k|_\infty$ and a section $\BR_+^{\times}\rightarrow k_{\nu_0}^{\times}\hookrightarrow\BA^{\times}$ of the norm map $\BA^{\times}\rightarrow\BR^{\times}_+$ and view 
$t\in\BR_+^{\times}$ as the $\nu_0$-component of $\BA^\times$. 
Define 
\[
\ell:\BA^{1}\cap( \BA_\infty^\times \times \Fo_f^\times)\rightarrow \BR^{r}, \quad 
\Fa\mapsto (\cdots,\log |\Fa|_{\nu},\cdots)_{\nu\in |k|_{\infty}-\{\nu_0\}}
\]
where $\Fo_f^\times=\prod_{\nu<\infty}\Fo_{\nu}^{\times}$, and $r:=r_1+r_2-1$ with $r_1$ being the number of real places and $r_2$ the number of complex ones. 
Let $\{\epsilon_i\}_{1\leq i\leq r}$ be a basis for the group of units in the ring of integers in $k$ modulo the group of roots of unity in $k$, and set  
\begin{align*}
    P=\left\{\sum_{i=1}^r x_i\ell(\epsilon_i)\ \colon 0\leq x_i< 1,\;\forall 1\leq i\leq r  \right\}\quad {\rm and}\quad 
    E_0=\left\{\Fa\in \ell^{-1}(P) \ \colon 0\leq \arg \Fa_{\nu_0}<\frac{2\pi}{\hbar_k}  \right\}
\end{align*}
where $\hbar_k$ is the class number of $k$. We choose representatives $\Fa^{(1)},\cdots,\Fa^{(\hbar_k)}$ of idele classes, and define $\displaystyle{E:=\cup_{i=1}^{\hbar_k} E_0\Fa^{(i)}}$. Then $E$ is a fundamental domain of $k^{\times}\backslash\BA^{1}$ according to \cite[Theorem 4.3.2]{Tat67}, which is compact. Hence we can write \eqref{es1} as 
\begin{align}\label{es2}
  \int_{\BA^{1}} |\phi(t\Fa)|\ud^\times\Fa=\int_E \sum_{\gamma\in k^{\times}}|\phi(\gamma t\Fa)|\ud^\times\Fa.  
\end{align}

Without loss of generality, we may take $\phi=\phi_\infty\otimes\phi_f\in\CS_\pi(\BA^\times)=\CS_{\pi_\infty}(\BA_\infty^\times)\otimes\CS_{\pi_f}(\BA_f^\times)$.
Write $t\Fa=(\alpha_\infty,\alpha_f)\in\BA^\times=\BA_\infty^\times\times\BA_f^\times$.  By \eqref{es-phif}, we have that 
\[
|\phi(\gamma t\Fa)|
=|\phi_{\infty}(\gamma\alpha_{\infty})\cdot\phi_f(\gamma\alpha_f)|\leq C_1|\phi_{\infty}(\gamma\alpha_{\infty})|\cdot|\gamma\alpha_f|_f^{-c}
\]
for any constant $c>b_\pi$. Since $\Fa\in\BA^1$, we must have that 
\[
|\gamma\alpha_f|_f^{-c}=|\gamma\alpha_\infty|_\infty^{c}\cdot|\gamma(\alpha_\infty,\alpha_f)|_\BA^{-c}
=
|\gamma\alpha_\infty|_\infty^{c}\cdot|\gamma t\Fa|_\BA^{-c}
=
|\gamma\alpha_\infty|_\infty^{c}\cdot|t|_\BA^{-c}
=
|\gamma\alpha_\infty|_\infty^{c}\cdot t^{-c}. 
\]
Hence we obtain 
\begin{align}\label{es3}
    |\phi(\gamma t\Fa)|\leq C_1|\phi_{\infty}(\gamma\alpha_{\infty})|\cdot|\gamma\alpha_\infty|_\infty^{c}\cdot t^{-c}. 
\end{align}
Since $\Fa$ belongs to a compact set $E$, the Archimedean part of $\Fa$ belongs to a compact subset of $\BA_\infty^\times$. Hence there is a constant $C_2$ such that 
\[
\sum_{\gamma\in k^{\times}}|\phi(\gamma t\Fa)|\leq C_2\cdot t^{-c}\cdot\sum_{\gamma\in k^{\times}} |\phi_{\infty}(\gamma t)|\cdot|\gamma t|_\infty^{c}.
\]
For $\phi_{\infty}\in\CS_{\pi_{\infty}}(\BA_{\infty}^{\times})$, we know from \cite[Proposition 3.7]{JL23} that 
$\phi_{\infty}(x)|x|_\infty^{c}$ 
for any constant $c$ is of rapid decay as $|x|_\infty\to\infty$. From the choice of the fundamental domain $E$, we must have that $\alpha_f\in\Fo_{f}^{\times}$. Due to \cite[Lemma 5.3]{JL23}, there are integers $e_1,\cdots,e_\kappa$ such that for $\gamma\in k^{\times}$, if $\phi(\gamma t\Fa)\neq 0$, then $\gamma\in \Fm:=\Fp_1^{e_1}\cdots\Fp_\kappa^{e_\kappa} $. According to \cite[Proposition 5.2]{Neu99}, the image of $\Fm$ in $\BA^{\times}_{\infty}$ is a lattice, and there is a constant $C_3$ such that the (partial) theta series 
\[
\sum_{\gamma\in k^{\times}} |\phi_{\infty}(\gamma t)|\cdot|\gamma t|_\infty^{c}\leq C_3.
\]
Thus we obtain that $\sum_{\gamma\in k^{\times}}|\phi(\gamma t\Fa)|\leq C_2C_3 t^{-c}$
and there is a constant $C_4$ such that
$\int_{\BA^{1}}|\phi(t\Fa)|\ud\Fa\leq C_4 t^{-c}$.
It follows that the integral
$\int_{1}^{\infty}\int_{\BA^{1}}\phi(t\Fa)t^{s-\frac{1}{2}}\ud^\times\Fa\ud^{\times}t$ 
converges absolutely as long as $\Re(s)<c+\frac{1}{2}$ for any $c>b_\pi$. 
Since $c$ is arbitrarily large with $c>b_\pi$, we obtain that the integral 
\[
\int_{1}^{\infty}\int_{\BA^1}\phi(t\Fa)t^{s-\frac{1}{2}}\ud^\times\Fa\ud^{\times} t
\]
converges absolutely for any $s\in\BC$ and hence is holomorphic as a function in $s\in\BC$.

Since a general element in $\CS_{\pi}(\BA^{\times})$ is a finite linear combination of the factorizable functions, it is clear that the above statement for the integrals 
hold for general $\phi\in\CS_{\pi}(\BA^{\times})$.
\end{proof}

From the above proof, we also obtain 
\begin{cor}\label{inner-c}
    For any $t\in\BR^{\times}_+$, the inner integral
\begin{align}\label{int-Fa}
    \int_{\BA^{1}}\phi(t\Fa)\ud^\times\Fa
\end{align}
always converges absolutely for any $\phi\in\CS_\pi(\BA^\times)$.
\end{cor}

By using the $\pi$-Poisson summation formula (Theorem \ref{thm:PSF}), we obtain 

\begin{prp}\label{prp:intFa}
For any $t\in\BR^\times_+$, the following identity 
\[
\int_{\BA^{1}}\phi(t\Fa)t^s\ud^\times\Fa=\int_{\BA^{1}}\CF_{\pi,\psi}(\phi)(t^{-1}\Fa)t^{s}\ud^\times\Fa
\]
holds for any $\phi\in\CS_\pi(\BA^\times)$.
\end{prp}

\begin{proof}
From \eqref{int-Fa}, the integral $\int_{\BA^{1}}\phi(t\Fa)t^s\ud^\times\Fa$ 
converges absolutely for any $t\in\BR^\times_+$ and for any $\phi\in\CS_\pi(\BA^\times)$. 
We write 
\begin{align*}
    \int_{\BA^{1}}\phi(t\Fa)t^s\ud^\times\Fa
&=\sum_{\alpha\in k^{\times}}\int_{\alpha E}\phi(t\Fa)t^{s}\ud^\times\Fa\\
&=\sum_{\alpha\in k^{\times}}\int_E \phi(\alpha t\Fa)t^s\ud^\times\Fa
=\int_E \left( \sum_{\alpha\in k^{\times}} \phi(\alpha t\Fa)  \right)t^s\ud^\times\Fa
\end{align*}
where $E$ is the fundamental domain of $k^\times$ in $\BA^1$ as above, which is compact. By Theorem \ref{thm:PSF}:
\[
\sum_{\alpha\in k^{\times}} \phi(\alpha t\Fa) = \sum_{\alpha\in k^{\times}}\CF_{\pi,\psi}(\phi)(\frac{\alpha}{t\Fa}),
\]
we obtain that 
\begin{align*}
\int_{\BA^{1}}\phi(t\Fa)t^s\ud^\times\Fa
=\int_E \left(\sum_{\alpha\in k^{\times}}\CF_{\pi,\psi}(\phi)(\frac{\alpha}{t\Fa})\right)t^s\ud^\times\Fa
=\int_{\BA^{1}}\CF_{\pi,\psi}(\phi)(t^{-1}\Fa)t^s\ud^\times\Fa,
\end{align*}
where all changes of the order of integrations are verified due to the absolute convergence. 
\end{proof}

Applying Proposition \ref{prp:intFa} to the second integral on the right-hand side of \eqref{zeta1+1}, we obtain that for $\Re(s)>\frac{n+1}{2}$, 
\begin{align}\label{zeta<1}
    \int_{0}^{1}\int_{\BA^{1}}\phi(t\Fa)t^{s-\frac{1}{2}}\ud^\times\Fa\ud^{\times}t
    &=
    \int_{0}^{1}\int_{\BA^{1}}\CF_{\pi,\psi}(\phi)(t^{-1}\Fa)t^{s-\frac{1}{2}}\ud^\times\Fa\ud^{\times}t
    =
    \int_{1}^\infty\int_{\BA^{1}}\CF_{\pi,\psi}(\phi)(t\Fa)t^{\frac{1}{2}-s}\ud^\times\Fa\ud^{\times}t.
\end{align}
By Proposition \ref{zeta>1} and \eqref{zeta<1}, we obtain the following 

\begin{cor}\label{cor:zeta<1} 
The second integral in \eqref{zeta1+1}:
$\int_{0}^{1}\int_{\BA^{1}}\phi(t\Fa)t^{s-\frac{1}{2}}\ud^\times\Fa\ud^{\times}t$ 
converges absolutely for $\Re(s)>\frac{n+1}{2}$ and has analytic continuation to an entire function in $s\in\BC$. Moreover, the following identity 
\[
\int_{0}^{1}\int_{\BA^{1}}\phi(t\Fa)t^{s-\frac{1}{2}}\ud^\times\Fa\ud^{\times}t
=
\int_{1}^\infty\int_{\BA^{1}}\CF_{\pi,\psi}(\phi)(t\Fa)t^{\frac{1}{2}-s}\ud^\times\Fa\ud^{\times}t
\]
holds by analytic continuation for $s\in\BC$, where the integral on the right-hand side converges absolutely for all $s\in\BC$.
\end{cor}

Set $\BA^{>1}:=\{x\in\BA^\times\ \colon\ |x|_\BA>1\}$.  By combining \eqref{zeta1+1} with \eqref{zeta<1}, we obtain that when $\Re(s)>\frac{n+1}{2}$
\begin{align}\label{zeta11}
\int_{\BA^{\times}}\phi(x)|x|_\BA^{s-\frac{1}{2}}\ud^{\times}x
=\int_{\BA^{>1}}\phi(x)|x|_\BA^{s-\frac{1}{2}}\ud^{\times}x+\int_{\BA^{>1}}\CF_{\pi,\psi}(\phi)(x)|x|_\BA^{\frac{1}{2}-s}\ud^{\times}x,
\end{align}
which holds for all $s\in\BC$ by analytic continuation. 
From the proof of Proposition \ref{zeta>1}, both integrals on the right-hand side converge absolutely when $s\in\BC$ belongs to the vertical strip 
$\frac{1}{2}-c<\Re(s)<\frac{1}{2}+c$ for any constant $c$ with $c>\max\{b_\pi,b_{\wt{\pi}}\}$. 
Hence they converge absolutely at any $s\in\BC$.

We are going to calculate the integral $\int_{\BA^{>1}}\phi(x)|x|_{\BA}^{s-\frac{1}{2}}\ud^{\times}x$ in another way. 
For $x=(x_\nu)\in\BA^\times$, we write $x=x_\infty\cdot x_f$ with $x_\infty\in \BA_\infty^\times$ and $x_f\in\BA_f^\times$. For $x\in\BA^{>1}$, 
we have that $|x|=|x_\infty|_{\BA}\cdot|x_f|_{\BA}>1$ and $|x_f|_{\BA}>|x_\infty|_{\BA}^{-1}$. 
For $\phi=\phi_\infty\otimes\phi_f\in\CS_\pi(\BA^\times)=\CS_{\pi_\infty}(\BA_\infty^\times)\otimes\CS_{\pi_f}(\BA_f^\times)$,  we write
\begin{align}\label{zeta>1-2}
   \int_{\BA^{>1}}\phi(x)|x|_{\BA}^{s-\frac{1}{2}}\ud^{\times}x
   =\int_{\BA_\infty^\times} \phi_{\infty}(x_{\infty})|x_{\infty}|_{\BA}^{s-\frac{1}{2}}\ud^{\times}x_{\infty}
   \int_{\BA_f^{\times}}^{>|x_{\infty}|^{-1}}\phi_f(x_f)|x_f|_{\BA}^{s-\frac{1}{2}}\ud^{\times}x_f, 
\end{align}
for any $s\in\BC$, where the inner integral is taken over $\{ x_f\in\BA_f^{\times}\ \colon\  |x_f|_{\BA}> |x_{\infty}|_{\BA}^{-1} \}$. By the Fubini theorem, we know 
(from the proof of Proposition \ref{prp:HKTD}) that the inner integral
\[
\int_{\BA_f^{\times}}^{>|x_{\infty}|_{\BA}^{-1}}\phi_f(x_f)|x_f|_{\BA}^{s-\frac{1}{2}}\ud^{\times}x_f
\]
converges absolutely for any $s\in\BC$ and any $x_\infty\in\BA_\infty^\times$.  

\begin{dfn}[Godement-Jacquet Kernels]\label{dfn:CK}
For any $\pi=\pi_\infty\otimes\pi_f\in\CA_\cusp(\RG_n)$, take any $\phi_f\in\CS_{\pi_f}(\BA^\times)$, the Godement-Jacquet kernels associated with $\pi$ are defined to be 
\[
H_{\pi,s}(x_\infty,\phi_f):=
    |x_{\infty}|_{\BA}^{s-\frac{1}{2}}
   \int_{\BA_f^{\times}}^{>|x_{\infty}|_{\BA}^{-1}}\phi_f(x_f)|x_f|_{\BA}^{s-\frac{1}{2}}\ud^\times x_f,
\]
for $x_\infty\in\BA_\infty^\times$ and for all $s\in\BC$.     
\end{dfn}

From \eqref{zeta>1-2}, we obtain that 
\begin{align}\label{CKpi1}
    \int_{\BA^{>1}}\phi(x)|x|_{\BA}^{s-\frac{1}{2}}\ud^{\times}x
    =\int_{\BA_\infty^\times}\phi_{\infty}(x_{\infty})H_{\pi,s}(x_{\infty},\phi_f)\ud^{\times} x_{\infty}.
\end{align}
In the spirit of \cite{Clo22}, to each $\pi\in\CA_\cusp(\RG_n)$, we define the {\bf dual kernel} of the Godement-Jacquet kernel $H_{\pi,s}(x_\infty,\phi_f)$ associated with $\pi$ to be 
\begin{align}\label{CKd}
    K_{\pi,s}(x_\infty,\phi_f):=
    |x_{\infty}|_{\BA}^{s-\frac{1}{2}}
   \int_{\BA_f^{\times}}^{>|x_{\infty}|_{\BA}^{-1}}\CF_{\pi_f,\psi_f}(\phi_f)(x_f)|x_f|_{\BA}^{s-\frac{1}{2}}\ud^\times x_f,
\end{align}
for $x_\infty\in\BA_\infty^\times$ and for all $s\in\BC$.  

With a suitable choice of the functions $\phi_f$, the kernel functions $H_{\pi,s}(x_\infty,\phi_f)$ and $K_{\pi,s}(x_\infty,\phi_f)$ may have simple expressions. We refer to 
Proposition \ref{prp:CKpi} for details. We establish the distribution property for $H_{\pi,s}(x_\infty,\phi_f)$ and $K_{\pi,s}(x_\infty,\phi_f)$.

\begin{prp}\label{prp:HKTD}
Set $\FB_\infty:=\{x_{\infty}\in\BA_{\infty}\ \colon\ |x_{\infty}|_{\BA}=0\}$ and write 
$\BA_\infty=\BA_\infty^\times\cup\FB_\infty$. 
For any $\phi_f\in\CS_{\pi_f}(\BA_f^\times)$ and for any $s\in\BC$, the Godement-Jacquet kernel function $H_{\pi,s}(x_\infty,\phi_f)$ and its dual kernel function $K_{\pi,s}(x_\infty,\phi_f)$ on $\BA_\infty^\times$ enjoy the following properties. 
\begin{enumerate}
    \item Both $H_{\pi,s}(x_\infty,\phi_f)$ and $K_{\pi,s}(x_\infty,\phi_f)$ vanish to infinity order along $\FB_\infty$.
    \item Both $H_{\pi,s}(x_\infty,\phi_f)$ and $K_{\pi,s}(x_\infty,\phi_f)$ have unique canonical extension across $\FB_\infty$ to the whole space $\BA_{\infty}$.
    \item Both $H_{\pi,s}(x_\infty,\phi_f)$ and $K_{\pi,s}(x_\infty,\phi_f)$ are tempered distributions on $\BA_\infty$.
\end{enumerate}
\end{prp}

\begin{proof}
By definition, we have that $K_{\pi,s}(x_{\infty},\phi_f)=H_{\widetilde{\pi},s}(x_{\infty},\CF_{\pi,\psi}(\phi_f))$. 
It is enough to show that Properties (1), (2), and (3) hold for the kernel function $H_{\pi,s}(x_{\infty},\phi_f)$.  
We prove (1) and (2) by using the work of S. Miller and W. Schmid in \cite{MS04-JFA} (in particular \cite[Definition 2.4, Lemma 2.8, Definition 2.6]{MS04-JFA}). Then we prove (3) 
by showing that $H_{\pi,s}(x_{\infty},\phi_f)$ is of polynomial growth as the Eucilidean norm of $x_{\infty}$ tends to $\infty$ (\cite[Theorem 25.4]{Tre67}). 

Without loss of generality, we may assume that $\phi_f=\otimes_{\nu}\phi_{\nu}\in\CS_{\pi_f}(\BA_f^\times)$ is factorizable. 
Let $T\subset|k|_f$ be a finite set such that for $\nu\notin T$, both $\psi_{\nu}$ and $\pi_{\nu}$ are unramified and $\phi_{\nu}(x)=\BL_{k_{\nu}}(x)$, the basic function in 
$\CS_{\pi_\nu}(k_\nu^\times)$. 
According to \cite[Lemma 5.3]{JL23}, there are integers $\{e_{\nu}\}_{\nu\in T}$ such that the support of $\phi_{f}$ is contained in
$\left(\prod_{\nu\in T}\left(\Fp_{\nu}^{e_{\nu}}\setminus\{0\}\right)\times\prod_{\nu\notin T}\left(\Fo_{\nu}\setminus\{0\}\right)\right)\bigcap\BA_f^{\times}$. 
According to (\ref{es-phif}), for any $c>b_{\pi}$, there is a constant $C_1$ such that
$|\phi_f(x_f)|\leq C_1|x_f|_{\BA}^{-c}$ 
for any $x_f\in\BA_f^{\times}$. Write
\[
\BA_f^{\times}=\bigsqcup_{\alpha=(\alpha_{\nu})}\left(\prod_{\nu\in|k|_f}\varpi_{\nu}^{\alpha_{\nu}}\Fo_{\nu}^{\times}\right),
\]
where $\varpi_{\nu}$ is the local uniformizer in $k_{\nu}$ and $\alpha$ runs over the algebraic direct sum $\oplus_{\nu\in|k|_f}\BZ$. Then for $x$ in the $\alpha=(\alpha_{\nu})$ component, we have $|x|_{f}=\prod_{\nu\in|k|_f}q_{\nu}^{-\alpha_{\nu}}$ and the inequality: $|x|_f\leq|x_{\infty}|_{\BA}$ is equivalent to the inequality: $\prod_{\nu\in|k|_f}q_{\nu}^{\alpha_{\nu}}<|x_{\infty}|_{\BA}$. 
We may write a fractional ideal $\Fl$ in $k$ as $\prod_{\nu}\Fp_{\nu}^{\alpha_{\nu}}$. We set 
$\Fm=\Fm_T:=\prod_{\nu\in T}\Fp_{\nu}^{e_{\nu}}$, 
which is the fractional ideal depending on the support of $\phi_f$. 

According to the normalization of our Haar measure on $\BA_f^{\times}$, we have that 
\begin{align*}
    \int_{\BA_f^{\times}}^{>|x_{\infty}|_{\BA}^{-1}}\left|\phi_f(x_f)|x_f|_{\BA}^{s-\frac{1}{2}}\right|\ud^{\times}x
    &\leq C_1\cdot\int_{\BA_f^{\times}}^{>|x_{\infty}|_{\BA}^{-1}}|x_f|_{\BA}^{-c+\Re(s)-\frac{1}{2}}\ud^{\times}x_f
    \leq C_1\cdot\sum_{\Fl\subset \Fm, \FN(\Fl)<|x_{\infty}|_{\BA}}\FN(\Fl)^{c+\frac{1}{2}-\Re(s)},
\end{align*}
where the last summation runs over all fractional ideals of $k$ that are contained in $\Fm$ with absolute norm less than or equal to $|x_{\infty}|_{\BA}$. 
Write $\Fj=\Fm^{-1}\Fl$ and obtain that 
\[
\sum_{\Fl\subset \Fm, \FN(\Fl)<|x_{\infty}|_{\BA}}\FN(\Fl)^{c+\frac{1}{2}-\Re(s)}
=\sum_{\Fj\subset \Fo, \FN(\Fj)<\frac{|x_{\infty}|_{\BA}}{\FN(\Fm)}} \FN(\Fm)^{c+\frac{1}{2}-\Re(s)}\cdot\FN(\Fj)^{c+\frac{1}{2}-\Re(s)}
\]
Let $a(n)$ be the number of ideals $\Fj\subset \Fo$ with $\FN(\Fj)=n$. According to the Wiener-Ikehara theorem (\cite[Corollary 8.8]{MV07}), there is a constant $C^{\prime}$ such that
$\sum_{n\leq x}a(n)\leq C^{\prime}x$ 
for all $x\geq 0$, and in particular $a(n)\leq C^{\prime}n$. We obtain that 
\begin{align*}
    \sum_{\Fj\subset \Fo, \FN(\Fj)<\frac{|x_{\infty}|_{\BA}}{\FN(\Fm)}}\FN(\Fm)^{c+\frac{1}{2}-\Re(s)}\cdot\FN(\Fj)^{c+\frac{1}{2}-\Re(s)}
    &= \FN(\Fm)^{c+\frac{1}{2}-\Re(s)}\sum_{n\leq\frac{|x_{\infty}|_{\BA}}{\FN(\Fm)} }a(n)n^{c+\frac{1}{2}-\Re(s)}\\
    &\leq \FN(\Fm)^{c+\frac{1}{2}-\Re(s)}C^{\prime}\sum_{n\leq\frac{|x_{\infty}|_{\BA}}{\FN(\Fm)}}n^{c+\frac{3}{2}-\Re(s)}.
\end{align*}
For a fixed $s\in\BC$, and any fixed $c>\max\{b_{\pi},\Re(s)-\frac{3}{2}\}$, we have that 
\[
\sum_{n\leq\frac{|x_{\infty}|_{\BA}}{\FN(\Fm)}}n^{c+\frac{3}{2}-\Re(s)}
\leq\int_{1}^{ \frac{|x_{\infty}|_{\BA}}{\FN(\Fm)}+1}x^{c+\frac{3}{2}-\Re(s)}\ud x
\leq \frac{   \left(  \frac{|x_{\infty}|_{\BA}}{\FN(\Fm)}+1   \right) ^{c+\frac{5}{2}-\Re(s)} }{c+\frac{5}{2}-\Re(s)}.
\]
When $|x_{\infty}|_{\BA}\geq \FN(\Fm)$, we deduce that 
\[
\frac{   \left(  \frac{|x_{\infty}|_{\BA}}{\FN(\Fm)}     +1   \right) ^{c+\frac{5}{2}-\Re(s)} }{c+\frac{5}{2}-\Re(s)}
\leq \frac{   2^{c+\frac{5}{2}-\Re(s)}\FN(\Fm)^{\Re(s)-c-\frac{5}{2}}|x_{\infty}|_{\BA}^{c+\frac{5}{2}-\Re(s)}  }{c+\frac{5}{2}-\Re(s)}.
\]
Hence we obtain that 
\begin{align}\label{estofH}
\int_{\BA_f^{\times}}^{>|x_{\infty}|_{\BA}^{-1}}\left|\phi_f(x_f)|x_f|_{\BA}^{s-\frac{1}{2}}\ud^{\times}x_f       \right|
     \leq\frac{2^{c+\frac{5}{2}-\Re(s)}C_1C^{\prime}}{(c+\frac{5}{2}-\Re(s))\FN(\Fm)^2}|x_{\infty}|_{\BA}^{c+\frac{5}{2}-\Re(s)}
     =\widetilde{C}|x_{\infty}|_{\BA}^{c+\frac{5}{2}-\Re(s)},
\end{align}
for any $c>\max\{b_{\pi},\Re(s)-\frac{3}{2}\}$ and $|x_{\infty}|_{\BA}>\FN(\Fm)$, where 
$
\widetilde{C}=\frac{2^{c+\frac{5}{2}-\Re(s)}C_1C^{\prime}}{(c+\frac{5}{2}-\Re(s))\FN(\Fm)^2}
$
is a constant depending on $k$, $\phi_f$, $s$, $c$, and is independent of $|x_{\infty}|_{\BA}$. Moreover, from the above calculation, we obtain that  
\begin{align}\label{zero}
    \int_{\BA_f^{\times}}^{>|x_{\infty}|_{\BA}^{-1}}\left|\phi_f(x_f)|x_f|_{\BA}^{s-\frac{1}{2}}\ud^{\times}x_f  
     \right|=0
\end{align}
if $|x_{\infty}|_{\BA}\leq\FN(\Fm)$. From \eqref{zero}, it is clear that 
the kernel function $H_{s,\pi}(x_{\infty},\phi_f)$ vanishes at some neighborhood for any point $x_{\infty}\in \FB_\infty$. 
By \cite[Definition 2.4, Lemma 2.8, Definition 2.6]{MS04-JFA}, $H_{s,\pi}(x_{\infty},\phi_f)$ vanishes of infinity order at $\FB_\infty$ and has a unique canonical extension across $\FB_\infty$ to the whole $\BA_{\infty}$, which we still denote by $H_{\pi,s}(x_{\infty},\phi_f)$. We establish Properties (1) and (2). 

For Property (3), because of the estimate in \eqref{estofH}, the kernel $H_{\pi,s}(x_{\infty},\phi_f)$ is of polynomial growth as the Eucilidean norm of $x_{\infty}$ tends to $\infty$. Hence $H_{\pi,s}(x_{\infty},\phi_F)$ is tempered as a distribution on $\BA_\infty$ according to \cite[Theorem 25.4]{Tre67}.
\end{proof}

\subsection{$\pi_\infty$-Fourier transform}\label{ssec-FTGJK}

From \eqref{CKd}, we obtain for $\phi=\phi_\infty\otimes\phi_f\in\CS_\pi(\BA^\times)$ that 
\begin{align}\label{CKd1}
    \int_{\BA^{>1}}\CF_{\pi,\psi}(\phi)(x)|x|^{\frac{1}{2}-s}\ud^{\times}x
    =
    \int_{\BA_\infty^\times}\CF_{\pi_\infty,\psi_\infty}(\phi_{\infty})(x_{\infty})K_{\pi,1-s}(x_\infty,\phi_f)\ud^\times x_\infty, 
\end{align}
which converges absolutely for all $s\in\BC$. 
The following is the duality relation of the Godement-Jacquet kernels $H_{\pi,s}(x_\infty,\phi_f)$ and $K_{\pi,s}(x,\phi_f)$ via the $\pi_\infty$-Fourier transform when $s\in\BC$ 
is such that $L_f(s,\pi_f)=0$, which is part of \cite[Theorem 1.1]{Clo22} for $\pi\in\CA_\cusp(\RG_n)$. 

\begin{prp}\label{prp:FTGJK}
For any $\pi\in\CA_\cusp(\RG_n)$, take $\phi=\phi_\infty\otimes\phi_f\in\CS_\pi(\BA^\times)$. Then the Godement-Jacquet kernel $H_{\pi,s}(x,\phi_f)$ associated with $\pi$ 
and its dual kernel $K_{\pi,s}(x,\phi_f)$ enjoy the following identity:
\[
H_{\pi,s}(x,\phi_f)
=-
\CF_{\pi_\infty,\psi_\infty}(K_{\pi,1-s}(\cdot,\phi_f))(x)
=
-\int_{\BA_\infty^{\times}}k_{\pi_{\infty},\psi_{\infty}}(x y)K_{\pi,1-s}(y,\phi_f)\ud^{\times}y.
\]
as distributions on $\BA_\infty^\times$ if $s$ is a zero of $L_f(s,\pi_f)$, where $k_{\pi_{\infty},\psi_{\infty}}$ is the $\pi_\infty$-kernel function as given in \eqref{kernel-ar} 
that gives the $\pi_\infty$-Fourier transform as a convolution integral operator as in \eqref{FO-k}.
\end{prp}

\begin{proof}
For $\Re(s)>\frac{n+1}{2}$, we have 
\[
\CZ(s,\phi)=\int_{\BA^{\times}}\phi(x)|x|_{\BA}^{s-\frac{1}{2}}\ud^{\times}x=\prod_{\nu\in|k|}\CZ_\nu(s,\phi_\nu).
\]
By the reformulation of the Godement-Jacquet local theory in \cite[Theorem 3.4]{JL23}, we obtain that 
\[
\int_{\BA^{\times}}\phi(x)|x|_{\BA}^{s-\frac{1}{2}}\ud^{\times}x
=
\CZ_\infty(s,\phi_\infty)\cdot L_f(s,\pi_f)\cdot\prod_{\nu\in|k|}\CZ^*_\nu(s,\phi_\nu)
\]
where at almost all finite local places $\nu$ with $\phi_\nu$ equal to the basic function $\BL_{\pi_\nu}$, we have that $\CZ^*(s,\phi_\nu)=1$, for the remaining finite local 
places $\nu$, where 
$\CZ^*_\nu(s,\phi_\nu):=L(s,\pi_\nu)^{-1}\CZ_\nu(s,\phi_\nu)$
is holomorphic in $s\in\BC$, and 
$\CZ_\infty(s,\phi_\infty):=\prod_{\nu\in|k|_\infty}\CZ_\nu(s,\phi_\nu)$ is holomorphic in $s\in\BC$ if $\phi_\infty\in\CC_c^\infty(\BA_\infty^\times)$.  
Hence we obtain that $\prod_{\nu\in|k|}\CZ^*_\nu(s,\phi_\nu)$ is a finite product of holomorphic functions. From \eqref{zeta11}, we have 
\[
Z(s,\phi)
=\int_{\BA^{>1}}\phi(x)|x|_{\BA}^{s-\frac{1}{2}}\ud^{\times}x+\int_{\BA^{>1}}\CF_{\pi,\psi}(\phi)(x)|x|_{\BA}^{\frac{1}{2}-s}\ud^{\times}x
\]
for all $s\in\BC$ by analytic continuation.
Hence we obtain that if $s\in\BC$ is such that $L_f(s,\pi_f)=0$, then we must have that 
\begin{align}\label{CK11}
    \int_{\BA^{>1}}\phi(x)|x|_{\BA}^{s-\frac{1}{2}}\ud^{\times}x=-\int_{\BA^{>1}}\CF_{\pi,\psi}(\phi)(x)|x|_{\BA}^{\frac{1}{2}-s}\ud^{\times}x.
\end{align}
Note that by Proposition \ref{zeta>1} both integrals converges absolutely for any $s\in\BC$. 
From \eqref{CKd1}, we have that 
\[
\int_{\BA^{>1}}\CF_{\pi,\psi}(\phi)(x)|x|^{\frac{1}{2}-s}\ud^{\times}x
    =
    \int_{\BA_\infty^\times}\CF_{\pi_\infty,\psi_\infty}(\phi_{\infty})(x_{\infty})K_{\pi,1-s}(x_\infty,\phi_f)\ud^\times x_\infty
\]
is absolutely convergent according to (\ref{estofH}). By \cite[Theorem 5.1]{JL22}, which is recalled in \eqref{FO-k}, there is a $\pi_\infty$-kernel function $k_{\pi_\infty,\psi_\infty}$, such that 
for any $\phi_\infty\in\CC_c^{\infty}(\BA_\infty^\times)$
\[
\CF_{\pi_\infty,\psi_\infty}(\phi_\infty)(x_\infty)=(k_{\pi_\infty,\psi_\infty}*\phi_\infty^\vee)(x_\infty)
=\int_{\BA_\infty^{\times}}k_{\pi_{\infty},\psi_{\infty}}(x_\infty y_\infty)\phi_{\infty}(y_\infty)\ud^{\times}y_\infty.  
\]
Since $\CF_{\pi_{\infty},\psi_{\infty}}(\phi_{\infty})\in\CS_{\pi_{\infty},\psi_{\infty}}(\BA^{\times}_{\infty})$,  By using the Fubini's theorem and Proposition \ref{zeta>1} again, 
we obtain that 
\begin{align*}
    \int_{\BA^{>1}}\CF_{\pi,\psi}(\phi)(x)|x|^{\frac{1}{2}-s}\ud^{\times}x
    &=\int_{\BA_\infty^\times}
    \int_{\BA_\infty^{\times}}k_{\pi_{\infty},\psi_{\infty}}(x_\infty y_\infty)\phi_{\infty}(y_\infty)\ud^{\times}y_\infty
    K_{\pi,1-s}(x_\infty,\phi_f)\ud^\times x_\infty\\
    &=\int_{\BA_\infty^\times}\phi_{\infty}(y_\infty)
    \int_{\BA_\infty^{\times}}k_{\pi_{\infty},\psi_{\infty}}(x_\infty y_\infty)K_{\pi,1-s}(x_\infty,\phi_f)\ud^{\times}x_\infty
    \ud^\times y_\infty.
\end{align*}
By definition as in \eqref{FO-k}, we write the $\pi_\infty$-Fourier transform of the dual kernel $K_{\pi,1-s}(x_\infty,\phi_f)$, viewed as a distribution on $\BA_\infty^\times$, 
to be 
\[
\CF_{\pi_\infty,\psi_\infty}(K_{\pi,1-s}(\cdot,\phi_f))(y_\infty)
=
\int_{\BA_\infty^{\times}}k_{\pi_{\infty},\psi_{\infty}}(x_\infty y_\infty)K_{\pi,1-s}(x_\infty,\phi_f)\ud^{\times}x_\infty.
\]
Hence we obtain that 
\[
\int_{\BA^{>1}}\CF_{\pi,\psi}(\phi)(x)|x|^{\frac{1}{2}-s}\ud^{\times}x
=
\int_{\BA_\infty^\times}\phi_{\infty}(y_\infty)\CF_{\pi_\infty,\psi_\infty}(K_{\pi,1-s}(\cdot,\phi_f))(y_\infty)\ud^\times y_\infty.
\]
By combining \eqref{CK11} with \eqref{CKpi1}, we obtain the following identity as distributions on $\BA_\infty^\times$
\[
\int_{\BA_\infty^\times}\phi_{\infty}(y_\infty)\CF_{\pi_\infty,\psi_\infty}(K_{\pi,1-s}(\cdot,\phi_f))(y_\infty)\ud^\times y_\infty
=-
\int_{\BA_\infty^\times}\phi_{\infty}(x_{\infty})H_{\pi,s}(x_{\infty},\phi_f)\ud^{\times} x_{\infty}
\]
for all $\phi_\infty\in\CC_c^\infty(\BA_\infty^\times)$. Therefore, as distributions on $\BA_\infty^\times$, we have that 
\[
\CF_{\pi_\infty,\psi_\infty}(K_{\pi,1-s}(\cdot,\phi_f))(x_\infty)=-H_{\pi,s}(x_{\infty},\phi_f).
\]
\end{proof}

For any $\pi=\pi_\infty\otimes\pi_f\in\CA_\cusp(\RG_n)$, we write that 
$
L_f(s,\pi_f)=\prod_{\nu<\infty}L(s,\pi_\nu)
$
when $\Re(s)$ is sufficiently positive. By \cite[Corollary 3.8]{JL23} and the theory of Mellin transforms, we obtain

\begin{prp}\label{prp:phip}
    For any $\nu\in|k|_f$, there is a function $\phi_\nu\in\CS_{\pi_\nu}(k_\nu^{\times})$ such that
\[
\int_{k_\nu^{\times}}\phi_\nu(x)|x|_\nu^{s-\frac{1}{2}}\ud ^{\times}x=L(s,\pi_\nu)
\]
holds as functions in $s\in\BC$ by meromorphic continuation.
\end{prp}

For any $\phi_{\infty}\in\CS_{\pi_{\infty}}(\BA_\infty^{\times})$, take $\phi^\star=\phi_\infty\otimes\phi_f^\star$, where $\phi_f^\star:=\otimes_\nu\phi_\nu$ with $\phi_\nu$ as given in Proposition \ref{prp:phip} and $\phi_{\nu}=\BL_{\nu}$, the basic function, for almost all $\nu$. It is clear that such a function $\phi$ belongs to the $\pi$-Schwartz space $\CS_{\pi}(\BA^{\times})$. As in \eqref{zetaG1}, the zeta integral
\[
\CZ(s,\phi^\star)=\int_{\BA^{\times}}\phi^\star(x)|x|_{\BA}^{s-\frac{1}{2}}\ud^{\times}x
\]
converges absolutely when $\Re(s)>\frac{n+1}{2}$ and can be written as 
\begin{align}\label{zetaL}
    \CZ(s,\phi^\star)
    =\CZ(s,\phi_\infty)\cdot\CZ(s,\phi^\star_f)
    =\CZ(s,\phi_\infty)\cdot L_f(s,\phi_f),
\end{align}
where 
$\CZ(s,\phi^\star_f)=\prod_{\nu\in|k|_f}\CZ(s,\phi_\nu)=\prod_{\nu\in|k|_f}L(s,\pi_\nu)$ 
when $\Re(s)>\frac{n+1}{2}$. We set 
\begin{align}\label{GJKstar}
    H_{\pi,s}(x):=H_{\pi,s}(x,\phi^\star_f)\quad {\rm and}\quad     K_{\pi,s}(x):=K_{\pi,s}(x,\phi^\star_f),
\end{align}
and call $H_{\pi,s}(x)$ the {\bf Godement-Jacquet kernel} associated with the Euler product $L_f(s,\pi_f)$, and $K_{\pi,s}(x)$ its {\bf dual kernel}. 

\begin{thm}\label{thm:H=FK}
For any $\pi\in\CA_\cusp(\RG_n)$, take $\phi^\star=\phi_\infty\otimes\phi^\star_f\in\CS_\pi(\BA^\times)$ with $\phi_f^\star:=\otimes_{\nu\in|k|_f}\phi_\nu$ where $\phi_\nu$ is as given in Proposition \ref{prp:phip}. Then the Godement-Jacquet kernel $H_{\pi,s}(x)$ associated with the Euler product $L_f(s,\pi)$ 
and its dual kernel $K_{\pi,s}(x)$ enjoy the following identity:
\begin{align}\label{CK-pi*}
H_{\pi,s}(x)
=-
\CF_{\pi_\infty,\psi_\infty}(K_{\pi,1-s})(x)
=
-\int_{\BA_\infty^{\times}}k_{\pi_{\infty},\psi_{\infty}}(x y)K_{\pi,1-s}(y)\ud^{\times}y.
\end{align}
as distributions on $\BA_\infty^\times$ if and only if $s$ is a zero of $L_f(s,\pi_f)$. 
\end{thm}

\begin{proof}
By Proposition \ref{prp:FTGJK}, we only need to consider that if \eqref{CK-pi*} holds, then $s\in\BC$ is such that $L_f(s,\pi_f)=0.$

By the choice of $\phi^\star=\phi_\infty\otimes\phi^\star_f$, we have from Proposition \ref{prp:phip} that 
\[
\int_{\BA^{\times}}\phi(x)|x|^{s-\frac{1}{2}}\ud^{\times}x
=
\CZ_\infty(s,\phi_\infty)\cdot L_f(s,\pi_f).
\]
From the proof of Proposition \ref{prp:FTGJK}, we deduce that if  \eqref{CK-pi*} holds, then we must have that 
$
\CZ_\infty(s,\phi_\infty)\cdot L_f(s,\pi_f)=0
$
for any $\phi_\infty\in\CC_c^\infty(\BA_\infty^\times)$. It is clear that one is able to choose a particular test function $\phi_\infty$ such that $\CZ_\infty(s,\phi_\infty)\neq 0$. 
Hence we get that $L_f(s,\pi_f)=0$. 
\end{proof}

\subsection{Clozel's theorem for $\pi$}\label{ssec-CThpi}

In \cite[Section 1]{Clo22}, Clozel defines the Tate kernel and its dual kernel associated with the Dirichlet series expression of the Dedekind zeta function $\zeta_k(s)$ 
of the ground number field $k$ and prove Theorem 1.1 of \cite{Clo22} by two methods, one is an approach from the Tate functional equation and the other is a more 
classical approach from analytic number theory. For a general $\pi\in\CA_\cusp(\RG_n)$, we define in Definition \ref{dfn:CK} and \eqref{CKd} the Godement-Jacquet kernels 
$H_{\pi,s}(x,\phi_f)$ and their dual kernels $K_{\pi,s}(x,\phi_f)$ via the $\pi_f$-Fourier transform $\CF_{\pi_f,\psi_f}$ associated with the global functional equation 
in the reformulation of the Godement-Jacquet theory. By using the testing functions for the local zeta integrals and the local $L$-factors $(L(s,\pi_\nu)$ at all finite local places 
$\nu\in|k|_f$ (Proposition \ref{prp:phip}), we obtain the $\pi$-version of \cite[Theorem 1.1]{Clo22} when the kernel functions are related to the $L$-function $L(s,\pi)$ with the 
Euler product expression (Theorem \ref{thm:H=FK}). In order to obtain the $\pi$-version of \cite[Theorem 1.1]{Clo22} when the kernel functions are related to the $L$-function $L(s,\pi)$ with its Dirichlet series expression, we are going to refine the structure of the testing functions in Proposition \ref{prp:phip} by using the construction in \cite{Hum21}. 

\begin{lem}\label{lem-testvec}
For $\nu\in|k|_f$, assume that $(\pi_{\nu},V_{\pi_{\nu}})\in\Pi_{k_\nu}(\RG_n)$ is generic. Then there exists a function $\phi_{\nu}\in\CS_{\pi_{\nu}}(k_{\nu}^{\times})$ such that 
    \[
    \CZ_\nu(s,\phi_\nu):=\int_{k_{\nu}^{\times}}\phi_{\nu}(x)|x|_{\nu}^{s-\frac{1}{2}}\ud^{\times}x=L(s,\pi_{\nu}),
    \]
     the support of $\phi_{\nu}$ is contained in $\Fo_{\nu}\setminus\{0\}$, and $\phi_{\nu}$ is invariant under the action of $\Fo_{\nu}^{\times}$.
\end{lem}

\begin{proof}
   If $n=1$, then $\pi_{\nu}$ is a quasi-character of $k_{\nu}^{\times}$. If $\pi_{\nu}$ is unramified, it is well-known that one takes $\phi_{\nu}(x)=|x|_{\nu}^{\frac{1}{2}}1_{\Fo_{\nu}}(x)$ with $1_{\Fo_{\nu}}$ the characteristic function of $\Fo_{\nu}$, and has the following identity
   \[
   \CZ_{\nu}(s,\phi_{\nu})=\int_{k_{\nu}^{\times}}\phi_{\nu}(x)|x|_{\nu}^{s-\frac{1}{2}}\ud^{\times} x=\frac{1}{1-\pi_{\nu}(\varpi_{\nu})q^{-s}}=L(s,\pi_{\nu}),
   \]
which holds for all $s\in\BC$ by meromorphic continuation , where $\varpi_{\nu}$ is the uniformizer of $k_{\nu}$. It is clear that in this case $\phi_{\nu}$ is supported on $k_{\nu}^{\times}\cap\Fo_{\nu}=\Fo_{\nu}\setminus\{0\}$ and is invariant under $\Fo_{\nu}^{\times}$.
If $\pi_{\nu}$ is ramified, then we know that $L(s,\pi_{\nu})=1$. We can take a function $\phi_\nu$ such that $\phi_\nu(x)=1$ if $x\in\Fo_\nu$ and $\phi_\nu(x)=0$ otherwise. 
Then according to our normalization of the Haar measure, we obtain by an easy computation that 
\[
   \CZ_{\nu}(s,\phi_{\nu})=\int_{k_{\nu}^{\times}}\phi_{\nu}(x)|x|_{\nu}^{s-\frac{1}{2}}\ud^{\times} x=1=L(s,\pi_{\nu}).
   \]
It is clear that in this case, $\phi_{\nu}$ is supported in $\Fo_{\nu}\setminus\{0\}$ and is invariant under the action of $\Fo_{\nu}^{\times}$.

In the following, we assume that $n\geq 2$. 
    For each non-negative integer $m$, we define the congruence subgroup $K_0(\Fp_{\nu}^m)$ as in \cite{Hum21} to be
    \[
    K_0(\Fp_{\nu}^m):=\{x=(x_{ij})\in \RG_n(\Fo_{\nu})\ \colon\  x_{n,1},\cdots,x_{n,n-1}\in\Fp_{\nu}^m  \}.
    \]
    According to the classification of irreducible generic representations and \cite[Theorem 5]{JPSS81}, there is a minimal positive integer $c(\pi_{\nu})$ for which the vector space
    \[
    V_{\pi_{\nu}}^{K_0(\Fp_{\nu}^{c(\pi_{\nu})})}:=\{v\in V_{\pi_{\nu}}\ \colon\  \pi_{\nu}(x)v=\omega_{\pi_{\nu}}(x_{n,n})v,\  \forall x\in K_0(\Fp_{\nu}^{c(\pi_{\nu})})  \}
    \]
    is non-trivial and in fact of dimension one. Choose $v^{\circ}\in V_{\pi_{\nu}}^{K_0(\Fp_{\nu}^{c(\pi_{\nu})})}$ and $\widetilde{v^{\circ}}\in V_{\widetilde{\pi_{\nu}}}^{K_0(\Fp_{\nu}^{c(\wt{\pi_{\nu}})})}$, respectively,  such that the matrix coefficient 
    $
    \varphi_{\pi_{\nu}}(g):=\langle \pi_{\nu}(g)v^{\circ},\widetilde{v^{\circ}} \rangle
    $
    has value $1$ at $\RI_n$. Since $n\geq 2$, we may take a Schwartz-Bruhat function $f_\nu\in \CS(\RM_{n}(k_{\nu}))$ of the form:
    \begin{align*}
        f_\nu(x)=\begin{cases}
          \frac{\omega_{\pi_{\nu}}^{-1}(x_{n,n})}{\mathrm{vol}(K_0(\Fp_{\nu}^{c(\pi_{\nu})}))}& \mathrm{if}\;
           x\in\RM_{n}(\Fo_{\nu}) \;\mathrm{with}\;x_{n,1},\cdots,x_{n,n-1}\in\Fp_{\nu}^{c(\pi_{\nu})}\;\mathrm{and}\;x_{n,n}\in\Fo_{\nu}^{\times},   \\
         0 & \mathrm{otherwise}.
                \end{cases}
    \end{align*}
Then by \cite[Theorem 1.2]{Hum21}, when $\Re s$ is sufficiently positive, one has that 
\[
\int_{\RG_n(k_{\nu})}f_\nu(g)\varphi_{\pi_{\nu}}(g)|\det g|^{s+\frac{n-1}{2}}\ud g=L(s,\pi_{\nu}).
\]
According to \cite[Propositon 3.2, Theorem 3.4]{JL23}, the fiber integration as defined in \eqref{fibration} yields that 
\begin{align}\label{fi1}
    \phi_{\pi_{\nu}}(x)=|x|_{\nu}^{\frac{n}{2}}\int_{\RG_n(F)_x}f_\nu(g)\varphi_{\pi_{\nu}}(g)\ud_xg, 
\end{align}
where $\RG_n(k_{\nu})_x$ is the fiber at $x$ of the determinant map as in \eqref{fiber}, 
is well defined and when $\Re(s)$ is sufficiently positive, we have that
\begin{align}\label{testf}
    \int_{k_{\nu}^{\times}}\phi_{\pi_{\nu}}(x)|x|_{\nu}^{s-\frac{1}{2}}\ud^{\times}x=L(s,\pi_{\nu}).
\end{align}

It remains to verify the invariance property for this function $\phi_\nu$. If $x\notin\Fo_{\nu}$, we must have that 
$
\RG_n(k_{\nu})_x\cap\RM_{n}(\Fo_{\nu})=\emptyset.
$
Thus for any $g\in\RG_n(k_{\nu})_x$ with $x\notin\Fo_{\nu}$, we must have that $f_\nu(g)=0$. By the fiber integration in \eqref{fi1}, we have that $\phi_{\pi_{\nu}}(x)=0$ when 
$x\notin\Fo_{\nu}$. Moreover, for any $u\in\Fo_{\nu}^{\times}$, we have
\[
\phi_{\pi_{\nu}}(xu)
=|xu|_{\nu}^{\frac{n}{2}}\int_{\RG_n(k_\nu)_{xu}}f_\nu(g)\varphi_{\pi_{\nu}}(g)\ud_xg
=|x|_{\nu}^{\frac{n}{2}}\int_{\RG_n(k_\nu)_x}f_\nu(hu^*)\varphi_{\pi_{\nu}}(hu^*)\ud_xh,
\]
where 
$
u^*:=\mathrm{diag}(u,1,1,\cdots,1).
$
Since $u^*\in K_0(\Fp_{\nu}^{c(\pi_{\nu})})$, one can see at once that $f_\nu(hu^*)=f_\nu(h)$ and $\varphi_{\pi_{\nu}}(hu^*)=\varphi_{\pi_{\nu}}(h)$. Therefore we obtain that  $\phi_{\pi_{\nu}}(xu)=\phi_{\pi_{\nu}}(x)$ for any $u\in\Fo_{\nu}^{\times}$.
\end{proof}

Let $\phi_f^{\circ}=\otimes_{\nu}\phi_{\nu}$ be the particularly chosen function such that for ramified places we take $\phi_{\nu}$ as given in Lemma \ref{lem-testvec} since each local component $\pi_{\nu}$ of $\pi$ is irreducible and generic when $n\geq 2$. Denote by 
\begin{align}\label{GJKcirc}
    \CH_{\pi,s}(x):=H_{\pi,s}(x,\phi_f^\circ)
\end{align}
the Godement-Jacquet kernel as in \eqref{GJKstar}, with $\phi_f^\star$ replaced by the particularly chosen $\phi_f^\circ$.

\begin{prp}\label{prp:CKpi}
    For any $\pi\in\CA_\cusp(\RG_n)$, the Godement-Jacquet kernel $\CH_{\pi,s}(x)$ as defined in \eqref{GJKcirc} enjoys the following expression: 
    \[
     H_{\pi,s}(x)=|x|^{s-\frac{1}{2}}\sum_{n\leq |x|}a_n n^{-s}, 
    \]
    as a function in $x\in\BA_\infty^\times$ for all $s\in\BC$.
\end{prp}

\begin{proof}
Write
\begin{align}\label{adeledec}
    \BA^{\times}_f=\bigsqcup_{\alpha=(\alpha_{\nu})}\left( \prod_{\nu}\varpi_{\nu}^{\alpha_{\nu}}\Fo_{\nu}^{\times}    \right),
\end{align}
where $\varpi_{\nu}$ is the local uniformizer in $k_{\nu}$ and $\alpha$ runs over the algebraic direct sum $\oplus_{\nu\in|k|_f}\BZ$. Consider the integral
\begin{align}\label{H-int}
\int_{\BA_f^{\times}}^{ \geq| x_{\infty}|_{\BA}^{-1}}\phi^{\circ}_f(x_f)|x_f|_{\BA}^{s-\frac{1}{2}}\ud^{\times}x_f,
\end{align}
where $\phi_f^{\circ}=\otimes_{\nu}\phi_{\nu}$ is as given in \eqref{GJKcirc}. 
We know $\phi_{\nu}$ is supported on the the ring $\Fo_\nu$ of $\nu$-integers according to Lemma~\ref{lem-testvec} and \cite[Lemma 5.3]{JL23}. It follows that we may assume $\alpha_{\nu}\geq 0$ for all $\nu<\infty$. 
If $x_f$ belongs to the $\alpha=(\alpha_{\nu})$-component of (\ref{adeledec}), then $|x_f|=\prod_{\nu}q_{\nu}^{-\alpha_{\nu}}$, where $q_{\nu}$ is the cardinality of 
the residue field of $k_{\nu}$. The range $|x_f|_{\BA}\geq |x_{\infty}|_{\BA}^{-1}$ of the integral in \eqref{H-int} is equivalent to the condition that 
$\prod_{\nu}q_{\nu}^{\alpha_{\nu}}\leq |x_{\infty}|_{\BA}$. Since the integrand in \eqref{H-int} is invariant under $\prod_{\nu}\Fo_{\nu}^{\times}$, we obtain that 
\begin{align}\label{valofphif}
\phi^{\circ}_f(x_f)|x_f|_{\BA}^{s-\frac{1}{2}}
=
\phi^{\circ}_f((\varpi_{\nu}^{\alpha_{\nu}}))(\prod_{\nu}q_{\nu}^{\alpha_{\nu}})^{\frac{1}{2}}(\prod_{\nu}q_{\nu}^{-\alpha_{\nu}})^s
\end{align}
with $(\varpi_{\nu}^{\alpha_{\nu}})\in\BA_f^\times$. We may write any fractional ideal $\Fl$ in $k$, in a unique way,  as $\Fl=\prod_{\nu}\Fp_{\nu}^{e_{\nu}}$ with 
$(e_\nu)\in \oplus_{\nu\in|k|_f}\BZ$, and regard the function $\phi_f^\circ$ as 
\[
\phi_f^{\circ}\ \colon\ \Fl=\prod_{\nu}\Fp_{\nu}^{e_{\nu}} \mapsto \phi_f^{\circ}((\varpi_{\nu}^{e_{\nu}}))=\prod_{\nu\in|k|_f}\phi_{\nu}(\varpi_{\nu}^{e_{\nu}})
\]
Then the function $\phi_f^\circ$ is supported on the set of integral ideals and (\ref{valofphif}) can be written as 
\[
\phi^{\circ}_f(x_f)|x_f|_{\BA}^{s-\frac{1}{2}}
=
\phi_f^\circ(\Fl)\cdot \FN(\Fl)^{\frac{1}{2}}\cdot \FN(\Fl)^{-s},
\]
for any fractional ideal $\Fl=\prod_{\nu\in|k|_f}\Fp_{\nu}^{\alpha_{\nu}}$. 
According to the normalization of the Haar measure, the integral \eqref{H-int} is equal to 
\[
\sum_{\FN(\Fl) \leq |x_{\infty}|_{\BA}}\phi_f^\circ(\Fl)\cdot \FN(\Fl)^{\frac{1}{2}}\cdot \FN(\Fl)^{-s}
=
\sum_{n\leq |x_{\infty}|_{\BA}}\left(\sum_{\FN(\Fl)=n}\phi^{\circ}_f(\Fl)\right)n^{\frac{1}{2}}n^{-s}.
\]
where the summation runs over all the integral ideals $\Fl$ of $k$.

On the other hand, for the particularly given Schwartz function $\phi^{\circ}_f\in\CS_{\pi_f}(\BA_f^\times)$, we have that 
\begin{align*}
L_f(s,\pi_f)
&=\sum_{n=1}^{\infty}a_nn^{-s}=\int_{\BA^{\times}_f}\phi^{\circ}_f(x_f)|x_f|_{\BA}^{s-\frac{1}{2}}\ud^{\times}x_f
=\lim_{|x_{\infty}|\rightarrow \infty}\int_{\BA_f^{\times}}^{ \geq| x_{\infty}|_{\BA}^{-1}}\phi^{\circ}_f(x_f)|x_f|_{\BA}^{s-\frac{1}{2}}\ud^{\times}x_f\\
&=\lim_{|x_{\infty}|_{\BA}\rightarrow \infty}\sum_{\FN(\Fl) \leq |x_{\infty}|_{\BA}}\phi_f^\circ(\Fl)\cdot \FN(\Fl)^{\frac{1}{2}}\cdot \FN(\Fl)^{-s}
=\sum_{n=1}^{\infty}\left(\sum_{\FN(\Fl)=n}\phi^{\circ}_f(\Fl)\right)n^{\frac{1}{2}}n^{-s}
\end{align*}
for $\Re(s)$ is sufficiently positive. 
By using the uniqueness of the coefficients of the Dirichlet series (see \cite[Theorem 1.6]{MV07}), we obtain that 
$a_n=
\sum_{\FN(\Fl)=n}\left(\phi^{\circ}_f(\Fl)\right)n^{\frac{1}{2}}$, 
and hence
\begin{align*}
\int_{\BA_f^{\times}}^{ \geq| x_{\infty}|_{\BA}^{-1}}\phi^{\circ}_f(x_f)|x_f|_{\BA}^{s-\frac{1}{2}}\ud^{\times}x_f
=\sum_{n\leq |x_{\infty}|_{\BA}}a_n n^{-s}.
\end{align*}
\end{proof}

In order to define the dual kernel, we consider the local functional equation in the reformulation of the local Godement-Jacquet theory in \cite[Theorem 3.10]{JL23}:
\begin{align}
    \CZ(1-s,\CF_{\pi_{\nu},\psi_{\nu}}(\phi_{\nu}))=\gamma(s,\pi_{\nu},\psi_{\nu})\cdot\CZ(s,\phi_{\nu}). 
\end{align}
By Proposition \ref{prp:phip} and \cite[Theorem 3.4]{JL23}, we have that 
\[
\gamma(s,\pi_{\nu},\psi_{\nu})\cdot L(s,\pi_{\nu})=\epsilon(s,\pi_{\nu},\psi_{\nu})\cdot L(1-s,\wt{\pi}_{\nu}).
\]
Hence for $\Re(s)$ sufficiently negative, we obtain that 
\begin{align}\label{zetaCF}
    \CZ_f(1-s,\CF_{\pi_f,\psi_f}(\phi_f))
    &=
    \int_{\BA^{\times}_f}\CF_{\pi_f,\psi_f}(\phi_f)(x_f)|x_f|_{\BA}^{\frac{1}{2}-s}\ud^{\times} x_f
    =\left(\prod_{\nu<\infty}\epsilon(s,\pi_{\nu},\psi_{\nu})\right)\cdot L_f(1-s,\wt{\pi}_f).
\end{align}
If we write 
\begin{align}\label{dual series}
\left(\prod_{\nu<\infty}\epsilon(1-s,\pi_{\nu},\psi_{\nu})\right)L(s,\widetilde{\pi})=\sum_{n=1}^{\infty}a_n^*n^{-s},
\end{align}
with $a_n^*\in\BC$, then following the argument as in the proof of Proposition \ref{prp:CKpi}, we can obtain a Dirichlet series expression for  
the dual kernel of the Godement-Jacquet kernel $H_{\pi,s}(x,\phi^{\circ}_f)$. 

\begin{prp}\label{prp:CKDpi}
    With $\phi_f^{\circ}$ as in Proposition \ref{prp:CKpi}, the dual kernel of the Godement-Jacquet kernel $\CH_{\pi,s}(x)$ as in \eqref{GJKcirc} is given by 
    \[
    \CK_{\pi,s}(x):=K_{\pi,s}(x,\phi_f^{\circ})=|x|^{s-\frac{1}{2}}\sum_{n\leq |x|}a_n^*n^{-s},
    \]
    where $\{a_n^*\}$ is defined via \eqref{dual series}.
\end{prp}

\begin{proof}

We first claim that each local component of $\CF_{\pi_f,\psi_f}(\phi_f^{\circ})$ is invariant under $\Fo_{\nu}^{\times}$. In fact, if $n=1$, the claim is clear because the classical Fourier transform of an $\Fo_{\nu}^{\times}$-invariant functions is still $\Fo_{\nu}^{\times}$-invariant by changing variables. If $n\geq 2$, at ramified places, since $\phi_{\nu}$ is as given by the fiber integration of $f_{\nu}$ and $\varphi_{\nu}$ as in Lemma \ref{lem-testvec}, we know from \cite[Equation (3.17)]{JL23} that
\[
\CF_{\pi_{\nu},\psi_{\nu}}(\phi_\nu)(x)=|x|_{\nu}^{\frac{n}{2}}\int_{\GL_n(k_{\nu})_x}\CF_{\psi_{\nu}}(f_{\nu})(g)\varphi_{\pi_{\nu}}(g^{-1})\ud g,
\]
where $\CF_{\psi_{\nu}}$ is the classical Fourier transform given by (\ref{eq:FTMAT}). If 
$
u^*:=\mathrm{diag}(u,1,\cdots,1)
$
for any $u\in\Fo_{\nu}^{\times}$, then we already know $\varphi_{\pi_{\nu}}\left((gu^*)^{-1}\right)=\varphi_{\pi_{\nu}}(g^{-1})$. Since
\[
\CF_{\psi_{\nu}}(f_{\nu})(gu^*)=\int_{\mathrm{M}_n(k_{\nu})} \psi_{\nu}(\tr(xu^*y))f_{\nu}(y)\ud^+y=\int_{\mathrm{M}_n(k_{\nu})} \psi_{\nu}(\tr(xy))f_{\nu}((u^{*})^{-1}y)\ud^+y,
\]
and by the definition of $f_{\nu}$, we see that $f_{\nu}((u^{*})^{-1}y)=f_{\nu}(y)$, we know $\CF_{\pi_{\nu},\psi_{\nu}}(\phi_{\nu})$ is invariant under $\Fo_{\nu}^{\times}$. At the remaining unramified places where $\phi_{\nu}=\BL_{\pi_{\nu}}$, we know $\CF_{\pi_{\nu},\psi_{\nu}}(\BL_{\pi_{\nu}})=\BL_{\widetilde{\pi_{\nu}}}$ and by \cite[Lemma 5.3]{JL23} we know $\CF_{\pi_{\nu},\psi_{\nu}}(\phi_{\nu})$ is invariant under $\Fo_{\nu}^{\times}$. Let $S_f$ be as in Proposition~\ref{zeta>1} for $\phi_f^{\circ}$. Then there are integers $a_1,\cdots,a_{\kappa}$ such that the support of $\CF_{\pi_f,\psi_f}$ is contained in
\begin{align}\label{anu}
\left(\prod_{\nu\in S_f}(\Fp_{\nu}^{a_{\nu}}\setminus\{0\})\times\prod_{\nu\notin S_f}(\Fo_{\nu}\setminus\{0\})\right)\cap\BA_f^{\times}.
\end{align}
Write
$\BA_f^{\times}=\bigsqcup_{\alpha=(\alpha_{\nu})}\left(\prod_{\nu\in|k|_f}\varpi_{\nu}^{\alpha_{\nu}}\Fo_{\nu}^{\times}\right)$. 
It is clear that $\CF_{\pi_f,\psi_f}(\phi_f^{\circ})$ is constant on each $\alpha$-component and supported on the $\alpha=(\alpha_{\nu})_{\nu}$-component with $\alpha_{\nu}\geq a_{\nu}$ for $\nu\in S_f$ and $\alpha_{\nu}\geq0$ for $\nu\notin S_f$. 
We may write any fraction ideal $\Fl$ in $k$ in a unique way as $\Fl=\prod_{\nu}\Fp_{\nu}^{e_{\nu}}$ and regard the function $\CF_{\pi_f,\psi_f}(\phi_f^{\circ})$ as a function on the set of fractional ideals sending $\Fl$ to
$
\prod_{\nu\in|k|_f}\CF_{\pi_{\nu},\psi_{\nu}}(\phi_{\nu})(\varpi_{\nu}^{e_{\nu}}).
$
Then we obtain that 
\[
\CF_{\pi_f,\psi_f}(\phi^{\circ}_f)(x_f)|x_f|_{\BA}^{s-\frac{1}{2}}
=
\phi_f^\circ(\Fl)\cdot \FN(\Fl)^{\frac{1}{2}}\cdot \FN(\Fl)^{-s}
\]
for $x$ in the $\alpha=(\alpha_{\nu})_{\nu}$-component, where $\Fl=\prod_{\nu\in|k|_f}\Fp_{\nu}^{\alpha_{\nu}}$. Write $\Fn=\prod_{\nu\in|k|_f}\Fp^{a_{\nu}}$, where for $\nu\in S_f$ $a_{\nu}$'s are defined from (\ref{anu}) and for $\nu\notin S_f$ we define $a_{\nu}=0$. Then by the same argument, we obtain that 
\[
a_n^*=\sum_{\Fl\subset\Fn,\CN(\Fl)=n}\CF_{\pi_f,\psi_f}(\phi_f^{\circ})(\Fl)\FN(\Fl)^{\frac{1}{2}}\FN(\Fl)^{-s},
\]
where the summation runs over all fractional ideal $\Fl$ of $k$ that are contained in $\Fn$ with norm $n$
and
$
\CK_{\pi,s}(x)=|x|^{s-\frac{1}{2}}\sum_{n\leq |x|}a_n^*n^{-s}. 
$
\end{proof}

Therefore we obtain a $\pi$-version of \cite[Theorem 1.1]{Clo22} when the kernel functions $\CH_{\pi,s}$ and $\CK_{\pi,s}$ are given in terms of  the $L$-function $L_f(s,\pi_f)$ with its Dirichlet series expression.

\begin{thm}\label{thm:CTh-pi}
For any $\pi\in\CA_\cusp(\RG_n)$, if the Godement-Jacquet kernel $\CH_{\pi,s}$ and its dual $\CK_{\pi,s}$ are defined as in Propositions \ref{prp:CKpi} and \ref{prp:CKDpi}, respectively, 
then 
\[
\CH_{\pi,s}(x)
=-
\CF_{\pi_\infty,\psi_\infty}(\CK_{\pi,1-s})(x)
=
-\int_{\BA_\infty^{\times}}k_{\pi_{\infty},\psi_{\infty}}(x y)\CK_{\pi,1-s}(y)\ud^{\times}y
\]
as distributions on $\BA_\infty^\times$ if and only if $s$ is a zero of $L_f(s,\pi_f)$. Any unexplained notation is the same as in Theorem \ref{thm:H=FK}.
\end{thm}

\bibliographystyle{alpha}
	\bibliography{references}

\end{document}